\documentclass[english,11pt,oneside]{amsart}

\usepackage[colorlinks,linkcolor=blue,citecolor=red]{hyperref}
\usepackage[english]{babel}
\usepackage{amssymb}
\usepackage{esint}
\usepackage{amscd}
\usepackage{mathrsfs}
\usepackage{tikz}
\usetikzlibrary{patterns}
\usepackage{subcaption}
\usetikzlibrary{fadings}
\usetikzlibrary{decorations}
\usepgflibrary{decorations.pathmorphing}
\tikzfading[name=fade out, inner color=transparent!0,
  outer color=transparent!100]
\usepackage{pgfplots}
\pgfplotsset{compat=1.17} 

\usepackage{geometry}
\usepackage{graphicx}

\theoremstyle{plain}
\newtheorem{theorem}{Theorem}[section]
\newtheorem{lemma}[theorem]{Lemma} 
\newtheorem{definition}[theorem]{Definition}
\newtheorem{proposition}[theorem]{Proposition}
\newtheorem{corollary}[theorem]{Corollary}
\newtheorem{notation}[theorem]{Notation}
\newtheorem{remark}[theorem]{Remark}

\newcommand{\figurezero}{
\begin{tikzpicture}[samples=100]
\tikzstyle{ann} = [fill=white,font=\footnotesize,inner sep=1pt]
\draw [dashed,black] (-2,0) -- (-2,-4.5) ;
\draw [black] (-2,0) -- (2,3.5) ;
\draw [black] (2,3.5) -- (2,-1) ;
\draw [color=black!70!white,line width=0.5pt,->] (2,3.5) -- (2,4.3) ;
\node at (2.2,4.2) [right] {$x_3$} ;
\draw [black] (2,3.5) -- (8,3.5) ;
\draw [dashed,black] (8,3.5)  -- (8,-0.5) ;
\draw [color=black!70!white,line width=0.5pt,->] (8,-0.5)  -- (8.7,-0.5) node [below] {$x_2$} ;
\draw [black] (-2,-4.5) -- (4,-4.5) ;
 \shadedraw [top color=blue!55!white, bottom color=blue!55!white] 
 (-2,-2) -- (-2,-4.5) -- (4,-4.5) -- (4,-2) ;
  \shadedraw [right color=blue!55!white, left color=blue!55!white, color=black] 
 (-2,-2) -- (-2,-4.5) -- (4,-4.5) -- (4,-2) ;
\filldraw  [color=blue!55!white] 
(4,-4.5) -- (4,-3) -- 
 plot [domain=-2:2,shift={(6,-1)},rotate=0.3] ({\x},{\x+exp(-\x*\x)}) -- (8,-0.5) ;
\shadedraw [top color=red!35!white, bottom color=blue!45!white, color=black] plot [domain=-2:2,rotate=0.3] ({\x},{\x+exp(-\x*\x)}) 
-- plot [domain=-2:4,shift={(4,1.55)}] ({\x},{(0.5*cos((0.524*(\x+2) r)))}) -- 
 plot [domain=2:-2,shift={(6,-1)},rotate=0.3] ({\x},{\x+exp(-\x*\x)}) -- plot [domain=4:-2,shift={(0,-2.5)}] ({\x},{(0.5*cos((0.524*(\x+2) r)))}) ;
\draw[color=black!20!white,domain=2:-2,shift={(1,-0.06)},rotate=0.3] plot  ({\x},{\x+exp(-\x*\x)}) ;
\draw[color=black!16!white,domain=2:-2,shift={(2,-0.25)},rotate=0.3] plot  ({\x},{\x+exp(-\x*\x)}) ;
\draw[color=black!14!white,domain=2:-2,shift={(3,-0.5)},rotate=0.3] plot  ({\x},{\x+exp(-\x*\x)}) ;
\draw[color=black!12!white,domain=2:-2,shift={(4,-0.75)},rotate=0.3] plot  ({\x},{\x+exp(-\x*\x)}) ;
\draw[color=black!10!white,domain=2:-2,shift={(5,-0.93)},rotate=0.3] plot  ({\x},{\x+exp(-\x*\x)});
\draw[color=black!20!white] (-1,-2.06) -- (-1,-4.5);
\draw[color=black!20!white] (0,-2.25) -- (0,-4.5);
\draw[color=black!20!white] (1,-2.5) -- (1,-4.5);
\draw[color=black!20!white] (2,-2.75) -- (2,-4.5);
\draw[color=black!20!white] (3,-2.93) -- (3,-4.5);
\draw [color=blue!60!white] (-2,-4.5) -- (4,-4.5);
\node at (0,-4) {$~$};
\draw [dashed,black] (4,0) -- (4,-4.5) ;
\draw [black] (4,0) -- (8,3.5) ;
\draw [black] (4,0) -- (-2,0) ;
\draw [black] (4,-4.5) -- (8,-0.5) ;
\draw [color=black!70!white,line width=0.5pt,->] (-2,-4.5) -- (-2.5,-5) ;
\node at (-2.5,-5.2) [color=black!70!white,right] {$x_1$};
\draw[style=thick,color=black,domain=4:-2,shift={(0,-2-0.5)}] plot  ({\x},{(0.5*cos((0.524*(\x+2) r)))}) ;
\node at (0.5,-2.8) [below] {$\mathcal{O}(t)$};
\draw[style=thick,color=black,domain=2:-2.018,shift={(6,-1)},rotate=0.3] plot  ({\x},{\x+exp(-\x*\x)}) ;
\node at (4.5,2.5) [below] {$\Sigma(t)=\partial\mathcal{O}(t)$};
\draw [->,color=black!60!white,line width=1.5pt] (1.5,1) -- (1.6,2.3) node [above] {$n$};
\end{tikzpicture}}

\newcounter{randymark}
\pgfdeclaredecoration{mark random y steps}{start}
{%
  \state{start}[width=+0pt,next state=step,%
  persistent precomputation={\pgfdecoratepathhascornerstrue%
  \setcounter{randymark}{0}}]{
  \stepcounter{randymark}
  \pgfcoordinate{randymark\arabic{randymark}}{\pgfpoint{0pt}{0pt}}
  }%
  \state{step}[auto end on length=1.5\pgfdecorationsegmentlength,
               auto corner on length=1.5\pgfdecorationsegmentlength,
               width=+\pgfdecorationsegmentlength]
  { \stepcounter{randymark}
    \pgfcoordinate{randymark\arabic{randymark}}{\pgfpoint{\pgfdecorationsegmentlength}{rand*\pgfdecorationsegmentamplitude}}
  }%
  \state{final}
  {
    \stepcounter{randymark}
    \pgfcoordinate{randymark\arabic{randymark}}{\pgfpointdecoratedpathlast}}%
}%

\DeclareMathOperator{\cn}{div}

\DeclareMathOperator{\supp}{supp}
\DeclareMathOperator{\diff}{d}
\DeclareMathOperator*{\osc}{osc} 
\def\loc{\textrm{loc}}
\def\eps{\varepsilon}
\def\sep{:}
\def\dx{\, dx}
\def\la{\left\lvert}
\def\lA{\left\lVert}
\def\ra{\right\rvert}
\def\rA{\right\rVert}

\def\ba{\begin{align}}
\def\be{\begin{equation}}
\def\ea{\end{align}}
\def\ee{\end{equation}}
\def\e{\eqref}

\def\dt{\diff \! t}
\def\dtau{\diff \! \tau}
\def\ds{\diff \! s}
\def\dx{\diff \! x}
\def\dy{\diff \! y}
\def\fract{\frac{\diff}{\dt}}
\def\dmH{\diff \! \mathcal{H}^{N-1}}

\def\defn{\mathrel{:=}}
\def\eps{\varepsilon}
\def\la{\left\vert}
\def\lA{\left\Vert}

\def\le{\leq}
\def\les{\lesssim}
\def\mez{\frac{1}{2}}
\def\ra{\right\vert}
\def\rA{\right\Vert}

\def\xN{\mathbb{N}}
\def\xR{\mathbb{R}}
\def\xS{\mathbb{S}}

\begin{document}

\title{The Hele-Shaw semi-flow}
\author{Thomas Alazard 
   and Herbert Koch 
}

\thanks{
Part of the research was performed during the program ``Mathematical Problems in Fluid Dynamics, Part~2,'' held during the summer of 2023 by the Simons Laufer Mathematical Sciences Institute (SLMath), which is supported by the National Science Foundation (Grant No.~DMS-1928930).
H.K. was supported by the DFG through the Hausdorff Center for Mathematics under Germany's Excellence Strategy - GZ 2047/1, Projekt-ID 390685813 Hausdorff Center for Mathematics and  Project ID 211504053 - SFB 1060. 
The authors acknowledge Hongjie Dong, Francisco Gancedo and Huy Q. Nguyen for their comments which helped improve the presentation of Section 6.
}

\begin{abstract}
In this article, we return to one of the simplest free boundary evolution problems and provide a new study of the case of an unbounded free surface with current insights leading to new results. 
Employing an elliptic formulation of the Hele-Shaw equation, 
we establish the well-posedness of the Cauchy problem in a comprehensive context. Notably, we prove that 
the Cauchy problem is well-posed in a strong sense and in a general setting.

Our main result is the construction of an abstract 
semi-flow for the Hele-Shaw problem within general fluid domains (enabling, for instance, changes 
in the topology of the fluid domain) and which satisfies several properties:  
We provide simple comparison arguments, establish a new stability estimate and derive several consequences, including monotonicity and continuity results for the 
solutions, along with many Lyapunov functionals. Applying Caffarelli's regularity results, we establish an eventual analytic regularity result for  arbitrary initial data. We also study numerous qualitative properties, including global regularity for initial data in sub-critical Sobolev spaces, well-posedness in a strong sense for initial data 
with barely a modulus of continuity, as well as waiting-time phenomena for Lipschitz solutions, in any dimension.

Our approach is grounded in insightful observations by Baiocchi and Duvaut, which consists in framing the Hele-Shaw equation as an obstacle problem. We extend this analysis to a noncompact free boundary, 
introducing intriguing new problems. Examining both solutions and subsolutions, we leverage elliptic regularity results such as the co-area formula, Caccioppoli and Moser estimates, along with the boundary Harnack inequality. We also prove a new stability estimate for solutions to the obstacle problem.

As side results, we establish a uniqueness result for positive harmonic functions defined on the subgraph of Lipschitz functions that vanish on the boundary. Additionally, we present a general result concerning the Rayleigh-Taylor sign condition in $C^{1,\alpha}$ domains.
\end{abstract}

\maketitle
\section{Introduction}

The objective of this article is to investigate the Cauchy problem concerning the evolution of a free surface, 
referred to as $\Sigma(t)$, which serves as the boundary of a fluid-occupied domain $\mathcal{O}(t)$ in $\mathbb{R}^N$, $N\ge2$. 
This free surface $\Sigma(t)$ is subject to movement driven by the fluid's velocity. Specifically, 
we assume that the normal velocity of $\Sigma$, denoted as $V_n$, is determined as follows:
$$
V_n=v\cdot n,
$$
where $v\colon \mathcal{O}\to \xR^N$ represents 
the fluid velocity, and $n$ is the  outer unit normal vector to $\Sigma$.

This formulation is quite general and encompasses many problems in fluid dynamics. In this article, 
our focus is on the Hele-Shaw equation, a fundamental model for analyzing incompressible fluids. 
This equation governs the behavior of the free surface $\Sigma$ when the fluid velocity $v$ satisfies the 
conditions of being divergence-free and obeying to Darcy's law:
\[
v=-\nabla (P+gx_N)\quad\text{and}\quad\nabla\cdot v =0\quad\text{in }\mathcal{O}, \quad P\arrowvert_{\Sigma}=0.
\]
In these equations, $P$ represents pressure, $g$ stands for the acceleration due to gravity, and $x_N$ denotes the vertical coordinate aligned with the direction of gravity.

Our results are of different nature, we will : $i)$ establish a connection to an obstacle problem, $ii)$ 
examine the regularity of solutions to this elliptic problem and $iii)$ draw various implications for the Cauchy problem associated with the Hele-Shaw equation. 
Notably, we will define a semi-flow for the Hele-Shaw equation that applies to very general fluid domains.

In this introduction, for the reader's convenience, we will:
clarify the structure of the paper, summarize these results, and
enhance this summary with illustrations to explain 
the range of generality permitted by our approach. 
To maintain the flow of the argument, we will 
postpone several discussions of related works 
to Section~\ref{I:background} as well as to subsequent chapters.

\subsection{An obstacle problem formulation}\label{S:I1.1}
We begin by assuming that the free surface is represented as the graph of a time-dependent function 
$h=h(t)$. Then $\Sigma(t)$ and the fluid domain $\mathcal{O}(t)$ are of the form:
\begin{align*}
\Sigma(t)&= \{(x',x_N) \in \mathbb{R}^{N-1}\times \mathbb{R}\,;\, x_N = h(t,x')\},\\
\mathcal{O}(t)&=\{(x',x_N) \in \mathbb{R}^{N-1}\times \mathbb{R}\,;\,x_N<h(t,x')\},
\end{align*}
and the normal $n$ and normal velocity $V_n$ are given by
$$
n=\frac{1}{\sqrt{1+|\nabla_{x'} h|^2}}\begin{pmatrix} -\nabla_{x'} h,\\ 1\end{pmatrix},\quad
V_n=n\cdot \fract \begin{pmatrix} x' \\ h(t,x')\end{pmatrix}=\frac{\partial_th}{\sqrt{1+|\nabla_{x'} h|^2}}.
$$

\begin{figure}[ht]
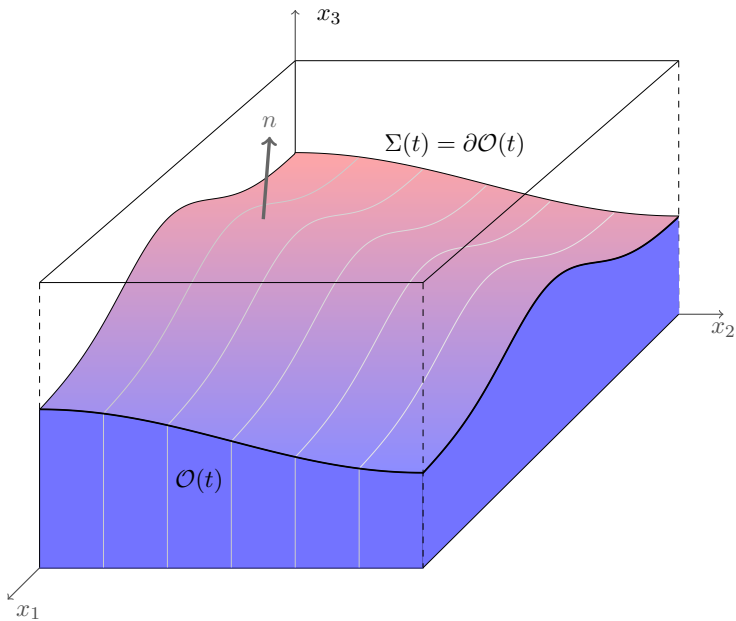

\centering
\resizebox{0.7\textwidth}{!}{\figurezero}
\caption{The fluid domain}\label{Fig:figure1}
\end{figure}

It has 
long\footnote{See~\cite{CaOrSi-SIAM90,Escher-Simonett-ADE-1997,Pruss-Simonett-book,SCH2004} 
as well as the discussions in \S\ref{S:1.1} and in Remark~\ref{R:3.4}.} 
been observed that the Hele-Shaw 
problem can be expressed as a parabolic evolution 
equation involving only the function $h$ (we derive such an equation in $\S\ref{S:2.1}$). 
This equation possesses two 
key characteristics: it is both a fractional parabolic equation and highly nonlinear. 
These attributes 
are shared by several equations that have attracted a lot of attention in recent 
years\footnote{We will discuss this in detail in section~\ref{S:comparison}.}. 
Among these equations, a notable feature of 
the Hele-Shaw equation is its ability to be reformulated as an obstacle 
problem. {Problems 
with obstacles give rise to free boundary problems (namely, the contact set 
between the solutions 
and the obstacle). 

However, the possible connections between obstacle problems 
and problems with free boundaries in fluid dynamics are not obvious. 
The first pioneering works in this direction are those of Baiochi \cite{Baiocchi-1971,Baiocchi-1972,Baiocchi-1974} 
for a particular stationary problem in porous medium, and Duvaut~\cite{Duvaut-1973} for the 
Stefan problem. Many other fluid dynamics questions on free boundary 
problems can be reformulated as a variational elliptic problem. This is indeed a vast subject and we refer the reader to Br\'ezis 
and Stampacchia \cite{Brezis-Stampacchia-1973}, and the books by 
Baiocchi and Capelo \cite{MR0745619} or 
Kinderlehrer and Stampacchia 
\cite{MR1786735}.}

What makes the obstacle problem formulation particularly appealing is 
that the time variable only appears as a parameter, enabling us to study 
this {\em evolution equation} with techniques developed for 
the study of {\em elliptic equations}. In some respects, this represents the most 
elementary formulation of the Hele-Shaw equation\footnote{This approach 
was once widely employed for related problems: it was discussed 
or evoked by Baiocchi, Br\'ezis, 
and Kinderlehrer and in their lectures at the ICM in 1974 (see~\cite{ICM1974}). 
However, it is safe to say that it gradually faded from the spotlight. Its 
resurgence came about with the research of Figalli, 
Ros-Oton and Serra~\cite{MR4179834,figalli2021singular} (for more details, refer to Figalli's plenary 
lecture at the ICM in 2018~\cite{Figalli-ICM2018}).}. For the Hele-Shaw problem, this insight was originally observed by 
Elliott and Janovsk\'y (\cite{MR611303}) for bounded fluid domains. The first 
contribution of this paper is the extension of this obstacle problem 
formulation to an unbounded domain. 

To get this obstacle problem formulation, the first step consists in 
introducing an elementary change of variables, which we call {\em free falling coordinates} 
since they eliminate the gravity in the equation: we will work with the unknowns 
\be\label{i1}
f(t,x')= h(t,x')+gt, \quad  p(t,x)=
\left\{
\begin{aligned}
&P(t,x',x_N-gt) \quad&&\text{for }x_N<f(t,x'),\\
&0 \quad&&\text{for }x_N\ge f(t,x').
\end{aligned}
\right.
\ee
Then, we use the so-called Baiocchi-Duvaut transformation, 
which involves the temporal integration of the pressure function. We set
$$
u(t,x)=\int_0^t p(s,x)\ds.
$$
Since the pressure within the fluid domain is positive, as can 
be deduced from the maximum principle, it follows 
that $\partial_t f\ge 0$ which means that the sets 
\be\label{def:Omega}
\Omega(t)\defn\{p(t,x)>0\}=\{P(t,x)>0\}+te_N=\mathcal{O}(t)+te_N
\ee
are increasing in time 
(see Figures~\ref{Fig:shift} and~\ref{Fig:equations}).  
This has the following key consequence: 
$\Omega(t)=\{ u(t,x)>0\}$. 
Furthermore, we will see in Section~\ref{S:2} that $u(t)$ 
satisfies the following obstacle problem:
\begin{equation*}
\left\{
\begin{aligned}
&\Delta u(t)=0 \quad &&\text{in}\quad \Omega(0)\cup \big(\xR^N\setminus \Omega(t)\big),\\
&\Delta u(t)=1 \quad &&\text{in}\quad \Omega(t)\setminus \Omega(0),
\end{aligned}
\right.
\end{equation*}
supplemented with the boundary conditions
\begin{equation*}
\left\{
\begin{aligned}
&\nabla u(t)=0 \quad \text{on}\quad\partial\Omega(t),\\
&\lim_{s\to -\infty}\nabla u(t,x+se_N) = -t e_N.
\end{aligned}
\right.
\end{equation*}
Once the elementary change of unknowns \e{i1} is taken into account, 
this reduction to an obstacle problem can be derived through classical 
computations (see Section~\ref{S:2}). 
Furthermore, we will present a broader approach for deriving the obstacle 
problem formulation, based on the co-area 
formula (cf.\ section~\ref{S:3.3}). This approach is interesting in 
that it also applies for sub- and supersolutions. 

\subsection{Study of the obstacle problem}
After rigorously establishing the previous 
calculations, the core of 
our analysis will be dedicated 
to the obstacle problem. The analysis of this problem, set in an {\em unbounded} domain, 
introduces intriguing challenges. 
More precisely, our goal is to establish the existence and uniqueness of solutions, with 
the most significant outcome being the derivation of a stability estimate under quite weak conditions.

Solutions will be constructed using an approximation approach involving obstacle 
problems in bounded domains, using comparison estimates. The analysis will 
draw upon various tools and questions, including classical variational methods, a Moser's iteration argument and Caccioppoli's inequality 
in the uniqueness argument, a Boundary Harnack inequality applicable in Lipschitz domains 
and a classification of nonnegative harmonic functions in 
unbounded 
domains located below the graph of a Lipschitz function. 

\subsection{The Cauchy problem for the graph case} 
The Cauchy problem for the Hele-Shaw equation has 
been investigated using various techniques (we will review many 
previous works in the next chapters). In this paper, 
we will examine both the graph and non-graph cases. 
We begin by examining the first case. We will 
establish the existence and uniqueness of a solution globally in time 
for smooth initial data $h_0$ (in subcritical Sobolev spaces). We will 
demonstrate that it is possible to extend the solution 
mapping $h_0\mapsto h(t)$ in a unique manner, thereby defining 
a semi-flow on the space of lower semicontinuous functions. This semi-flow 
will satisfy various properties of continuous dependence. 

Our main results 
about the Cauchy problem are presented in Section~\ref{S:mainR}. We can 
summarize them as follows: in any dimension $N\ge 2$, the following three properties are observed:

\begin{enumerate}
\item \textbf{Smoothness and Global Existence for $H^s$ Initial Data}: 
 When the initial data $h_0$ belongs to $H^s(\mathbb{R}^{N-1})$ for some $s>(N-1)/2+1$, 
 the solution exists globally in time and exhibits $C^\infty$ smoothness. This extends 
 previous results (see \cite{Cheng-Belinchon-Shkoller-AdvMath,AMS,MR4090462}), 
 which had established the local-in-time well-posedness of smooth solutions.
  
\item \textbf{Global Solvability and uniqueness for Rough Initial Data}: 
The Cauchy problem can be solved globally in time even when the initial data are very 
rough (see Figure~\ref{Fig:Holder}). For instance, if the initial data $h_0$ is uniformly continuous, 
a unique weak solution exists that is global in time and remains 
uniformly continuous as a function of space and time. This complements prior 
results in~\cite{Kim-ARMA2003,ChangLaraGuillenSchwab,AMS,MR4655356} (see Section~\ref{I:background} and subsequent chapters for further discussions of prior results and approaches, as well as for more references). 

Furthermore, we prove $L^p$-contraction estimates ($1\le p \le \infty$)
for the dependence with respect to initial data and that any modulus 
of continuity of the initial data is preserved. Moreover, we prove that the solution becomes analytic after a finite time without any regularity assumption on the initial data. See  \cite{MR2218549} for similar results under stronger assumptions on the initial data. 

\item\textbf{Waiting time and 
of smoothness for Lipschitz Initial Data}: when the initial data 
$h_0$ belongs to $W^{1,\infty}(\mathbb{R}^{N-1})$, 
we provide a comprehensive understanding of the interplay between 
two scenarios: a waiting time property (where the initial Lipschitz 
singularity persists for a finite duration) and a smoothing effect 
(where solutions instantaneously become smooth, see Figure~\ref{Fig:nonacute}). 
As discussed in Remark~\ref{R:2.6}, 
this extends previous results (\cite{MR1363758,MR2306045,zbMATH05034334,MR4520423}) to the case of an unbounded fluid domain. We give a simple proof that moreover applies for general domains (not only subgraphs). 
\end{enumerate}

\begin{figure}[htb]
\centering
\resizebox{0.8\textwidth}{!}{%
\begin{tikzpicture}[samples=100]
\draw [gray!40!white] (-5,-1) -- (13,-1) ;
\filldraw [color=blue!55!white] 
(-5,-1) -- (-5,-5) -- (13,-5) -- (13,-1)  -- cycle;
\filldraw[fill=blue!10!white, draw=blue!10!white] (-5,-1) -- (-5,-2) arc (180:270:1) -- (-4,-3) -- (-3,-3) arc (-90:0:1) -- (-2,-1)  -- cycle;
\draw [blue!10!white] (-2,-1) -- (-2,-2) ;
\filldraw[fill=blue!10!white, draw=blue!10!white] (-2,-1) -- (-2,-2) arc (180:270:1) -- (-1,-3) -- (0,-3) arc (90:0:1) -- (1,-1)  -- cycle;
\draw [blue!10!white] (-2,-1) -- (-2,-2) ;
\filldraw[fill=blue!10!white, draw=blue!10!white] (1,-1) -- (1,-4) arc (180:90:1) -- (2,-3) -- (3,-3) arc (-90:0:0.5) -- (3.5,-1)  -- cycle;
\draw [blue!10!white] (1,-1) -- (1,-2) ;
\filldraw[fill=blue!10!white, draw=blue!10!white] (3.5,-1) -- (3.5,-2.5) arc (180:270:0.5) -- (5,-3) -- (6,-3) arc (-90:0:2) -- (8,-1)  -- cycle;
\draw [blue!10!white] (4,-1) -- (4,-2) ;
\filldraw[fill=blue!10!white, draw=blue!10!white] (8,-1) arc (180:270:2) -- (10,-3) -- (11,-3) arc (-90:0:1) -- (12,-1)  -- cycle;
\draw [blue!10!white] (7,-1) -- (7,-2) ;
\filldraw[fill=blue!10!white, draw=blue!10!white] 
(10,-1) -- (10,-2) arc (180:270:1) -- (11,-3) -- (12,-3) arc (-90:0:1) -- (13,-1)  -- cycle;
\draw [blue!10!white] (10,-1) -- (10,-2) ;
\end{tikzpicture}
}\caption{A H\"older graph}\label{Fig:Holder}
\end{figure}
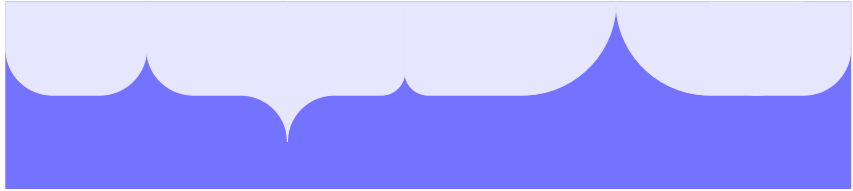
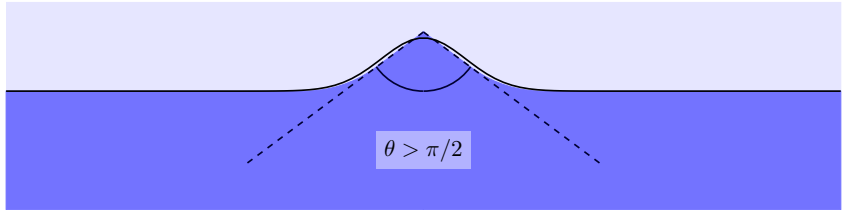
\begin{figure}[htb]
\centering
\resizebox{0.8\textwidth}{!}{
\begin{tikzpicture}[samples=100]
\filldraw [color=blue!10!white] (-7,5.5) -- (7,5.5) -- (7,2) -- 
(-7,2)  ;
\filldraw [color=blue!55!white]   
(-7,4) -- (-2,4) arc (-90:-55:2) -- (0,5) -- (0,2) -- (-7,2) -- cycle ; 
\draw[color=black!80!blue,dashed,thick] (0,5) -- (-3,2.75) ;
\draw[color=black!80!blue,thick]
  (0,4)
  arc[start angle=-90,end angle=-145,radius=1cm-1pt];
\begin{scope}[xscale=-1]
\filldraw [color=blue!55!white]   
(-7,4) -- (-2,4) arc (-90:-55:2) -- (0,5) -- (0,2) -- (-7,2) -- cycle ; 
\draw[color=black!80!blue,dashed,thick] (0,5) -- (-3,2.75) ;
\draw[color=black!80!blue,thick]
  (0,4)
  arc[start angle=-90,end angle=-145,radius=1cm-1pt];
\end{scope}
\node at (0,3) [fill=blue!30!white] {{$\theta>\pi/2$}};
\draw [white!80!black,draw=black,thick]
plot [domain=-7:7] ({\x},{4+0.9*exp(-\x*\x)})  ;
\end{tikzpicture}
}\caption{A non acute angle is regularized instantaneously}
\label{Fig:nonacute}
\end{figure}

\subsection{A generalized semi-flow}
Our main result is that one can expand the semi-flow defined on graph domains 
to encompass highly diverse initial domains. The only condition 
is that the initial fluid domain is connected and 
its boundary is located between two horizontal 
hyperplanes (without any assumption about regularity). 
Furthermore, we will establish that this semi-flow exhibits numerous continuous dependency properties. 

One application is to consider domains that are not situated beneath the graph of a function.

\begin{figure}[htb]
    \centering
    \resizebox{0.8\textwidth}{!}{%
\begin{tikzpicture}[x=5mm,y=5mm,decoration={mark random y steps,segment length=3mm,amplitude=1mm}]
\filldraw[fill=blue!10!white, draw=blue!10!white] (-4,1) -- (14.5,1) 
-- (14.5,-5) -- (-4,-5) -- cycle;
  \path[decorate] (-4, -0.5) -- (2,0) -- (2,0.5) -- (0.5,-4) -- (2,-3.5) -- (3.5,-2) -- (14.5, 0);
\filldraw [color=blue!55!white]  plot[color=blue!70!white,variable=\x,samples at={1,...,\arabic{randymark}},smooth] 
   (randymark\x) -- (14.5,-5) -- (-4,-5) -- cycle;
\end{tikzpicture}
}\caption{A smooth interface which is not a graph}
\end{figure}
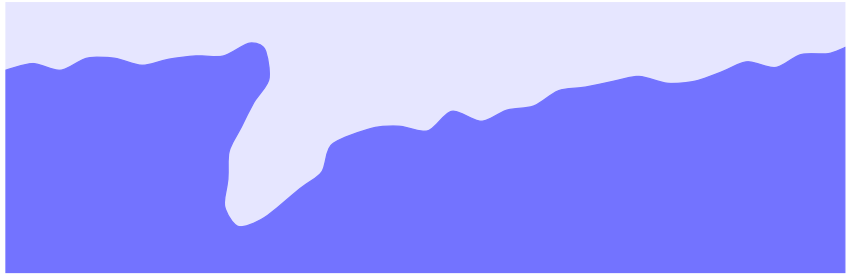

 \begin{figure}[htb]
    \centering
    \resizebox{0.8\textwidth}{!}{%
  \begin{tikzpicture}[samples=100,scale=0.7]

\shadedraw [top color=blue!70!white, bottom color=blue!30!white, color=blue!70!white]   
 plot [domain=-10.3:10.3] ({\x},{4+exp(-\x*\x)}) -- (10.3,-6) -- (-10.3,-6) -- (-10.3,4) ;

    \foreach \x in {-10,-9,...,9}{                           
    \foreach \y in {-1,0,...,3}{                        
    \node[blue!10!white,draw,circle,inner sep=1pt,fill] at (\x + 0.7*rnd,\y +0.9*rnd-0.3) {};
    }} 
    
    \foreach \x in {-10,-9,...,9}{                           
    \foreach \y in {-2,-1,...,3}{                      
    \node[blue!10!white,draw,circle,inner sep=0.5pt,fill] at (\x + 0.7*rnd,\y +0.9*rnd) {};
    }}

       \foreach \x in {-10,-9,...,9}{                          
    \foreach \y in {-3,-2,...,3}{                       
    \node[blue!10!white,draw,circle,inner sep=0.1pt,fill] at (\x + 0.7*rnd,\y +0.9*rnd) {};
    }}  

\draw [thick,draw=gray,dashed] (-10.3,-4) -- (10.3,-4) ;
     \end{tikzpicture}
}\caption{A fluid domain with bubbles of air inside}\label{Fig:bubbles}
\end{figure}
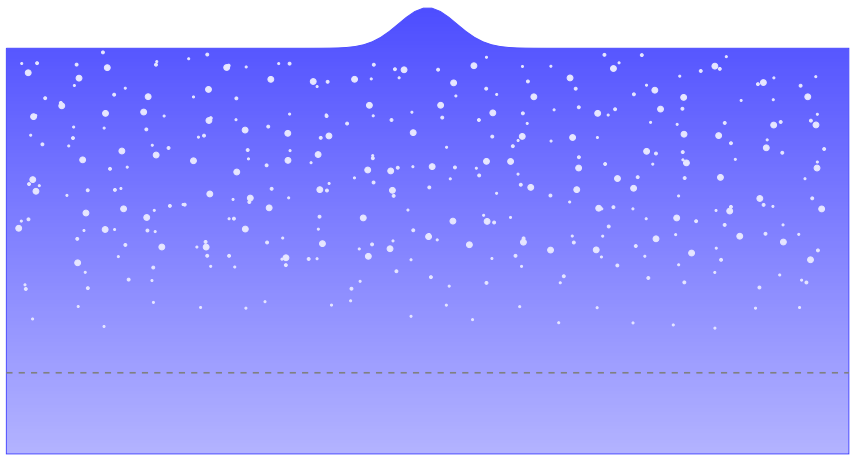

The main new aspect in this paper is the 
construction of a general {\em global in time} semi-flow, and the proof that  
this semi-flow will lead to the evolution of the fluid domain into a smooth subgraph after 
a certain time. Specifically, 
we will prove that the domain is a subgraph after a period 
of time $T\ge 2C$ where $2C$ 
is the width of the strip in which the initial free surface lies (see~\e{C}) and analytic after a period of time 
$\kappa C$ where $\kappa$ is a geometric constant depending on a regularity result of Caffarelli 
for the obstacle problem. We conjecture 
that smoothness holds  with $\kappa=2$.

Given the generality of our framework, numerous scenarios are conceivable. 
For instance, one can consider scenarios allowing changes in topology. 
Specifically, we can begin with an initial domain containing multiple bubbles (see Figure~\ref{Fig:bubbles}). 
Over time, we will prove that this domain will evolve to the point where the fluid domain becomes a smooth 
subgraph (and thus simply connected).

\subsection{Background}\label{I:background}

The Hele-Shaw problem has been a prominent problem over the last four decades, with many important contributions from different people, groups and perspectives. 
\begin{enumerate} 
\item  The obstacle formulation of Baiocchi \cite{Baiocchi-1971,Baiocchi-1972,Baiocchi-1974}, Duvaut
\cite{Duvaut-1973}  and 
Elliott and Janovsk\'{y} \cite{MR611303} translates the Hele-Shaw problem into a family of obstacle problems parametrized by time. The obstacle problem is well studied, which allows to use strong comparison argument, with  seminal contributions by Caffarelli and coworkers (see \cite{CS-book}) , and, more recently by Figalli and coworkers \cite{MR4179834,Figalli-ICM2018}. See also Section~\ref{S:2}.
\item Various type of weak solutions have been studied. 
\begin{itemize} 
\item Global weak solutions to the Hele-Shaw equations have been 
studied by Antontsev, Meirmanov, and Yurinsky \cite{MR2072944}, Perthame, Quir\'os, and V\'azquez \cite{MR3162474} and David and Perthame~\cite{MR4324293}. They regarded the 
Hele-Shaw equation as an evolution equation for $\chi_{\Omega(t)}$, as described 
in~\e{n42}, and established the existence of solutions by taking limits in a 
sequence of regularized equations of porous medium type. 
\item Viscosity solutions, 
on the other hand, have been the focus of studies conducted by Kim~\cite{Kim-ARMA2003}, Chang-Lara, Guillen, and Schwab \cite{ChangLaraGuillenSchwab}, as well as Dong, Gancedo and Nguyen~\cite{MR4655356}. 
They are close to our use of sub- and supersolutions for a study of the waiting time (see Section~\ref{sec:super} and Section~\ref{S:8}). 
Related problems have been studied by 
Maury, Roudneff-Chupin and Santambrogio~\cite{MR2735914}. Connections between an 
obstacle problem for the pressure 
and these solutions have already been studied by 
Guillen, Kim and Mellet~\cite{MR4367913}, and Kim, Mellet and Wu \cite{kim2023density}. 
\item One could probably build a theory at least for the periodic case around the approach by  Crandall  and  Liggett  \cite{MR0287357}. 
It is not clear whether it could be adapted to the  unbounded setting in this paper. This would be of independent interest. 
\end{itemize} 
\item  There is the perspective from elliptic problem in Lipschitz domains, by Jerison, Kim and coworkers \cite{MR2306045,MR2203166,MR2210142}, with tools like the boundary Harnack inequality and other techniques from harmonic analysis. The boundary Harnack inequality and its consequences for the Hele-Shaw problems 
are studied in Section~\ref{sec:bh}. 

\item  The Hele-Shaw problem can be considered as a pseudo-differential parabolic problem, with a focus on local existence for initial data given as the subgraph of function is Sobolev of Hölder spaces. This approach is discussed in more detail in Section~\ref{S:mainR}. 
Eventual regularity has been proven by Chang-Lara and Guillen \cite{changlara2016free} in this context. We provide a more classical argument: It is an immediate  consequence of Caffarelli's results that singularities in the obstacle problem only occur if the contact set is thin. This seminal result is applied in Section~\ref{sec:eventualreg} to prove the eventual regularity for the Hele-Shaw problem. 
\end{enumerate} 

This brief survey reflects the limited knowledge and the taste of the authors. We apologize for any omissions. 
Further references will be discussed in subsequent chapters.

\subsection{Notations}
We gather here some classical notations which will be used continually in the rest of the paper.

\begin{enumerate}
\item Given $0\le t\le T$  
and a function $\varphi\colon [0,T]\times D \to X$, we denote by $\varphi(t)$ 
the function $D\ni x \mapsto \varphi(t,x)\in X$.

\item The indicator function of a subset $A\subset \mathbb{R}^N$ is denoted by $\chi_A$.

\item Given $s\ge 0$, we denote by $H^s(\xR^d)$ the Sobolev space 
of functions $u\in L^2(\xR^d)$ 
such that $(I-\Delta)^{s/2}u$ belongs to 
$L^2(\xR^d)$, where $(I-\Delta)^{s/2}$ is the Fourier multiplier with symbol $(1+\la\xi\ra^2)^{s/2}$. 
\item We denote by $C(U)$ the space of functions which are continuous on $U$, without assuming that these functions are bounded. 
The notation $C^0(U)$ refers instead to the space 
of continuous and bounded functions.  
\end{enumerate}

%\subsection{Plan of the paper}

%Our main results about the Cauchy problem are presented in Section~\ref{S:mainR}. In Section~\ref{S:2} we will 
%rigorously establish the derivation of the obstacle problem and recall various known-results. Then the bulk of our analysis in Sections~\ref{sec:obstaclebounded}--\ref{S:5} will focus on this generalized obstacle problem. 

%%%%%%\clearpage
 
\section{Main results}\label{S:mainR}

In this chapter, we state our main results about the Cauchy problem for the Hele-Shaw 
problem. Our approach to introduce
a generalized semi-flow for the Hele-Shaw equation for general fluid domains unfolds 
in three sequential steps. 
After recalling the derivation of the Hele-Shaw equation in section~$\S\ref{S:2.1}$, 
we will consider, in $\S\ref{S:1.1}$, the case where the free surface can be 
represented as the graph of a regular function (we start with this case since defining a meaningful concept of a solution in this context is straightforward).
Subsequently, we will systematically expand these 
solutions in $\S\ref{S:1.3}$  
to encompass cases where the free surface corresponds 
to a graph of a rough function, specifically, any bounded 
lower semicontinuous function. Finally\footnote{In fact, 
the proof of our construction of a semi-flow for the Hele-Shaw 
equation goes in the opposite direction: it will commence 
with general domains and progressively refine it to less 
general ones. However, for the sake of clarity in explanation, 
it is more intuitive to reverse this process, beginning with 
smooth solutions where the meaning of solving the 
equation is unambiguous.}, in $\S\ref{S:2.4}$, we further 
extend our analysis to encompass general domains, permitting 
the free surface to serve as the boundary for any connected 
open subset enclosed between two parallel horizontal 
hyperplanes. The last section is about Lipschitz 
singularities.

\subsection{The equations}\label{S:2.1}
In this section, we provide a detailed derivation of the 
Hele-Shaw equation as an evolution equation involving 
the Dirichlet-to-Neumann operator. 

For the sake 
of simplicity in notation, we assume throughout 
the following that the acceleration due to gravity $g$, 
is set to a value of $1$.

Consider a time-dependent fluid domain which is at 
time $t\ge 0$ of the form
$$
\mathcal{O}(t)=\{x= (x',x_N) \in \xR^{N-1}\times \xR\,;\, x_N < h(t,x')\}.
$$
Assume that the velocity 
$v\colon \mathcal{O}\rightarrow \xR^{N}$ 
and the pressure $P\colon\mathcal{O}\rightarrow \xR$ 
solve the 
Darcy's equations, so that
\be\label{Darcy}
\nabla\cdot v=0, \quad\text{ and }\quad  v=-\nabla (P+x_N) \quad \text{in }\mathcal{O}.
\ee
One also imposes that 
$\lim_{x_N\to-\infty}v=0$. 
On the free surface $\Sigma=\partial\mathcal{O}$, 
it is assumed that 
the normal component of $v$ 
is equal to the normal component of the velocity 
of the free surface. This implies that
\be\label{HS3}
\partial_t h=\sqrt{1+|\nabla_{x'} h|^2} \, v\cdot n\quad \text{on}\quad x_N=h,
\ee
where $\nabla_{x'}=(\partial_{x_1},\ldots,\partial_{x_{N-1}})$ 
and $n$ denotes the outward unit normal to $\Sigma$, 
given by
$$
n=\frac{1}{\sqrt{1+|\nabla_{x'}h|^2}} \begin{pmatrix} -\nabla_{x'}h\\ 1 \end{pmatrix}.
$$
The final equation states that the pressure vanishes on the free surface:
$$
P=0\quad \text{on}\quad\Sigma.
$$

Notice that $v$ is a gradient, that is $v=-\nabla\phi$ 
where $\phi=P+gx_N$. In addition, 
since $\nabla\cdot v=0$, the function $\phi$ is harmonic, and hence 
is fully determined by its trace on the free surface, 
which is $h$ since $P=0$ on the free surface. We thus have
\be\label{HS4}
\Delta\phi=0\quad \text{in }\mathcal{O},\qquad \phi\arrowvert_{x_N=h}=h.
\ee
To summarize, the Hele-Shaw problem can be reduced to the equation 
\e{HS3} which is an evolution equation for $h$ only, since $\phi$ and 
therefore $v$ is entirely determined by $h$. 
Once $h$ is determined, one defines
$\phi$ by solving~\e{HS4} and then recovers the velocity and the pressure 
by setting $v=-\nabla\phi$ and $P=\phi-x_N$.

In order to explicitly write the problem as 
an evolution equation for the unknown $h$, we introduce the Dirichlet-to-Neumann operator $G(h)$. 
For a given time~$t$, that is omitted here, 
and a function $\psi=\psi(x')$, 
$G(h)\psi$ is defined by (see Appendix~\S\ref{A:DN} for details)
\begin{align*}
(G(h)\psi) (x')&=\sqrt{1+|\nabla_{x'} h|^2}\partial_n\varphi\arrowvert_{x_N=h(x')}\\
&=\partial_{x_N}\varphi(x',h(x'))-\nabla_{x'}h(x')\cdot\nabla\varphi(x',h(x'))
\end{align*}
where $\varphi$ is 
the bounded harmonic extension\footnote{Let us comment that part of the analysis 
will be devoted to study  such harmonic functions.} of $\psi$, given by
\be\label{defi:varphi}
\Delta \varphi=0\quad \text{in }\Omega,\qquad \varphi\arrowvert_{x_N=h}=\psi.
\ee
Then, with this notation, it follows that 
the equation~\e{HS3} can be written as follows
\begin{equation}\label{n7}
\partial_{t}h+G(h)h=0.
\end{equation}

The study of the Hele-Shaw problem with this formulation has received a lot of attention recently (see~\cite{ChangLaraGuillenSchwab,MR4090462,AMS}). 
It is remarkable that this formulation allows also to study viscosity solutions as well (see~\cite{ChangLaraGuillenSchwab} and \cite{MR4655356} where the authors prove the existence of such solutions using a new method which makes the solutions to satisfy the contour-dynamics formulation).

\subsection{Global existence of smooth solutions}\label{S:1.1}
To study the Cauchy problem, 
we begin by considering regular solutions. Many 
results have been obtained in the past decades about this problem. 
The study began with analysis of local in time well-posedness (as well as global regularity for small data), in works of Escher  and Simonett~\cite{Escher-Simonett-ADE-1997}, Antontsev, Gon\c{c}alves and Meirmanov~\cite{MR1942849}, 
Cheng, Granero-Belinch{\'o}n and Shkoller~\cite{Cheng-Belinchon-Shkoller-AdvMath}, Pr\"uss and Simonett~\cite{Pruss-Simonett-book}. 
In particular, there are now many possible ways to study this problem 
as well as various proofs that the Cauchy problem is well-posed {\em locally in time} 
for initial data in the Sobolev space $H^s(\xR^d)$ 
with $s>1+d/2$ (\cite{AMS,MR4090462}). 
We refer to the statement of Theorem~\ref{T:Cauchy} which recalls several known results. 

Our first main result asserts that the Cauchy problem is in fact well posed globally in time and moreover any Sobolev norm $\lA \cdot\rA_{W^{1,p}}$ is a Lyapunov functional.

\begin{theorem}\label{IT1}
Consider an integer 
$d=N-1\ge 1$ and a real number $s>d/2+1$. 
For any initial data $h_0$ in $H^s(\xR^d)$, the Cauchy problem
\begin{equation*}
\partial_{t}h+G(h)h=0,
\quad h\arrowvert_{t=0}=h_0,
\end{equation*}
has a unique global solution 
\begin{equation*}
h\in C^0([0,+\infty);H^s(\xR^d))\cap C^\infty((0,+\infty)\times \xR^d).
\end{equation*}
Moreover, $t\mapsto \lA h(t)\rA_{L^{p}}$ and $t\mapsto\lA \nabla h(t)\rA_{L^{p}}$ are decaying for any $p\in [1,+\infty]$. 
\end{theorem}
\begin{remark}
$(i)$ The fact that $\lA \nabla h(t)\rA_{L^{p}}$ are Lyapunov functionals was known for $p=+\infty$ (as a consequence of the maximum principle~\cite{Kim-ARMA2003,ChangLaraGuillenSchwab,AMS}). 
We extend this to $1\le p<\infty$ and in fact 
give a general principle to get Lyapunov functionals (see Remark~\ref{R:5.12}).

$(ii)$ All these Lyupunov functionals are a consequence of a 'mother' $L^1$ stability type estimate presented later in Theorem \ref{thm:mainobstacle}, part \ref{item:stability}.
\end{remark}

\subsection{General initial data}\label{S:1.3}
Our second result extends the notion of a solution in a unique 
manner to a very general setting. 
This extension applies to initial data that are barely bounded and lower semicontinuous. 
In particular, notice that 
we might consider initial fluid domain which do not have locally a 
finite perimeter for $N=2$ (and, more broadly, a finite area functional in dimensions $N\ge 2$). 

%%%%%%\clearpage

\begin{theorem}\label{T:B}
Let $d=N-1\ge 1$ and denote by $\mathcal{B}$ the space of 
bounded lower semicontinuous 
functions $f\colon \xR^d\to\xR$. 
There exists a semi-flow
\[
\Psi\colon  [0,\infty) \times  \mathcal{B} \to \mathcal{B}
\]
with the following properties: 
\begin{enumerate} 
\item It generalizes the notion of smooth solutions: for any 
$h_0\in H^s(\xR^d)$ with $s>1+d/2$, there holds 
$\Psi(t,h_0)=h(t)$ 
where $h$ is given by Theorem~\ref{IT1}.

\item It is a semi-flow, i.e. 
\[  \Psi(t,\Psi(s,h_0)) = \Psi(t+s,h_0) ,\qquad \Psi(0,h_0) = h_0.\]

\item It is a monotone contraction in the sense 
\begin{equation}
\left\{
\begin{aligned}
&h_0 \le g_0 \quad \Longrightarrow \quad\Psi(t,h_0) \le \Psi(t,g_0), \\[1ex]
&\Vert \Psi(t,h_0) - \Psi(t,g_0) \Vert_{L^p} \le \Vert h_0-g_0 \Vert_{L^p} \quad\text{for all}\quad p \in [1,+\infty].
\end{aligned}
\right.
\end{equation}
\item 
It is monotone in time in the following sense: 
for all $h_0\in\mathcal{B}$, the function $f$ defined by
$$
f(t,x')=h(t,x')+t \quad\text{with}\quad 
h\colon (t,x') \mapsto \Psi(t,h_0)(x')
$$
satisfies $f(s,x') \le f(t,x') $ if $s \le t $. Moreover, it is continuous in time in the sense that, for all $x'\in \xR^d$ and all $0\le s< t$,
\[ \Big| \int_{B^{\xR^{d}}_R} f(t,x') - f(s,x') \dx' - (t-s)\big\vert B^{\xR^d}_R\big\vert\Big| \lesssim (t-s)\lA h_0\rA_{L^\infty}^{\frac25} R^{d-\frac25}, \] 
for any ball $B^{\xR^{d}}_R\subset \xR^d$ of radius $R>\Vert h_0\Vert_{L^\infty}$.

\item It preserves regularity: If $ h_0$ has a modulus of continuity 
$\omega$ then $\omega $ is a modulus of continuity for $\Psi(t,h_0)$. 
In particular $\Psi(t,h_0) $ is Lipschitz continuous with the same Lipschitz 
constant as $h_0$ if $h_0$ is Lipschitz continuous. If $ h_0$ is uniformly 
continuous then the function 
$h\colon (t,x') \mapsto \Psi(t,h_0)(x')$ 
is continuous on $[0,\infty) \times \xR^d$.
\end{enumerate}
\end{theorem}

\subsection{General fluid domains}
\label{S:2.4}

A key aspect of this obstacle problem formulation is that it will allow us to
consider general initial domains.% These domains need not be situated under the
%graph of a function, nor do they require finite perimeters. Instead, 
We will see that it suffices to assume that initially the fluid domain is an open and connected set with its boundary contained
within a strip.

\begin{notation}
We denote by $\mathcal{U}$ the set of all 
open connected 
sets $\mathcal{O}\subset \xR^N$ 
for which there exists a constant $C>0$ such that
\be\label{C}
\{ x_N < -C \} \subset \mathcal{O} \subset \{ x_N < C\}.
\ee
We say that $\mathcal{O}\in\mathcal{U}$ is a subgraph it there exists 
a function $f:\xR^{N-1} \to \xR$ so that 
\[ \mathcal{O} = \{ x=(x',x_N)\in \xR^{N-1}\times \xR : x_N < f(x') \}.
\]
\end{notation}

The following theorem states that there exists a 
semi-flow defined on $\mathcal{U}$ which maps the initial fluid domain to its value at a given time. Part of the statements will be given using the change of variables, called free-falling coordinates, which amounts to work with the sets $\Omega_t$ given by \e{def:Omega}. 

\begin{theorem}\label{thm:mainobstacleintro} 
Let $d=N-1\ge 1$. There exists a mapping
\[ \Phi:  [0,\infty) \times  \mathcal{U} \to \mathcal{U} \]
with the following properties: 
\begin{enumerate} 
\item It is a semi-flow, i.e.\ for all $s,t  \ge 0 $,
\[  \Phi(t, \Phi(s,\mathcal{O})) = \Phi(t+s, \mathcal{O})\]
and 
\[ \Phi(0, \mathcal{O}) =\mathcal{O}. \]
\item It generalizes the notion of smooth solutions: for any 
$\mathcal{O}_0\in \mathcal{U}$ of the form $\mathcal{O}_0=\{x_N<h_0(x')\}$ 
with $h_0\in \mathcal{B}$, there holds 
$\Phi(t,\mathcal{O}_0)=\{x_N<h(t,x')\}$ 
where $h$ is given by Theorem~\ref{T:B}.
\item It is a monotone contraction in the following sense
(here $\Delta $ denotes the symmetric difference and $\la\cdot\ra$ the Lebesgue measure): 
for all $\mathcal{O}_1$ and $\mathcal{O}_2$ in $\mathcal{U}$, we have\begin{equation}
|\Phi(t,\mathcal{O}_2) \Delta \Phi(t,\mathcal{O}_1)| \le |\mathcal{O}_2 \Delta \mathcal{O}_1|
\end{equation}
and
\[
\mathcal{O}_1 \subset \mathcal{O}_2 \Longrightarrow \Phi(t,\mathcal{O}_1)
\subset \Phi(t,\mathcal{O}_2).
\]
(Actually, we prove a stronger estimate, see Theorem~\ref{thm:mainobstacle}.)

\item Let $\Omega_t=\Phi(t,\mathcal{O})+te_N$. The semiflow is monotone in the sense that $\Omega_s \subset \Omega_t$ if $s\le t$ and 
it is continuous with respect to time in the sense that 
\[   \Big|\big| \big(\Omega_t\backslash \Omega_s \big) \cap \big(B_R^{\xR^{d}}(x') \times \xR\big)\big| - (t-s) |B_R^{\xR^d}(x')|    \Big|  \le c(N)|t-s|C^{\frac25}R^{d-\frac25},
\] 
for any ball $B^{\xR^{d}}_R(x')\subset \xR^d$ of radius $R>C$ and center $x'$.

\item (Eventual subgraph) Outside the initial strip $\xR^d\times [-C,C]$, the set $\Omega_t$ is the subgraph of a lower semicontinuous function. More precisely: if $t>0$ and $ (x',C) \in \Omega_t$ then there exist $r>0$ and a lower semicontinuous function
$$
f\colon B_r^{\xR^{d}}(x') \to (C,\infty)
$$
so that 
\[ \Omega_t \cap (B_r(x')\times \xR) = \{  y: y' \in B_r(x'), y_N < f(y')\}.  \]
In particular, for $t\ge 2C$, $\Omega_t$ (and hence $\Phi(t,\mathcal{O})$) is a subgraph.

\item  (Eventual regularity) \label{part:eventual}
There exists $ \kappa>0$ such that if
$\mathcal{O}\in\mathcal{U}$ satisfies \e{C}, then, for all $t\ge C \kappa$, $\Phi(t,\mathcal{O})$ 
is a subgraph of an analytic function. 
\end{enumerate}

\end{theorem}

\subsection{The waiting time}
Let us emphasize the connection between the scaling of the equation 
and Lipschitz regularity. Specifically, if $h=h(t,x')$ solves the Hele-Shaw 
equation, then a rescaled function 
defined as $h_\lambda(t,x')=\lambda^{-1}h(\lambda t,\lambda x')$ 
also satisfies the same equation, and hence the scaling invariance corresponds to Lipschitz regularity. The preceding theorems apply in two 
distinct 
scenarios: firstly, the sub-critical case, where, in terms of Sobolev embedding, 
the initial data exhibits at least $C^{1,\alpha}$ regularity for some $\alpha>0$. In this case, we observed that the solution immediately becomes smooth. In addition, we have presented a pair of results that allow us to consider initial data that are considerably less regular than Lipschitz. 

The final result presented in this section lies at the interface between these various results. It pertains specifically to Lipschitz initial data and is localized in nature. This result postulates that a Lipschitz singularity, reminiscent of the apex of a cone, will either undergo immediate smoothing or persist for a positive duration. This dichotomy depends on the angle of the cone.

Consider an angle
$\alpha\in (0,2\pi)$ and denote by
$C_\alpha$ the cone with aperture $\alpha$ and apex 
at the origin, 
that is 
\[
C_\alpha = \{ x=(x',x_N) \in \xR^{N-1}\times \xR : |x'|<-\tan(\alpha/2) x_N  \}.
\]

\begin{theorem}\label{T:di}
Let $d=N-1$ and $ \mathcal{O}_0 \in \mathcal{U}$ 
as above and introduce the sets
$$
\mathcal{O}_t=\Phi(t,\mathcal{O}_0)
\quad \text{and}\quad\Omega_t=\mathcal{O}(t)+te_N
$$
and consider a point $x_0 \in \partial \mathcal{O}_0=\partial \Omega_0$. The following results hold:
\begin{enumerate}
\item Suppose there exist $r>0$ and a round cone $C_\alpha$ with angle of opening 
$\alpha <2\arctan{\sqrt{d}} $ so that 
\[
B_r(x_0) \cap \mathcal{O}_0 \subset B_r(x_0) \cap (C_\alpha+ x_0).
\]
Then there exist $T>0$ and $\rho >0$ such that 
$ x_0 \in \partial \Omega_t $ for $t\le T$ and
\[ \Omega_t \cap B_\rho( x_0)\subset   B_\rho(x_0)\cap  (C_{\arctan \sqrt{d}}+x_0). 
\]
\item On the other hand, if there exists $r>0$ and a round cone $C_\alpha$ of angle 
$\alpha > 2\arctan \sqrt{d} $ so that 
\[ B_r(x_0) \cap (C_\alpha+x_0) \subset B_r(x_0) \cap \mathcal{O}_0, \]
then 
\[ x_0 \in \Omega_t  \]
for all $t>0$.
\end{enumerate}
Moreover if $\mathcal{O}_0$ is a subgraph of a bounded lower semicontinuous function then $\mathcal{O}_t$ is again a subgraph of a bounded lower semicontinuous functions which is analytic near $x_0'$
in the second case for all $t>0$.
\end{theorem} 
\begin{remark}\label{R:2.6}
$i)$ In \cite{MR1363758}, King, Lacey and V\'azquez provide an interesting and more detailed study of selfsimilar solutions in $\xR^2$ which by comparison arguments implies the waiting time results of this paper for $N=2$. Concerning the immediate smoothing, let us mention that 
Choi, Jerison and Kim \cite{MR2306045} proved 
a deeper local immediate regularization
theorem under a less precise condition on the
angle. These previous results were later generalized by Choi and Kim~\cite{zbMATH05034334}, who obtain a dichotomy similar to Theorem~\ref{T:di} for the Hele-Shaw problem in a bounded domain. See also the recent paper by Agrawal, Patel and Wu \cite{MR4520423} for a different treatment of the two-dimensional problem without use of the maximum principle at the prize of much stronger regularity assumption on the initial data
(away from the singularity) and the benefit of techniques applicable in situations where a maximum principle is not available. 

In this paper, we use different tools and we will provide fairly direct and elementary proofs based on simple explicit sub- and supersolutions and comparison arguments.

$ii)$ Observe that the first part seems to contradict Theorem 1.2 of Figalli, Ros-Oton and Serra \cite{MR4179834}. However their formulation of the Hele-Shaw problem is slightly different and applies away from the initial domain. 
\end{remark}

\subsection{Comparison with other parabolic type equations}\label{S:comparison}
In recent decades, significant progress has been made in understanding the Cauchy
problem for fractional parabolic equations. While we have already discussed various results related to the 
Hele-Shaw equation, in this section, we aim to shed light on some key aspects of 
the Hele-Shaw flow by comparing it to other equations.

To begin, it is insightful to compare these results with those for the
Muskat equation, which corresponds to a two-phase version of the Hele-Shaw equation. 
C\'ordoba and Gancedo~\cite{CG-CMP} introduced a beautiful concise formulation
of the Muskat equation:
\begin{equation}\label{Muskat}
\partial_tf=\frac{1}{\pi}\int_\mathbb{R}\frac{\partial_x\Delta_\alpha f}{1+\left(\Delta_\alpha f\right)^2}\mathrm{d}\alpha \quad\text{where} \quad\Delta_\alpha f(t,x)=\frac{f(t,x)-f(t,x-\alpha)}{\alpha}\cdot
\end{equation}
Using this formulation, 
Cameron~\cite{Cameron} established the existence of a modulus of continuity for 
the derivative, leading to a global existence result with the sole requirement that the product of 
the maximal and minimal slopes remains bounded by 1. More recently, Abedin and Schwab also obtained a modulus
of continuity in \cite{Abedin-Schwab-2020} through the Krylov-Safonov estimates. 
Furthermore, C\'ordoba and Lazar, using a novel formulation of the Muskat equation involving 
oscillatory integrals \cite{Cordoba-Lazar-H3/2}, demonstrated that the Muskat equation 
is globally well-posed in time, 
with conditions that the initial data is suitably smooth and 
that the $\dot H^{3/2}(\mathbb{R})$-norm is sufficiently 
small (extended to the 3D case in \cite{Gancedo-Lazar-H2}). 
Additionally, \cite{AN3} established that the Cauchy problem is well-posed locally in time
on the endpoint Sobolev space $H^{3/2}(\mathbb{R})$, which is optimal with respect to the scaling of the equation. Indeed, 
there are blow-up results for certain large enough data 
by Castro, C\'ordoba, Fefferman, Gancedo, and L\'opez-Fern\'andez \cite{CCFG-ARMA-2013,CCFG-ARMA-2016,CCFGLF-Annals-2012} (in contrast to what we prove in this paper). Eventually, Garc\'{\i}a-Ju\'arez, G\'omez-Serrano, Haziot 
and Pausader~\cite{GJGSHP2023desingularization} proved recently a result comparable to instantaneous smoothing for the Muskat equation~\e{Muskat}, 
for initial data which contain a finite set of small corners, providing 
a precise description of the solution.

It is worth noting that recent research has focused on the existence and potential 
non-uniqueness of weak solutions, especially in the unstable regime where the heavier 
phase is situated above the lighter one. Several papers have addressed this issue \cite{Brenier2009,cordoba2011lack,szekelyhidi2012relaxation,castro2016mixing,forster2018piecewise,noisette2020mixing}.

Numerous papers have also explored critical problems in other parabolic 
equations. Consider, for example, the equation:
\begin{align*}
\partial_t\theta+u\cdot\nabla \theta+(-\Delta)^{\frac{\alpha}{2}}\theta=0\quad \text{with}
\quad u=\nabla^\perp(-\Delta)^{-\frac{1}{2}}\theta.
\end{align*}
This equation arises as a dissipative version of the surface quasi-geostrophic 
equation introduced by Constantin-Majda-Tabak \cite{CMT-1994}. In the 
critical case, where $\alpha=1$, global well-posedness in time has been 
established by Kiselev-Nazarov-Volberg \cite{KNV-2007}, 
Caffarelli-Vasseur \cite{Caffarelli-Vasseur-AoM}, and Constantin-Vicol \cite{CV-2012} (see also \cite{KN-2009,Silvestre-2012,VaVi}). 
Lastly, we would like to mention papers by Vazquez \cite{Vazquez-DCDS}, Granero-Belinch\'on, 
and Scrobogna \cite{Granero-Scrobogna}, which discuss related global existence results 
for different fractional parabolic equations, as well as Cobb~\cite{cobb2023wellposedness} 
and Dalibard, Guillod and Leblond~\cite{dalibard2023longtime} for the study of the Cauchy problem for fractional Stokes-transport systems.

\section{The Hele-Shaw equation as an obstacle problem}\label{S:2}

In this Section we will rigorously establish the derivation of the
obstacle problem formulation and recall various known-results.

\subsection{Regular solutions}

We begin by recalling a well-posedness result along with some bounds that are valid
for sufficiently smooth solutions of the Hele-Shaw equation.

\begin{theorem}[from \cite{AMS,MR4090462}]\label{T:Cauchy}
Consider an integer 
$d\ge 1$ and a real number $s>d/2+1$. 
For any initial data $h_0$ in $H^s(\xR^d)$, 
there exists 
a time $T>0$ depending only on $d$, $s$ and $\lA h_0\rA_{H^s}$ 
such that the Cauchy problem
\begin{equation}\label{Hele-Shaw100}
\partial_{t}h+G(h)h=0,
\quad h\arrowvert_{t=0}=h_0,
\end{equation}
has a unique solution 
\be\label{n34}
h\in C^0([0,T];H^s(\xR^d))\cap C^\infty((0,T]\times \xR^d).
\ee
Moreover, this solution satisfies the pointwise bound:
\begin{equation}\label{n30}
G(h)h<1,
\end{equation}
together with the two following maximum principles:
\begin{align}
&\sup_{x'\in \xR^d}\la h(t,x')\ra\le \sup_{x'\in \xR^d}\la h(0,x')\ra,\label{n31}\\
&\sup_{x'\in \xR^d}\la\nabla_{x'}h(t,x')\ra\le \sup_{x'\in \xR^d}\la \nabla_{x'}h(0,x')\ra.\label{n32}
\end{align}
\end{theorem}
\begin{remark}\label{R:4}
$(i)$ Denote by $\phi$ the harmonic extension of $h$ which we assume 
to be of class $C^{1,\alpha}$, that is the solution to 
\e{n80}--\e{n82} with $\psi$ replaced by $h$. Then set $P(t,x)=\phi(t,x)-x_N$. 
Then we will prove in Proposition~\ref{prop:taylor} that
\be\label{n83}
P>0 \quad\text{in}\quad\mathcal{O}(t) 
\quad\text{and}\quad \lambda \ge -\partial_{x_N}P\ge \frac{1}{\lambda}\quad\text{on}\quad \partial\mathcal{O}(t),
\ee
for some constant $\lambda$ depending on $\lA h\rA_{C^{1,\alpha}}$. 
The latter inequality, related to the fact that the 
Taylor coefficient is positive, is a uniform version of the property \eqref{n30}. 
These results were already obtained for smoother interfaces 
(see~\cite{AMS,CCFG-ARMA-2013,CCFGLF-Annals-2012,CCFG-ARMA-2016,
Cheng-Belinchon-Shkoller-AdvMath,MR4090462}). 
In particular, for the solution $h$ given by the previous theorem, for all $0< t < T$, there holds 
\begin{equation}\label{boundedbelowlambda}
h(t,x') \ge h(0,x') + t -  \frac{t}\lambda.   \end{equation}

$(ii)$ We refer to~\cite{Kim-ARMA2003,ChangLaraGuillenSchwab,AMS} for a discussion 
of the maximum principles~\e{n31}--\e{n32} together with other 
inequalities. More generally, it is known, for smooth enough solutions, that whenever $\omega$ is a 
modulus of continuity for $h(0,\cdot)$, $\omega$ is also a modulus of
continuity for $h(t,\cdot)$, for any $t \in [0,T]$.
\end{remark}

\begin{definition}\label{D:1}
We say that $h$ is a regular solution to \e{Hele-Shaw100} defined on $[0,T]$ 
if $h$ satisfies the 
conclusions \e{n34}--\e{n32} 
of the above result.
\end{definition}

\subsection{Free falling coordinates}\label{S:2.3one}
In this section, as explained in the introduction (see~$\S\ref{S:I1.1}$), 
we will utilize a change of variables 
to remove gravity $g=1$ from the equation and move gravity  to the boundary 
condition at $x_N=- \infty$, see \eqref{eq:bcinfty} below. 

Consider a regular solution $h$ to \e{Hele-Shaw100}, defined over 
some time interval $[0,T]$. We start by performing a basic change 
of variables to simplify the problem, making it such that the fluid 
domain increases over time. Specifically, we 
introduce\footnote{Recall that we set $g=1$. Otherwise, we would consider $f(t,x')=h(t,x')+gt$.}:
$$
f(t,x')\defn h(t,x')+t,
$$ 
and then set
$$
\Omega(t)\defn \{ (x',x_N)\in\xR^{N-1}\times\xR \sep x_N<f(t,x')\}=\mathcal{O}(t)+t e_N,
$$
where $e_N$ is the unit vector with coordinates $(0,\ldots,0,1)$.

Then we have
\be\label{n30c}
G(f)f=G(h)h,
\ee
as can be verified by noting that the harmonic extension of 
$f(t,\cdot)$ in $\Omega(t)$ 
is simply given by $\varphi(t,x)=\phi(t,x',x_N-t)+t$ where $\phi(t,\cdot)$ 
is the harmonic extension of $h$ in 
$\mathcal{O}(t)$. We deduce that
\be\label{n36}
\partial_tf+G(f)f=1.
\ee
Now, it 
follows at once from \e{n36}, \e{n30} and \e{n30c}  
that (recall that we assume that $f \in C^{1,\alpha}$)
\be\label{n36.5}
\partial_t f > 0.
\ee
This implies that $\Omega(t_1)\subset\Omega(t_2)$ for $t_2>t_1$, which will play a 
crucial role (see Figure~\ref{Fig:shift}).

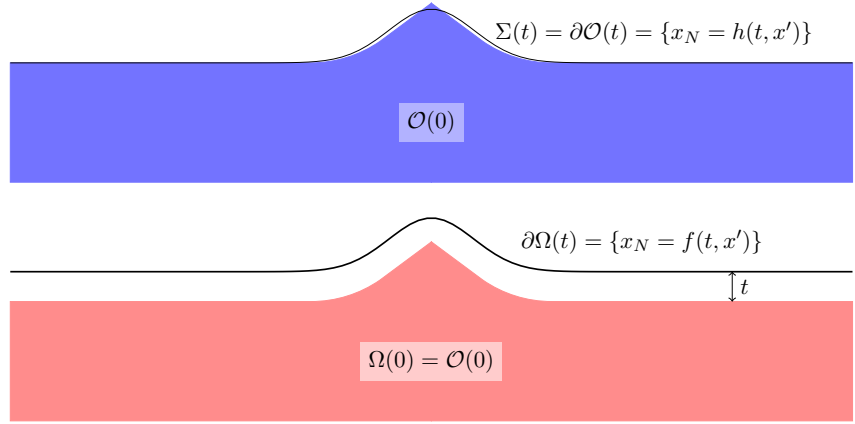
\begin{figure}[htb]
\centering
\resizebox{0.8\textwidth}{!}{%
\begin{tikzpicture}[samples=100]
\filldraw [color=blue!55!white]   
(-7,4) -- (-2,4) arc (-90:-55:2) -- (0,5) -- (0,2) -- (-7,2) -- cycle ; 
\begin{scope}[xscale=-1]
\filldraw [color=blue!55!white]   
(-7,4) -- (-2,4) arc (-90:-55:2) -- (0,5) -- (0,2) -- (-7,2) -- cycle ; 
\end{scope}
\draw [black]
plot [domain=-7:7] ({\x},{4+0.9*exp(-\x*\x)})  ;
\node at (0,3) [fill=blue!30!white] {{$\mathcal{O}(0)$}};
\node at (3.7,4.5)  {$\Sigma(t)=\partial\mathcal{O}(t)=\{x_N=h(t,x')\}$};
\filldraw [color=red!45!white]   
(-7,0) -- (-2,0) arc (-90:-55:2) -- (0,1) -- (0,-2) -- (-7,-2) -- cycle ; 
\begin{scope}[xscale=-1]
\filldraw [color=red!45!white]   
(-7,0) -- (-2,0) arc (-90:-55:2) -- (0,1) -- (0,-2) -- (-7,-2) -- cycle ; 
\end{scope}
\draw [draw=black,thick]
plot [domain=-7:7] ({\x},{0.5+0.9*exp(-\x*\x)})  ;
\node at (0,-1) [fill=red!20!white] {{$\Omega(0)=\mathcal{O}(0)$}};
\draw[<->,black] (5,0) -- (5,0.5) ;
\node [right]at (5,0.25)  {$t$};
\node at (3.5,1)  {$\partial\Omega(t)=\{x_N=f(t,x')\}$};
\end{tikzpicture}
}\caption{Illustrate the difference between $\mathcal{O}(t)$ and $\Omega(t)$. 
Here, initially, $\mathcal{O}(0)$ represents a domain characterized 
by a Lipschitz singularity. 
This singularity is immediately smoothed out, and one can visualize 
the mass spreading, as depicted in the upper illustration. In this 
scenario, 
$\mathcal{O}(t)$ does not contain $\mathcal{O}(0)$. However, by vertically shifting $\mathcal{O}(t)$lly by an amount of $t$, the resulting domain 
$\Omega(t)$ contains $\Omega(0)=\mathcal{O}(0)$ for all $t>0$.}\label{Fig:shift}
\end{figure}

Recall that $P$ denotes the pressure, that is the function 
$$
P(t,x)=\phi(t,x)-x_N
$$
where $\phi$ denotes the harmonic extension of $h$ (see~\e{HS4}). 
At a given time $t\in [0,T]$, we define 
\be\label{n36.7}
p(t,x)=P(t,x',x_N-t).
\ee
With these new unknowns, the problem becomes 
\begin{alignat}{3}
- \Delta p  &  = 0 
\qquad &&\text{ in } &&\Omega(t),\notag \\
\lim_{x_N\to -\infty} \partial_{x_N} p  & = -1 \qquad &&\text{ in } &&\xR^{N-1}, \label{eq:bcinfty}\\ 
 p& = 0 \qquad &&\text{ on } &&\partial\Omega(t),\notag\\
 \partial_t f &= -\nabla p \cdot\left( \begin{matrix} -\nabla_{x'} f \\ 1 \end{matrix} \right) 
 \qquad &&\text{ on } &&\partial \Omega(t).\label{kc}
\end{alignat}  
\begin{remark}\label{R:3.4}
The kinematic boundary condition~\e{kc} can be expressed in various equivalent forms. 
To illustrate this, we start by observing that, given that $p=0$ on $\partial\Omega(t)$, 
through direct calculations using the chain rule, we obtain:
\be\label{n84}
\partial_t p=(\partial_{x_N}p)\partial_t f \quad\text{and}\quad
\nabla_{x'}p=-(\partial_{x_N}p)\nabla_{x'}f \quad\text{on}\quad \partial \Omega(t).
\ee
Consequently, as in Kim and Jerison \cite{MR2203166}, we have
\[
p_t - |\nabla p |^2= 0 \quad \text{ on } \partial \Omega(t).
\]

Notice in addition that \e{n83} implies that $-\partial_{x_N}p>0$ on $\partial\Omega(t)$. 
By combining this with the second identity in \e{n84}, we get that
$$
\la \nabla p\ra=-\sqrt{1+|\nabla_{x'}f|^2}(\partial_{x_N}p)\quad \text{ on } \partial \Omega.
$$
Consequently, by the implicit function theorem $\frac{p_t}{|\nabla p|}$ 
is the velocity of the free boundary in the upper normal direction $V_n$, 
and, as in 
Meimarnov and Zaltzman \cite{MR1925260}, there holds
\[ |\nabla p| = V_n. \] 
\end{remark}

\subsection{Baiocchi-Duvaut transformation}\label{S:2.3} Using the free 
falling coordinates introduced in the previous section, 
we will now introduce a change of unknowns to rewrite the 
Hele-Shaw problem as an elliptic obstacle problem, 
following the works of 
Baiocchi~\cite{Baiocchi-1971,Baiocchi-1972,Baiocchi-1974}, 
Duvaut~\cite{Duvaut-1973}, Elliott and Janovsk\'y~\cite{MR611303}, and Gustafsson~\cite{Gustafsson2}.

More precisely, we
follow Elliott and Janovsk'y \cite{MR611303}
and rewrite the problem as an obstacle problem. 
To do so, we begin by extending 
the pressure $p(t)$ by $0$ outside $\Omega(t)$ and still denote by $p(t)$ 
the extended function (recall that $p$ is defined by \e{n36.7}). Then we use a
Baiocchi-Duvaut transformation which consists of introducing the
new unknown $u\colon [0,T]\times\xR^N\to\xR$, defined by
\be\label{BDT}
u(t,x)=\int_0^t p(\tau,x)\dtau.
\ee
As we will see, a remarkable aspect of this transformation is that 
in the rest of the argument, the time variable will play 
the role of a parameter. Given $t \in (0,T]$, we will write simply 
$u(t)$ or even simply $u$ the function $\xR^N\ni 
x\mapsto u(t,x)\in \xR$.

\vspace{1cm} 

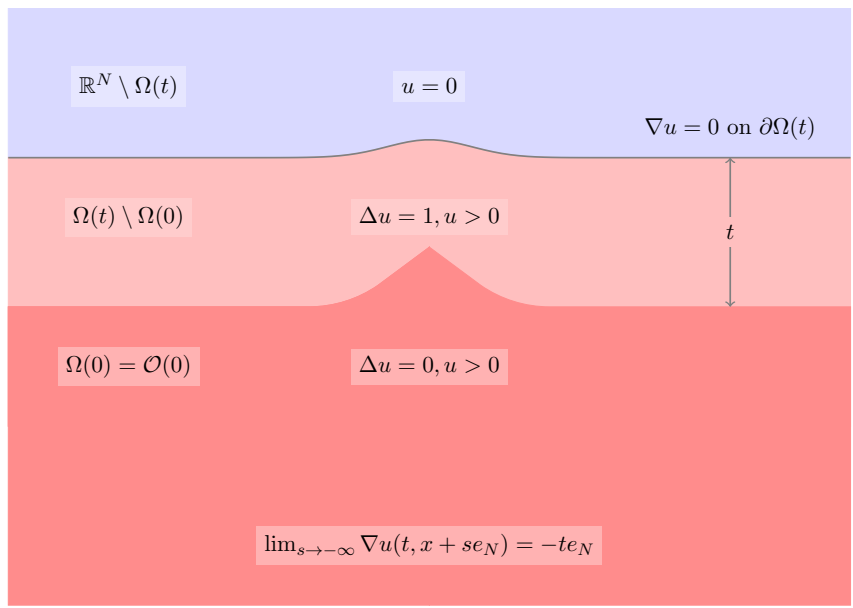
\begin{figure}[h]
\centering
\resizebox{0.8\textwidth}{!}{%
\begin{tikzpicture}[samples=100]
\filldraw [color=blue!15!white]  
(7,5) -- (-7,5) -- (-7,3) --   
plot [domain=-7:7] ({\x},{2.5+0.3*exp(-\x*\x)}) -- cycle ; 
\filldraw [color=red!25!white]   
plot [domain=-7:7] ({\x},{2.5+0.3*exp(-\x*\x)}) -- (7,-2) -- (-7,-2) -- cycle ; 
\filldraw [color=red!45!white]   
(-7,0) -- (-2,0) arc (-90:-55:2) -- (0,1) -- (0,-5) -- (-7,-5) -- cycle ; 
\filldraw [color=red!45!white]   
(-7,0) -- (-2,0) arc (-90:-55:2) -- (0,1) -- (0,-2) -- (-7,-2) -- cycle ; 
\begin{scope}[xscale=-1]
\filldraw [color=red!45!white]   
(-7,0) -- (-2,0) arc (-90:-55:2) -- (0,1) -- (0,-5) -- (-7,-5) -- cycle ; 
\end{scope}
\draw [draw=gray,thick]
plot [domain=-7:7] ({\x},{2.5+0.3*exp(-\x*\x)})  ;
\node at (-5,-1) [fill=red!30!white] {$\Omega(0)=\mathcal{O}(0)$};
\node at (-5,1.5) [fill=red!20!white] {$\Omega(t)\setminus\Omega(0)$};
\node at (-5,3.7) [fill=blue!10!white] {$\xR^N\setminus\Omega(t)$};
\node at (0,-1) [fill=red!30!white] {$\Delta u=0, u>0$};
\node at (0,-4) [fill=red!30!white] {$\lim_{s\to -\infty}\nabla u(t,x+se_N) = -t e_N$};
\node at (0,1.5) [fill=red!20!white] {$\Delta u=1, u>0$};
\node at (0,3.7) [fill=blue!10!white] {$u=0$};
\draw[<->,gray,thick] (5,0) -- (5,2.5) ;
\node at (5,1.25) [fill=red!25!white] {$t$} ;
\node at (5,3)  {$\nabla u=0 \text{ on }\partial\Omega(t)$};
\end{tikzpicture}
}\caption{The obstacle problem equations.}\label{Fig:equations}
\end{figure}

\clearpage

\begin{proposition}\label{prop:eulerobstacle} 
Consider a regular solution (in the sense of Definition~\ref{D:1}) $h$ 
to the Hele-Shaw equation defined over some time interval $[0,T]$ and define $u$ by \e{BDT}. 
Then, for all time $t\in (0,T]$, the function $u(t)$ satisfies
\be\label{n60}
u(t)\ge 0 \quad\text{and}\quad
\{x\in \xR^N\sep u(t,x)>0\} = \Omega(t).
\ee
Moreover $u(t)$ belongs to $C^1(\xR^N)$ and satisfies, in the sense of distributions,
\be\label{n85}
\left\{
\begin{aligned}
&\Delta u(t)=0 \quad &&\text{in}\quad \Omega(0)\cup \big(\xR^N\setminus \Omega(t)\big),\\
&\Delta u(t)=1 \quad &&\text{in}\quad \Omega(t)\setminus \Omega(0).
\end{aligned}
\right.
\ee
Finally, $u$ satisfies the boundary conditions
\be\label{n86}
\left\{
\begin{aligned}
&\nabla u(t)=0 \quad \text{on}\quad\partial\Omega(t),\\
&\lim_{s\to -\infty}\nabla u(t,x+se_N) = -t e_N.
\end{aligned}
\right.
\ee
\end{proposition}
\begin{proof}
Firstly, as previously mentioned (see Remark~\ref{R:4}), the classical 
maximum principle for harmonic functions implies that $p\ge 0$ and $p>0$ 
in $\Omega(t)$. Consequently, property~\e{n60} arises from the fact
that $t\mapsto\Omega(t)$ is increasing (see~\e{n36.5}).

Now, we proceed to derive \e{n85}. It is
important to note that we will prove a more general result
later (see Proposition~\ref{lem:subsup}). For this reason, we 
will only perform formal calculations at this stage and refer 
to Gustafson~\cite{Gustafsson2} for their rigorous justification 
(see also \cite{MR611303,MR2072944}). 

We begin by recalling that the kinematic boundary condition (see~\e{kc}) implies that,  
for any smooth function $\phi=\phi(t,x)$ compactly supported in 
$\xR\times \xR^N$, one has
\be\label{n134b}
\fract \int_{\Omega(t)} \phi(t,x)\dx
=\int_{\Omega(t)} (\partial_t +\nabla p\cdot \nabla)\phi \dx.
\ee
Indeed,
\begin{align*}
\int_{\xR^{N-1}} ( \phi \partial_t f)\arrowvert_{x_N=f}\dx' 
&=\int_{\xR^{N-1}} ( \phi\partial_n p)\arrowvert_{x_N=f} 
\sqrt{1+|\nabla_{x'}f|^2}\dx'\\
&=\int_{\partial\Omega(t)}n \cdot (\phi\nabla p) \dmH =\int_{\Omega(t)} \cn(\phi\nabla p)\dx\\
&=\int_{\Omega(t)}\nabla \phi\cdot \nabla p \dx.
\end{align*}

Now, let us 
introduce the function $\rho=\rho(t,x)$ defined as follows:
\begin{equation*}
\rho(t,x)=\chi_{\Omega(t)}(x).
\end{equation*}
Here, $\chi_{\Omega(t)}$ represents the indicator function of $\Omega(t)$. 
Since $p=0$ on $\mathbb{R}^N\setminus \Omega(t)$, the identity~\e{n134b} implies that the Hele-Shaw problem can be expressed as an evolution equation of the form:
\begin{equation}\label{n42}
\partial_t \rho - \Delta p = 0.
\end{equation}
As a result, directly from \e{n42} and the definition of $u$, we arrive at:
\begin{equation*}
\partial_t(\rho - \Delta u) = 0.
\end{equation*}
Since $u(0,x)=0$ by definition, integrating the previous equation in time leads to:
\begin{equation}\label{n35}
\Delta u(t) = \rho(t) - \rho(0) = \chi_{\Omega(t)\setminus\Omega(0)},
\end{equation}
which is equivalent to the desired result~\e{n85}.

Once equation~\e{n35} is derived, the rest of 
the argument is straightforward. Since the 
left-hand side is a bounded function, elliptic regularity shows that at any given time $t\in [0,T]$, $u(t)$ belongs to $C^{1,\nu}(\mathbb{R}^N)$ for any $\nu < 1$. 
In particular, $\nabla u(t)$ is continuous. Now, we know that $p(t,x)=0$ for all $x\in \mathbb{R}^N\setminus \Omega(t)$, and thus $u(t,x)=0$ for all $x\in \mathbb{R}^N\setminus \Omega(t)$. 
Consequently, $\nabla u(t)=0$ on $\mathbb{R}^N\setminus \Omega(t)$ and, by continuity, $\nabla u(t)=0$ on $\partial\Omega(t)$. Finally the limit in \eqref{n86} is a consequence of the construction. 
\end{proof}

\subsection{Definition of solutions to the Hele-Shaw problem}\label{S:3.3}

%As already mentioned in the introduction, global weak solutions to the Hele-Shaw equations have been 
%studied previously, we refer to the work of Antontsev, Meirmanov, and Yurinsky \cite{MR2072944}, Perthame, Quir\'os, and V\'azquez \cite{MR3162474} and David and Perthame~\cite{MR4324293}. They regarded the 
%Hele-Shaw equation as an evolution equation for $\chi_{\Omega(t)}$, as described 
%in~\e{n42}, and established the existence of solutions by taking limits in a 
%sequence of regularized equations of porous medium type. Viscosity solutions, 
%on the other hand, have been the focus of studies conducted by Kim~\cite{Kim-ARMA2003}, Chang-Lara, Guillen, and Schwab \cite{ChangLaraGuillenSchwab}, as well as Dong, Gancedo and Nguyen~\cite{MR4655356}. Related problems have been studied by 
%Maury, Roudneff-Chupin and Santambrogio~\cite{MR2735914}. Connections between an 
%obstacle problem for the pressure 
%and these solutions have already been studied by 
%Guillen, Kim and Mellet~\cite{MR4367913}, and Kim, Mellet and Wu %\cite{kim2023density}. 
%In this paper, we will rely on the obstacle problem formulation. Alternatively one could build a theory for the periodic case around the approach by  Crandall  and  Liggett  \cite{MR0287357}.}

In this section, we revisit the concept of a solution to introduce one that will facilitate our use of comparison arguments later on. Additionally, for our future needs, we find it necessary to consider 
local (in space) solutions.

\begin{definition}Let $U \subset \xR^N$ 
be an open and connected set. 
Consider a nonnegative function 
$p \in C([0,T)\times U)$ such that 
$p(t) \in H^1_{\loc}(U)$ for all $t\in [0,T)$. We say that $p$ is 
the pressure of a solution to the Hele-Shaw problem if 
the two following properties hold:
\begin{enumerate}
\item For all $0<s<t<T$,
$$
p(s,x)>0  \quad\Rightarrow\quad  p(t,x)>0. 
$$
\item For all $t\in (0,T)$, the function
\[
u(t,x) = \int_0^t p(s,x) \ds,
\]
is a weak solution (i.e.\ in the sense of distributions) to the obstacle problem
\[
\Delta u = \chi_{\{u>0\} \cap A} \quad\text{where}\quad 
A=U\backslash \{ p(0,x)>0 \}.
\]
\end{enumerate}
\end{definition}

We also need a convenient notion of subsolutions 
and supersolutions.

\begin{definition}
Let $U \subset \xR^N$ 
be an open and connected set 
and consider a Lipschitz function $p\colon [0,T)\times {U}\to [0,+\infty)$ and write $\{ p>0\}=\{(t,x)\colon p(t,x)>0\}$. Suppose that, for all $0<s<t<T$,
\be\label{co:3}
p(s,x)>0  \quad\Rightarrow\quad  p(t,x)>0. 
\ee
Assume that $p\in C^2(\{ p>0\})$ and that $\nabla_{t,x}p\in C\big(\overline{\{ p>0\}}\big)$. 
We say that $p$ is a classical supersolution if
\begin{alignat*}{2}
&-\Delta p \ge 0 \quad &&\text{ in } \{ p>0\} \\
&\partial_tp \ge |\nabla p|^2 \quad &&\text{ on } \partial\{p>0\}\cap (0,T)\times {U}.
\end{alignat*}
We say that $p$ is a classical subsolution if there exists $C >1 $ so that 
\begin{alignat*}{2}
&-\Delta p \le 0 \quad &&\text{ in } \{ p>0\} \\
&C^{-1} |\nabla p |\le   \partial_tp \le |\nabla p|^2 \quad &&\text{ on } \partial\{p>0\}\cap (0,T)\times {U}.
\end{alignat*}
\end{definition}
\begin{remark}
One comment is in order about the difference in the boundary conditions. 
The extra lower bound $C^{-1} |\nabla p |\le   \partial_tp$ is used below at the 
end of the proof of Proposition~\ref{lem:subsup}: 
It ensures that the boundary of $\{(t,x): p(t,x)>0\} $ 
is a Lipschitz graph of a function $g(x)$, which in turn allows to 
apply the coarea formula. For the supersolution we 
need to avoid that  condition to be able to deal later 
with the waiting time property (see Section~\ref{S:8}). This is why we only 
rely on  the inequality  $ \partial_t p \ge |\nabla p|^2$ and allow $\partial_t p $ 
(and hence $\nabla p)$ to vanish on $\partial\{p>0\}$. Eventually, let us mention that, for 
solutions which are subgraphs of functions in $C^{1,\alpha} $, $\nabla p $ never vanishes on $\partial\{p>0\}$ (see Proposition~\ref{prop:taylor} in the appendix).
\end{remark}

\begin{proposition}\label{lem:subsup}
Let $U \subset \xR^N$ 
be an open and connected set. Let  $p^{+}$ be a supersolution and $p^{-}$ be a 
subsolution  
on $[0,T)\times U$. Assume furthermore that, for each $t$, the 
boundary $\partial\{x : p^+(t,x)>0\}$ is a Lipschitz graph: 
there exists a continuous function 
$f\in C([0,T)\times \xR^{N-1})$ which is Lipschitz in $x'$ and such that
$$
\forall t\in [0,T),\quad 
\partial\{x\in U : p^+(t,x)>0\}=\{x\in U : x_N=f(t,x')\}.
$$
Let 
\[ u^{\pm}(t,x) = \int_0^t p^\pm(s,x) \ds.\]
Then $u^{\pm}\in C([0,T)\times U)$ and for each time $t$, 
$u^{\pm}(t)$ belongs to $H^1_{\loc}(U)$. Moreover, in the distributional sense, there holds
\be\label{co:0}
- \Delta u^+(t) \ge - \chi_{\{u^+(t)>0\}\cap \{p^+(0)=0\}}
\ee
and 
\be\label{co:0b}
-\Delta u^-(t) \le -\chi_{\{u^-(t)>0\}\cap \{p^-(0)=0\}}.
\ee
\end{proposition} 
\begin{remark}
Notice that $\{u^+(t)>0\}=\{p^+(t)>0\}$.
\end{remark}
\begin{proof}
We will prove the statement for supersolutions only (the one for 
subsolutions is obtained with obvious modifications).

Consider a supersolution $p^+$ that satisfies the above-mentioned regularity 
assumptions. Take a time $t\in (0,T)$ and define $u^+(t,x)=\int_0^t p^+(s,x)\ds$.

For a nonnegative test function $\Phi\in C^\infty_0(U)$, by applying the divergence theorem, we can observe:
\begin{align*}
&\int_U u^+(t)\Delta \Phi \dx \\
&\qquad\qquad = \int_0^t\int_U p^+(s) \Delta \Phi \dx\ds
\\ &\qquad\qquad = \int_0^t \left(-\int_{U}\nabla p^+(s) \cdot\nabla \Phi \dx\right) \ds
=- \int_0^t\int_{\{p^+(s)>0\}}\nabla p^+(s)\cdot \nabla \Phi \dx \ds 
\\ &\qquad\qquad = \int_0^t \int_{\{p^+(s)>0\}} \Delta p^+(s) \Phi \dx \ds 
- \int_0^t \int_{\partial \{p^+(s)>0\} } \partial_n p^+(s) \Phi \dmH  \ds\\
&\qquad\qquad\le   \int_0^t \int_{\partial \{p^+(s)>0\} }
\la \nabla p^+(s)\ra \Phi \dmH\ds,
\end{align*}
where we have used the facts that $\Delta p^+\le 0$ and $\partial_n p^+=-\la \nabla p^+\ra$ on $\partial\{p^+>0\}$.

Consequently, the desired result~\e{co:0} will follow directly from the following claim:
\begin{equation}\label{co:1}
\int_0^t \int_{\partial \{p^+(s)>0\} } \la\nabla p^+(s)\ra \Phi
\dmH \ds\le
\int_{ \{p^+(t)>0\}\cap \{p^+(0)=0\}} \Phi \dx.
\end{equation}
To prove~\e{co:1}, we introduce the function $g\colon U\to [0,t]$, which describes the first time at which the liquid reaches the point $x$ (see Figure~\ref{Figg}). More precisely, $g(x)$ is defined as:
\[
g(x) = \max\{ 0, \min\{ t, \inf\{ s\in (0,t): p^+(s,x)>0\}\}\}.
\]
The fact that this function is well-defined follows from our assumption~\e{co:3}.

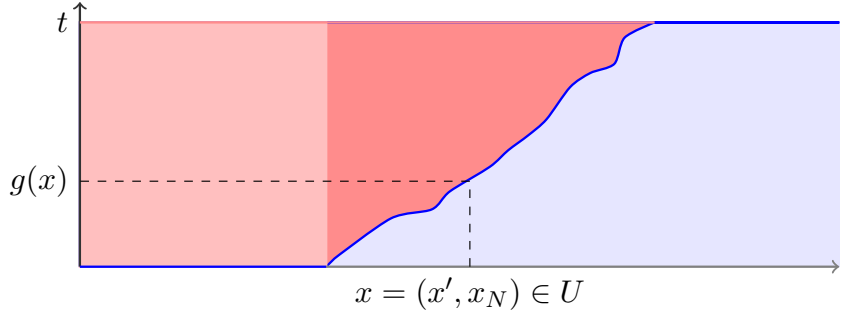
\begin{figure}[htb]
    \centering
    \resizebox{0.8\textwidth}{!}{%
\begin{tikzpicture}[x=5mm,y=5mm,decoration={mark random y steps,segment length=3mm,amplitude=1mm}]
\filldraw[fill=blue!10!white, draw=blue!10!white] (-4,1) -- (14.5,1) 
-- (14.5,-5) -- (-4,-5) -- cycle;
  \path[decorate] (2.2,-4.8) -- (2.5,-4.1) -- (3,-4) -- (4,-3.8) -- (5.5,-2.9) -- (7.5,-1) -- (10,1) ;
\filldraw [color=red!45!white,draw=blue,thick]  (-4,-5) -- (2,-5) -- plot[color=blue!60!white,variable=\x,samples at={1,...,\arabic{randymark}},smooth] 
(randymark\x) -- (14.5,1) -- (-4,1) -- (-4,-5) ;
\filldraw [color=red!25!white,thick]  (-4,-5) rectangle (2,1) ;
\draw[->,gray,thick] (2,-5) -- (14.5,-5) ;
\draw[->,color=black!80!white,thick] (-4,-5) -- (-4,1.5) ;
\draw[color=red!45!white,thick] (-4,1) node [left,color=black] {$t$} -- (10,1) ;
\draw[blue,thick] (-4,-5) -- (2,-5);
%\draw[dashed,black] (5.5,-5) node [below]{$x=(x',x_N)$} -- (5.5,-2.9) ;
\draw[dashed,black] (5.5,-2.9) -- (-4,-2.9) node [left] {$g(x)$};
\draw[dashed,black] (5.5,-5) node [below]{$x=(x',x_N)\in U$} -- (5.5,-2.9) ;
\end{tikzpicture}
}\caption{Illustrate the graph of $g$. The darkest red area represents $\{p^+>0\} \cap \{p^+(0)=0\}$, 
the lightest red area represents $\{p^+(0)>0\}$, and the blue part is the set $\{p^+=0\}$.}\label{Figg}
\end{figure}

Notice that the implicit function theorem implies that
for almost every $x$
\[ \nabla g =  -\frac{\nabla p^+}{\partial_tp^+}\bigg\vert_{(t,x)=(g(x),x)}\]
and hence $g$ is $C^1$ away from the set $ \partial_t p_+ = 0 $ on $\partial \{p >0\} $ for some $t$. 
We then conclude the proof by
using the coarea formula, first neglecting this bad set, 
$$
\int_U \varphi(x)\la \nabla g(x)\ra \dx=\int_0^t \int_{g^{-1}(\lambda)}\varphi(x)\dmH(x) d\lambda,
$$
with $\varphi(x)=\Phi(x)\la \nabla p^+(g(x),x)\ra$. This gives
\begin{multline*}
\int_0^t \int_{\partial \{p^+(s)>0\} } \la\nabla p^+(s)\ra \Phi
\dmH \ds
\\
=\int_{ \{p^+(t,\cdot)>0\}\cap \{p^+(0,\cdot)=0\}}
\frac{|\nabla p|^2}{\partial_t p}\bigg|_{(t,x)=(g(x),x)} \Phi \dx.
\end{multline*}
Consequently, the desired result \e{co:1} follows from the 
inequality
$$
|\nabla p^+|^2\le \partial_t p^+
$$
valid for any supersolution on the boundary $\partial\{p>0\}\cap (0,T)\times {U}$ and any nonnegative integrable function $\phi$, again assuming that $\partial_t p $ is bounded from below. In general this identity holds for a truncated $\Phi$, i.e. by replacing $ \Phi$ by 
\[   \chi_{\{\partial_t p(g(x),x))>\varepsilon\}}   \Phi.   \]
By monotone convergence we can take the limit $ \varepsilon\to 0 $ and obtain the identity for $\Phi$ replaced by 
\[ \chi_{\{\partial_t p(g(x),x) \ne 0\}} \Phi. \]
However $ \partial_t p(g(x),x)) = 0 $ implies $\nabla p(g(x),x)= 0 $
and both integrands vanish. 

We argue similarly for subsolutions, for which the stronger condition 
$$
|\nabla p(g(x),x)| \le C \partial_t p(g(x),x)
$$
ensures the applicability of the coarea formula. 
\end{proof}

\subsection{Sub- and supersolutions}
\label{sec:super} 
Let $U \subset \xR^N$ be open and bounded and $ A\subset U$ be relatively closed. 
We call $ u^+ \in H^1(U) $ a supersolution of the obstacle problem if 
\be\label{Nsup1}
- \Delta u^+ \ge - \chi_{\{u^+ >0\}\cap A}
\ee
and $u^- $ a subsolution of the obstacle problem if 
\be\label{Nsub1}
- \Delta u^- \le - \chi_{\{ u^- >0\}\cap A }.
\ee
\begin{lemma} \label{lem:subobstacle}
Suppose that $ u^\pm \in C(\overline\Omega)\cap H^1(U)$ is a sub- (resp.\ super-)solution such that 
$ u^- \le u^+ $ on $ \partial U$. Then 
$ u^- \le u^+$ on $ \overline{U}$.
\end{lemma} 
\begin{proof}
Set $w=u^--u^+$. Directly from \e{Nsup1} and \e{Nsub1}, we get
$$
\Delta w\ge \chi_A\big(\chi_{\{u^->0\}}-\chi_{\{u^+>0\}}\big).
$$
We can multiply the two sides of this inequality by $w_+=\max\{0,w\}\ge 0$ to get 
$$
(\Delta w)w_+\ge \big(\chi_{\{u^->0\}}-\chi_{\{u^+>0\}}\big) w_+.
$$
Now, observe that
$$
\big(\chi_{\{u^->0\}}-\chi_{\{u^+>0\}}\big) w_+\ge 0.
$$
Consequently
$$
\int_U (\Delta w)w_+\dx \ge 0
$$
which in turn implies that $\int \la \nabla w_+\ra^2\dx\le 0$, thereby proving that $w_+=0$, equivalent to $u^{-}\le u^{+}$.
\end{proof}

\section{Obstacle problem in bounded domains}
\label{sec:obstaclebounded}

In the previous chapters, we have demonstrated that the Hele-Shaw equation can be formulated as an obstacle problem on a noncompact set. As a preparation for the analysis of such problems, in this chapter, we will develop the necessary framework to consider the obstacle problem in bounded domains. More precisely, we will establish the existence and uniqueness of solutions, 
together with a stability estimate under very weak conditions. These results 
will serve as the background for constructing solutions 
in the unbounded case as limits of solutions defined on bounded domains.

We will develop a variational approach for which it is natural to work in the Sobolev spaces $H^1(U)$ or $H^1_{\loc}(U)$. Furthermore, we will consider 
non-negative functions and use the notations
$$
H^1_+(U)=
\{ u \in H^1(U) \sep u\ge 0 \},\quad 
H^1_{\loc,+}(U)=
\{ u \in H^1_{\loc}(U) \sep u\ge 0 \}.
$$
Given a function $v$, we set $v_+=\max\{v,0\}$. Recall that if $v\in H^1(U)$, then $v_+\in H^1_+(U)$. 
Similarly, if $v\in H^1_{\loc}(U)$, then $v_+\in H^1_{\loc,+}(U)$. 

Throughout  this chapter we assume that 
 $ U \subset \xR^N$ is  open and 
 bounded and consider a measurable set $ A \subset U$.

\subsection{Variational formulation}

\begin{definition}\label{def:obstacle}
We say that a non-negative function  $u \in H^1_{\loc,+}(U)$ 
is a variational solution to the obstacle problem $\Delta u=\chi_{A\cap \{u>0\}}$ if 
\begin{equation}\label{eq:obstacle}
\int_U \bigg[\Big(\frac12|\nabla v|^2+ \chi_{A} v_+\Big)
- \Big(\frac12 |\nabla u|^2 + \chi_{A}  u\Big)\bigg]\dx  \ge 0 ,
\end{equation} 
for all $v\in H^1_{\loc}(U)$ such that 
$\supp( u-v) $ is compact in $U$. 
\end{definition}
\begin{remark}Two remarks are in order concerning this formulation.

$(i)$
Notice that we assume that $u$ is non-negative, but we do not require that $v$ is non-negative. 
This explains why the first term involves $v_+$ while the second term involves $u=u_+$.

$(ii)$ Notice that it is important to write the inequality as the integral of a difference (instead of the difference of two integrals) since the two integrals in question are not well-defined in general (since we only assume that \(u\) and \(v\) are locally integrable).

\end{remark} 

We begin by proving that the equivalence of variational solutions and the Euler-Lagrange equation holds in this general setting. 

\begin{proposition}\label{lem:EulerLagrange}
Let $u \in H^1_{\loc,+}(U)$.
The following assertions are equivalent 
\begin{enumerate}
\item  $u$ is a variational solution in the sense of Definition~\ref{def:obstacle};
\item $u$  is a weak solution to the Euler-Lagrange equation 
\be\label{n27.5}
\Delta u = \chi_{A \cap \{ u>0\}},
\ee
in 
the sense that, for all 
$\phi \in C^\infty_0(U)$,
\begin{equation}\label{n27}
\int_U \Big[\nabla u \cdot\nabla \phi
+
\chi_{A\cap \{u>0\}}\phi  \Big]\dx=0.
\end{equation}
\end{enumerate} 
\end{proposition} 
\begin{proof} 
$i)$ Let $u \in H^1_{\loc,+}(U)$ be a variational solution in the sense of Definition~\ref{def:obstacle}. 
We want to prove that \e{n27} holds for any function $\phi \in C^\infty_0(U)$. Clearly, this is a local result 
and it suffices to prove that, for any ball $B\subset U$, \e{n27} holds for any 
$\phi \in C^\infty_0(B)$. Now, observe that 
the restriction $\tilde{u}=u\arrowvert_{B}\in H^1_{+}(B)$ is obviously a variational solution in the sense of Definition~\ref{def:obstacle}. Therefore, we may assume without loss of generality that $U=B$ and that 
$u\in H^1(B)$. Observe that $u$ is unique in the following sense: 
Since the functional $v\mapsto \int_B \big(\tfrac12|\nabla v|^2+ \chi_{A} v_+\big)\dx$ is 
uniformly convex, there exists a unique variational solution in $u+H^1_0(B)$. 

To get the Euler-Lagrange equation, we 
regularize the functional in question by considering 
\be\label{def:Feps}
\int_B  \Big(\frac12 |\nabla v|^2 + \chi_A  F_\varepsilon( v) \Big) \dx,
\ee
where 
\[  F_\varepsilon( v) = \left\{ \begin{array}{cl} 
0 & \text{ if }  v\le 0  \\ 
\frac1{2\varepsilon} v^2 & \text{ if } 0 < v < \varepsilon \\
v-\frac{\varepsilon}2 & \text{ if } \varepsilon \le v .
\end{array} 
\right. 
 \]
Then $F_\varepsilon$ is convex and continuously differentiable with derivative 
\[ f_\varepsilon( v) = \left\{ \begin{array}{cl} 
0 & \text{ if }  v\le 0 \\
\frac1{\varepsilon} v & \text{ if } 0 < v < \varepsilon \\
1 & \text{ if } \varepsilon \le v .
\end{array} 
\right. 
 \]
Since the functional~\e{def:Feps} is uniformly convex, classical methods in functional 
analysis guarantee that there exists a unique minimizer $u^\varepsilon$ in $u + H^1_0(B)$, 
which satisfies the Euler-Lagrange equation 
\[ -\Delta u^\varepsilon +\chi_A f_\varepsilon (u^\varepsilon) = 0. \]
This immediately implies that 
\[ \Delta u^\varepsilon \in [0,1] \]
almost everywhere. In particular 
$\Delta u^\varepsilon \in L^\infty(B)$ and classical result about 
elliptic regularity implies 
estimates for $ \Vert u^\eps \Vert_{W^{2,p}(\frac12 B)}  $ uniformly in  $\varepsilon$ for each $p<+\infty$ 
and hence $\Vert u^\varepsilon \Vert_{C^{1,\alpha}(\frac12 B)} $ is uniformly bounded for every $ \alpha \in (0,1)$. Now 
\[   v_+ - \frac{\varepsilon}{2} \le F_\varepsilon(v) \le v_+ \]
and it is not hard to see that $ u^\varepsilon \to u $ in $H^1$, hence\footnote{A more elaborate 
argument, not needed for our purpose, gives the stronger conclusion that $ u \in C^{1,1}(\tfrac12 B)$.} 
we have $u \in W^{2,p} \cap C^{1,\alpha}(\frac12 B)$. 

Consequently, for a  subsequence $u^{\varepsilon_n}$, we have  
\[ f_{\varepsilon_n}(u^{\varepsilon_n}) \to f \quad\text{weakly in } L^2\] 
for some function $f\in L^2$. 
In particular $ 0 \le f \le 1$, $ f(x) = 1 $ if $u(x) >0 $ and 
\[ \Delta u = \chi_A f \]
almost everywhere. We want to prove that 
$f=\chi_{\{u>0\}}$. 

We observe that 
\[ u(x)=0 \Longrightarrow \nabla u(x) =0 \]
since $u$ is nonnegative. We now claim that there exists a set $D \subset \{ u(x) = 0 \} $ 
of full measure so that~$f=0$ on $D$. More precisely we verify that for every $1\le i,j \le N$ 
there exists a set of full measure $ E\subset \{ u = 0 \} $ so that $ \partial^2_{ij} u = 0 $ on $E$. 
We recall $ \partial_i u \in W^{1,p}_{\loc}(\tfrac{1}{2}B)$ for some $p >N$. This is a consequence of the following Lemma.

\begin{lemma} 
Let $ p > N$ and consider a function $w \in W^{1,p}(B)$ for some ball $B$. Then 
there exists a set of full measure $E\subset B$ so that
$w$ is differentiable at every $x\in E$ and 
either $ w(x) \ne 0 $ or $\nabla  w(x) = 0  $.
\end{lemma} 
\begin{proof}
By Morrey's inequality there is a representative of $ w \in W^{1,p}(B)$ which is in $C^{1-\frac{N}p}$. In a slightly incorrect fashion we write $W^{1,p}(B) \subset C^{1-\frac{N}p}$ and we use pointwise values of $w$ below. Let $A\subset B  $ be the subset of Lebesgue points for $ \partial_j w $ for $ 1\le j \le N$. Choosing a suitable representative we may assume that for $ x\in A $
\[ \nabla w(x) = \lim_{r\to 0} \fint_{B_r(x)} \nabla w(y) dy.  \]
Let 
\[ D = \{ x \in A: w(x) = 0 , \nabla w(x) \ne 0 \}. \]
We claim that there exists a constant $C>0$ such that, for all 
$\varepsilon>0$, we have $|D| \le C \varepsilon$, which implies $|D|=0$.

To prove the claim we observe that 
for every $x\in D$ (since $w(x) = 0 $, $ w$ is differentiable at $x$  and $ \nabla w(x) \ne 0$) there exists a ball of radius $r(x) < 1-|x|$ so that 
\[ |D \cap B_r(x)| \le \varepsilon |B_r(x)|.\]

By the Vitali covering lemma there exists a disjoint sequence of such balls $B_{r_n}(x_n)$ so that 
\[ \Big|D \backslash \Big( \bigcup B_{r_n}(x_n) \Big)\Big| = 0.  \]
Thus 
\[ |D| = \sum_{n=1}^\infty |D\cap B_{r_n}(x_n)| \le \varepsilon 
\sum_{n=1}^\infty |B_{r_n}(x_n)| \le \varepsilon |B|. 
\]
This completes the proof of the lemma.
\end{proof} 

This concludes the proof of the first implication: Variational solutions satisfy the Euler-Lagrange equation.

$ii)$ It remains to prove that, if $u\in H^1_{\loc}(U)$ is a weak-solution of the 
Euler-Lagrange equation \e{n27.5}, then it is a variational 
solution in the sense of Definition~\ref{def:obstacle}. 
Again, the result is local in nature, and we can assume that $U$ is a ball $B$. 

Suppose that $u$ is a weak solution to the Euler-Lagrange equation and let 
$ v \in u + H^1_0(B)$  be a variational solution 
(as mentioned previously, such a variational solution always exists). By the first part 
\[
\Delta v = \chi_{A \cap \{ v>0 \} }\]
hence
\be\label{weak:v}
\Delta (u-v) = \chi_{A\cap \{ u>0\}} - \chi_{A\cap \{ v>0\} }
\ee
holds the sense that, for all 
$\phi \in C^\infty_0(B)$,
\begin{equation}\label{n27t}
\int_U \Big[\nabla (u-v) \cdot\nabla \phi
+(\chi_{A\cap \{ u>0\}} - \chi_{A\cap \{ v>0\} })\phi  \Big]\dx=0.
\end{equation}
By density, \e{n27t} holds for any $\phi\in H^1_0(B)$.  
Now observe that, since $u$ and $v$ are nonnegative, we have
\be\label{ineq:u-v}
(\chi_{A\cap \{ u>0\}} - \chi_{A\cap \{ v>0\} })(u-v) \ge 0,
\ee
since both factors have the same sign. 
Then, by applying \e{n27t} with $ u-v\in H^1_0(B)$, 
we conclude that
\[ \int_B |\nabla (u-v)|^2 \dx \le 0 \]
which proves that $v=u$. By definition of $v$, this means that $u$ is a variational solution. 
This concludes the proof of the second implication and hence that  
of Proposition~\ref{lem:EulerLagrange}.
\end{proof}

\subsection{Solutions to the  obstacle problem in bounded domains}
\label{subsec:bounded}

Let $U \subset \xR^N$ be a bounded domain with a Lipschitz boundary and consider a measurable subset $A \subset U$. 
We seek a non-negative solution to the obstacle problem written as 
\begin{equation} 
\Delta u =  \chi_{A \cap \{ u > 0\}},
\end{equation} 
with either the Dirichlet boundary data $u=1$ or the Neumann boundary condition 
$ \partial_n u = 1 $ ($n$ again the outer unit normal) on $\partial U$. 
We consider solutions 
in the sense of Definition \ref{def:obstacle}. 
To do this, we begin by considering the case with boundary data $u=1$. 

\begin{proposition}\label{P15}
Introduce the set
\[
\mathcal{A}=\{ v\in 1 + H^{1}_0(U)\sep v \ge 0 \}.
\]
There exists a unique function 
$u \in \mathcal{A}$ so that, for all 
$v\in \mathcal{A}$,
\[
\int_U \left(\frac12 |\nabla v|^2  + \chi_A v_+ \right)\dx 
\ge \int_U \left(\frac12 |\nabla u|^2 + \chi_A u_+ \right)\dx.
\]
\end{proposition}
\begin{proof}
Since $ \mathcal{A}$ is closed and convex and since 
the functional is uniformly convex 
and coercive, the existence of a unique minimizer follows from classical arguments (see~\cite{Figalli2018free}).
\end{proof}
We now turn to the case where $u$ satisfies the Neumann boundary condition 
$ \partial_n u = 1 $ on $\partial U$. 
In this case, we further assume that the set $A$ satisfies 
\begin{equation}\label{eq:measure}  \lvert A\rvert > \mathcal{H}^{N-1}(\partial U), 
\end{equation} 
where $\lvert A\rvert$ denotes the Lebesgue measure of $A$ and $\mathcal{H}^{N-1}(\partial U)$ refers to the $(N-1)$-dimensional Hausdorff measure of 
the boundary $\partial U$. Now, set 
\[
H^1_+(U)=
\{ u \in H^1(U) \sep u\ge 0 \}.
\]
We then search a minimizer $u\in H^1_+(U)$ to the functional 
\[
E(u)= \int_U \left(\frac12 |\nabla u |^2 + \chi_A u_+ \right)\dx 
- \int_{\partial U} u \dmH.
\]
\begin{proposition}\label{P16}
Suppose that \eqref{eq:measure} holds. 
Then $\inf_{v\in H^1_+(U)}E(v)>-\infty$ and 
there exists a unique $u\in H^1_+(U)$ 
such that $E(u)=\inf_{v\in H^1_+(U)}E(v)$.
\end{proposition} 
\begin{proof} 
We begin by proving that it is sufficient to consider functions that vanish on a set of positive measure.
\begin{lemma}
Introduce the set $\tilde{H}^1_+(U)$ of those functions 
$u\in H^1_+(U)$ such that
\begin{equation}\label{eq:boundmeasure}
\lvert\{ x\in A \sep w(x)=0 \}\rvert \ge  \lvert A\rvert - \mathcal{H}^{N-1}(\partial U).
\end{equation}
Then, for all $v\in H^1_+(U)$, there holds
\begin{equation}\label{n11}
E(v)\ge \inf_{w\in \tilde{H}^1_+(U)}E(w).
\end{equation}
\end{lemma}
\begin{proof}
Let $v\in H^1_+(U)$. If $v$ belongs to 
$\tilde{H}^1_+(U)$, then \eqref{n11} is obvious. Otherwise, there exists 
$ \varepsilon>0$ so that 
\begin{equation}\label{n12}
\lvert \{ x \in A \sep v(x) > \varepsilon\}\rvert > 
\mathcal{H}^{N-1}(\partial U).
\end{equation}
Then the function $w=(v-\varepsilon)_+=\max\{ v-\varepsilon, 0\}$ belongs to $H^1_+(U)$. Moreover, since 
$\lvert \nabla w\rvert \le \lvert \nabla v\rvert$ and since $v-(v-\varepsilon)_+\le \varepsilon$, 
it follows from \eqref{n12} that
\begin{align*}
E( w) & \le E(v)+\int \chi_A \big( (v-\varepsilon)_+-v\big)\dx 
+\int_{\partial U} (v-(v-\varepsilon)_+) \dmH(x)\\
&\le E(v) - \varepsilon \Big( \lvert\{ x \in A \sep v> \varepsilon\} \rvert 
- \mathcal{H}^{N-1}(\partial U )\Big)< E(v).
\end{align*}
Noticing that $w$ satisfies \eqref{eq:boundmeasure} by construction, we get the wanted result.
\end{proof}

For functions belonging to the set $\tilde{H}^1_+(U)$ introduced above, we have the following Poincar\'e and trace inequalities.

\begin{lemma}\label{L3}
There exists a positive constant $C$ 
such that, for all $w\in \tilde{H}^1_+(U)$, the two following inequalities hold
\begin{align}
&\int_U w^2\dx\leq C \int_U \lvert\nabla w\rvert^2\dx,\label{n13}\\
&\int_{\partial U} w^2\dmH\leq C \int_U \lvert\nabla w\rvert^2\dx.\label{n14}
\end{align}
\end{lemma}
\begin{proof}
Notice that the trace estimate \eqref{n14} 
will be a consequence of the Poincar\'e inequality \eqref{n13} 
and the classical trace inequality 
$\lA w\rA_{L^2(\partial U)}\lesssim \lA w\rA_{H^1(U)}$. 
So it is sufficient to prove \eqref{n13}. The latter inequality is a classical result and we recall its proof for the sake of completeness. If \eqref{n13} is false we can 
construct a sequence $(w_m)$ of elements of $H^1(U)$ 
such that
\begin{equation}\label{n10}
\lvert\{ x\in A \sep w_m(x)=0 \}\rvert \ge  \lvert A\rvert - \mathcal{H}^{N-1}(\partial U),
\quad \int_U w_m^2\dx=1,\quad \int_U \lvert\nabla w_m\rvert^2\dx
\rightarrow 0.
\end{equation}
Consequently, since $U$ is bounded with regular boundary, 
we can suppose that $(w_m)$ converges to $\bar{w}\in H^1(U)$, strongly in $L^2(U)$ and 
weakly in $H^1(U)$. Then $\bar{w}$ is a nonzero constant. 
Moreover
\begin{align*}
0&=\lim_{m\to +\infty}\int_U \lvert w_m-\bar{w}\rvert^2\dx
\ge \lim_{m\to +\infty}\int_{ A\cap \{w_m=0\} }  \lvert w_m-\bar{w}\rvert^2\dx\\
&\ge \lvert \bar{w}\rvert^2\inf_m \lvert \{ x\in A\sep w_m=0\}\rvert>0,
\end{align*}
hence the wanted contradiction.
\end{proof}

Now, continuing with the proof of Proposition \ref{P16}, 
it follows from \eqref{n14} and the Cauchy-Schwarz inequality that, for all $w\in \tilde{H}^1_+(U)$,
\begin{equation}\label{n15}
\left(\int_{\partial U} w\dmH\right)^2\le K
\int_{\partial U}w^2\dmH
\le C \int_{U} \la \nabla w\ra^2\dx .
\end{equation}
Hence, by definition of the functional $E$, we have 
$$
E(w)\ge \int_U \lvert\nabla w\rvert^2\dx-\left(\frac{1}{C}\int_U \lvert\nabla w\rvert^2\dx\right)^{\frac12},
$$
which immediately implies that $\inf_{w\in \tilde{H}^1_+(U)}E(w)>-\infty$. 
We are thus in position to introduce 
a sequence $(w_n)_{n\in\mathbb{N}}$ with $v_n\in \tilde{H}^1_+(U)$ such that 
$$
\lim_{n\to +\infty}E(w_n)= \inf_{w\in \tilde{H}^1_+(U)}E(w).
$$
Notice that, for any $\delta>0$, directly from the definition of $E(\cdot)$ 
and Young's inequality, we have
$$
\int_{U}\la \nabla w_m\ra^2\dx\le 
E(w_m)+\int_{\partial U}w_m\dmH
\le E(w_m)+\delta \left( \int_{\partial U}w_m\dmH\right)^2+\frac{1}{4\delta}.
$$
Therefore, in light of \eqref{n15}, we get that the sequence 
$(w_m)_{m\in\mathbb{N}}$ is bounded in $H^1(U)$ and hence 
has a subsequence that converges to a minimizer. 

It remains to prove that the minimizer is unique. To do so, suppose that $u, v \in \mathcal{A}$ are minimizers. 
Hence, for all $t\in [0,1]$, we have
\begin{align*}
&\int _U \left(\frac12 |\nabla u |^2 + \chi_A u\right) \dx - \int_{\partial U} u \dmH\\ 
&\qquad\qquad  =  \int_U\left( \frac12 |\nabla v |^2 + \chi_A v\right)\dx 
- \int_{\partial U} v \dmH\\ 
&\qquad\qquad =\int_U\left( \frac12 |\nabla (v+ t(u-v))  |^2 + \chi_A (v+ t (u-v))\right)\dx \\
&\qquad\qquad\quad- 
\int_{\partial U} v+t(u-v) \dmH,
\end{align*}
where the last equality follows from convexity and the fact that $u$ and $v$ are minimizers. The right-hand side is a constant quadratic polynomial in $t$, hence $ \Vert \nabla (u-v) \Vert_{L^2} = 0 $ so that 
$ u-v$ is constant. 
As above, we see that the constant is $0$ which proves that $u=v$.
\end{proof}

\begin{proposition}\label{prop:regularity} Let $U \subset \xR^N$ be bounded with Lipschitz boundary. 
Consider the variational solution $u$ given by either Proposition~\ref{P15} or 
Proposition~\ref{P16}. Then $u\in C^1(U)\cap C(\overline{U})$.
\end{proposition}
\begin{proof}
It follows from Lemma~\ref{lem:EulerLagrange} that $ \Delta u \in L^\infty(U)$ 
and so $u\in C^1(U)$. We only need to prove that $u\in  C(\overline{U})$.

Consider the solution $u$ given by 
Proposition~\ref{P15} for the problem with 
Dirichlet boundary condition $u=1$. 
In this case the result is well-known. Indeed, since the boundary is Lipschitz by assumption, it satisfies an exterior cone condition at any point and so 
every boundary point is regular for the Dirichlet problem (alternatively, one can apply Theorem~$10.1$ in \cite{MR1658089}). This implies that 
$u \in C(\overline{U})$. 

In the case of the solution $u$ given by 
Proposition~\ref{P16} for the problem with 
Neumann 
boundary condition $\partial_n u=1$, the continuity of $u$ on $\overline{U}$ is a consequence of Theorem 9.2 in \cite{MR1658089}: 
There exists $ q_0$ and $s_0$ so that any $q>q_0$ and $s_0<s<1$ satisfies the
conditions of Theorem 9.2.
We choose $q$ large and $s<1$ so that
$L^q_{1-s+\frac1q} \subset L^\infty $ in the notation of that paper, i.e.
\[  1-s +\frac1q > \frac{N}q. \]
Since $\frac1q -s-1 < 0 $ we have 
\[ L^\infty(U) \subset L^q_{\frac1q-s-1}(U).  \]
Then 
\[   \Vert u \Vert_{C_b} \lesssim  c \Vert u \Vert_{L^q_{1-s+\frac1q} } 
\lesssim \Vert \chi_A \Vert_{L^q_{\frac1q-s-1}}+ \Vert 1 \Vert_{L^q_s(\partial U)}
\lesssim \Vert \chi_A \Vert_{L^\infty }
+ \Vert 1 \Vert_{L^q_s(\partial U)}.
\]
This implies the desired result.
\end{proof}

The previous continuity result has the following immediate 
implication. 
\begin{corollary}\label{C:4.10}
Let $U \subset \xR^N$ be a bounded open set with Lipschitz boundary. 
Let $ A \subset U$ be relatively closed and such that 
$\overline{U} \backslash A$ is connected. Consider also a real number $t>0$. For all variational solution $u$ 
to the obstacle problem $\Delta u=\chi_{A\cap \{u>0\}}$ 
satisfying either $ u=t$ on $\partial U$ or $\partial_n u = t $ on $\partial U$ and $ A$ is closed, there holds 
\begin{equation}\label{eq:zeroset} \{x\in U:  u =0\} \subset A.  \end{equation}
\end{corollary}
\begin{proof}
Recall that $u\in C(\overline{U})$ by Proposition \ref{prop:regularity}, and so 
the set $\{u=0\}$ is closed while 
the set $\{u>0\}$ is open. The latter contains a neighborhood of the boundary. Then $u$ is nonnegative and harmonic in interior of the connected set $\overline{U} \backslash A$. It is positive there by the strong  maximum principle. 
\end{proof}

\subsection{Stability estimates}

We now prove some key inequalities to compare solutions of the obstacle problems.

\begin{theorem}[Stability estimate]\label{theorem:19}
Let $ A,B\subset U$ be two relatively closed sets so that $\overline{U} \backslash A$ and $\overline{U} \backslash B$ are connected. Let $u,v\in H^1(U)$ 
be solutions of the obstacle problems 
\begin{equation}\label{eq:obstacleA}
\begin{aligned}
&\Delta u = \chi_{\{ u>0\} \cap A} \quad \text{ in } U \qquad \text{ and }\qquad  u = t \quad \text{ on } \partial U,\\
&\Delta v = \chi_{\{ v>0\} \cap B} \quad \text{ in } U \qquad \text{ and }\qquad  v = t \quad \text{ on } \partial U.
\end{aligned}
\end{equation}

\begin{enumerate} 
\item \label{thm4.11(1)} 
Let $A_t$ resp. $B_t$ be the sets defined by 
\begin{equation}\label{eq:defA}
A_t := \{ x \in U : u(x)=0 \},\quad B_t := \{ x \in U : v(x)=0\}.
\end{equation}
Then
$$
A_t\subset A \quad\text{and}\quad B_t\subset B.
$$
Furthermore the following estimate holds for the symmetric difference:
\begin{equation} \label{eq:stability1}
| A_t \Delta B_t  | 
\le  | A \Delta B |.
\end{equation}
In addition, one has the stronger one-sided estimate 
\begin{equation} \label{eq:stability2}
| (U \backslash A_t) \cap B_t  | \le   |B \cap (U \setminus A)|  .
\end{equation} 

\item\label{item:2T19}  Assume now that 
$ A\subset B $ and that $u,v\in H^1(U)$ are solutions to~
\begin{equation}\label{N322}
\Delta u = \chi_{\{ u>0\} \cap A} \quad \text{ and }
\quad 
\Delta v = \chi_{\{ v>0\} \cap B} \quad \text{ in } U, 
\end{equation}
with $v\le u$ on $\partial U$. 
\e{eq:obstacleA}. Then $ u \ge v $. 
In particular the solution $u$ to \e{eq:obstacleA} satisfies $ 0 \le u \le t$.
\end{enumerate}
\end{theorem}
\begin{remark}
Using the divergence theorem 
\[  |A \backslash A_t| = \int_U  \Delta u_t  \dx 
= \int_{\partial U}  \partial_n u(t) \dmH.  \]
If we assume Neumann boundary conditions $ \partial_n u =t $
this implies 
\[ |A\backslash A_t| = t \mathcal{H}^{N-1}(\partial U).  \]
For the Dirichlet boundary condition we obtain a bound as follows. 
Let $u_0$ by the unique continuous function on $ \overline{U} $ which is harmonic function on 
$U \backslash A$ and 
satisfies $u_0=0$ on $A$ and $u_0=1 $ on $ \partial U$. 
Suppose that the following flux is finite: 
\[ F = \int_{\partial U} \partial_n u_0 \dmH  < \infty. \]  
By the  maximum principle we have $u_t\le u_0$ and so, $\mathcal{H}^{N-1}$ almost everywhere, we obtain
\[ \partial_n u_t \le t \partial_n u_0. \] 
This implies that
\[  |A \backslash A_t | \le t F. \] 
\end{remark}

\begin{proof}[Proof of Theorem \ref{theorem:19}] 
$i)$ We begin by proving the first point \eqref{thm4.11(1)}. The set inclusions 
$$
A_t\subset A \quad\text{and}\quad B_t\subset B.
$$
follow from Corollary~\ref{C:4.10}. Observe next that it is sufficient to prove the one sided estimate
\eqref{eq:stability2}. Indeed, the 
latter immediately implies \eqref{eq:stability1}.  
To to prove 
\eqref{eq:stability2}, we start with the following preliminary results.

\begin{lemma} Let $u,v$ be as in the statement of Theorem~\ref{theorem:19}. 
There holds
\begin{equation}\label{n25}
0\ge \int_U ( \chi_{A\cap \{ u >0 \}} 
- \chi_{B\cap \{ v >0 \} })   \chi_{\{u>v\}}  \dx.
\end{equation}
\end{lemma}
\begin{proof}
 The proof of \eqref{n25} consists 
in testing the equation satisfied by $w=u-v$ 
by $w_+/\la w\ra$. To 
make this rigorous, we use the formulation of the obstacle problem as weak solution to the Euler-Lagrange equation (see Proposition~ \ref{lem:EulerLagrange}). 

The difference $w=u-v$ is such that
\be\label{eqbis:w}
\forall \phi\in C^\infty_0(U),\quad
\int_U \Big[\nabla w \cdot\nabla \phi + \big(\chi_{A\cap \{u>0\}}
-\chi_{B\cap \{v>0\}}\big)\phi  \Big]\dx=0.
\ee
By density, the previous identity still holds for any 
$\phi\in H^1_0(U)$. We apply it with the function 
$$
w_\lambda = \frac{ w_+}{(\lambda^2+ |w|^2)^{1/2}},\quad \lambda\in (0,1].
$$
It is elementary to check that $w_\lambda\in H^1_0(U)$. 
By so doing, we obtain that
\begin{align*} 
0
& = \int  |\nabla w_+|^2 \frac{\lambda^2}{(\lambda^2+ w^2)^{\frac32}} dx 
+ \int  (\chi_{A \cap \{ u >0\} }- \chi_{B \cap \{ v>0\} }) 
\frac{w_+}{(\lambda^2+ w^2)^{\frac12} } \dx\\ 
&\ge\int  (\chi_{A \cap \{ u >0\} }- \chi_{B \cap \{ v>0\} }) 
\frac{w_+}{(\lambda^2+ w^2)^{\frac12} } \dx.
\end{align*}
Hence, by letting $\lambda$ goes to $0$, 
we see that the wanted inequality~\eqref{n25} follows from 
Lebesgue's dominated convergence theorem.
\end{proof}

We are now in position to prove 
the stability estimate~\eqref{eq:stability2}  by 
checking some elementary set inclusions. 
Firstly, 
$$
( \chi_{A\cap \{ u >0 \}} 
- \chi_{B\cap \{ v >0 \} })   \chi_{\{u>v\}}
=( \chi_A - \chi_{B\cap \{ v >0 \} })   \chi_{\{u>v\}}. 
$$
Using \eqref{n25} and \eqref{eq:zeroset} we see that
\[
\begin{split} 
0 \, & \ge \int_U ( \chi_A - \chi_{B\cap \{ v >0 \} })   \chi_{\{u>v\}}  dx 
\\ & = |\{ u>v=0 \}|  - | (U\setminus A)  \cap \{ u >v=0\} | \\
&\quad + |A \cap \{ u > v > 0\} | - |B \cap \{ u> v >0 \} | 
\\ & \ge  |\{ u>v=0\}| - |(U\setminus A)  \cap \{ u > v =0 \} | - |(U \setminus A ) \cap B \cap \{ u > v >0\} | 
\\ & \ge |\{ u>v=0\} | - |(U\setminus A) \cap B |,
\end{split} 
\]
where we have used $\{v=0\}\subset B$ to get the last inequality. 
This immediately implies the wanted result \eqref{eq:stability2}.

It remains to prove \eqref{item:2T19}. Assume that $A\subset B$ so that $\chi_A\le \chi_B$. 
Parallel to \e{ineq:u-v}, observe that, since $u$ and $v$ are nonnegative, we have
\be
(\chi_{B\cap \{ v>0\}} - \chi_{A\cap \{ u>0\} })(v-u)_+ 
\ge (\chi_{B\cap \{ v>0\}} - \chi_{B\cap \{ u>0\} })(v-u)_+\ge 0.
\ee
Then, by testing \e{eqbis:w} by $ (v-u)_+\in H^1_0(U)$, we conclude that
\[ \int_U |\nabla (v-u)_+|^2 \dx \le 0 \]
which proves that $(v-u)_+=0$, equivalent to $u\le v$. 
This concludes the proof 
of Theorem~\ref{theorem:19}.
\end{proof}

\section{The Hele-Shaw problem with initial interface in a strip}\label{S:5}

We are now prepared to present and 
prove the central result regarding 
the existence, uniqueness and stability of the solutions to the Hele-Shaw problem, as expressed in the obstacle problem formulation. 
The solutions will be 
constructed via an approximation by obstacle problems in 
bounded domains, combined 
with comparison with subsolutions 
and supersolutions. The last step of the construction is an 
argument for uniqueness. 

\subsection{Main result}

In the previous chapter, we studied the case where the ambient open set $U$ was bounded. 
Here, we consider the case $U=\xR^N$. We introduce a notion of variational solution, parallel to the one introduced in definition~\ref{def:obstacle} for the bounded case.

\begin{definition}\label{def:obstacleunbounded}
Let $ \Omega_0 \subset  \xR^N$ be an open set. Set $A=\xR^N\setminus \Omega_0$. 
We say that a non-negative function 
$u\in H^1_{\loc,+}(\xR^N)$ is a variational solution to the obstacle problem $\Delta u=\chi_{A\cap \{u>0\}}$ if 
\begin{equation}\label{eq:obstacleunbounded}
\int_{\xR^N} \bigg[ \Big(\frac12|\nabla v|^2+ \chi_{A} v_+\Big)-\Big(\frac12 |\nabla u|^2 + \chi_{A}  u\Big)\bigg]\dx\ge0,
\end{equation} 
for all $v\in H^1_{\loc}(\xR^N)$ such that 
$\supp( u-v) $ is compact in $\xR^N$. 
\end{definition}%

One reason it is interesting to consider functions which are only locally integrable is that 
our analysis automatically applies for periodic functions as well. A second observation, more fundamental, is that with this definition it is then elementary to prove that 
weak limits in $H^1_{\loc}(\xR^N)$ of variational solutions are variational solutions.

\begin{proposition}\label{P:weakl}
Consider a sequence $(v_n)_{n\in\xN}$ where $v_n\in H^1_{\loc}(\xR^N)$ is a variational solution to the obstacle problem $\Delta v_n=\chi_{A\cap \{v_n>0\}}$ in the sense 
of Definiton~\ref{def:obstacleunbounded}. Assume that $(v_n)$ converges to $v$ in $H^1_{\loc}(\xR^N)$ when $n$ goes to $+\infty$. Then $v$ is a variational solution to the obstacle problem $\Delta v=\chi_{A\cap \{v>0\}}$.
\end{proposition}
\begin{proof}
Consider a function $\varphi\in H^1(\xR^N)$ with support $\supp \varphi\subset K$ for some compact $K\subset \xR^N$. 
Since $v_n$ is a variational solution, we have
\begin{align*}
0&\le \int_{\xR^N} \bigg[ \Big(\frac12|\nabla (v_n+\varphi)|^2+ \chi_{A} (v_n+\varphi)_+\Big)-\Big(\frac12 |\nabla v_n|^2 + \chi_{A}  v_n\Big)\bigg]\dx\\
&\le\int_{K} \bigg[ \Big(\frac12|\nabla (v_n+\varphi)|^2+ \chi_{A} (v_n+\varphi)_+\Big)-\Big(\frac12 |\nabla v_n|^2 + \chi_{A}  v_n\Big)\bigg]\dx.
\end{align*}
Now we can pass to the limit and verify then that 
$$
0\le \int_{\xR^N} \bigg[ \Big(\frac12|\nabla (v+\varphi)|^2+ \chi_{A} (v+\varphi)_+\Big)-\Big(\frac12 |\nabla v|^2 
+ \chi_{A}  v\Big)\bigg]\dx,
$$
which means that $v$ is a variational solution.
\end{proof}
Again, we observe that the variational formulation is equivalent to Euler-Lagrange equations.

\begin{proposition}\label{lem:EulerLagrange2}
A function $u\in H^1_{\loc,+}(\xR^N)$ is 
a variational solution in the sense of Definition~\ref{def:obstacleunbounded} if and only if 
it is a weak solution to the Euler-Lagrange equation 
\be\label{n27.5b}
\Delta u = \chi_{A \cap \{ u>0\}},
\ee
in 
the sense that, for all 
$\phi \in C^\infty_0(\xR^N)$,
\begin{equation}\label{n27b}
\int_{\xR^N} \Big[\nabla u \cdot\nabla \phi
+
\chi_{A\cap \{u>0\}}\phi  \Big]\dx=0.
\end{equation}
\end{proposition} 
\begin{proof}
Suppose $u\in H^1_{\loc,+}(\xR^N)$ is 
a variational solution in the sense of Definition~\ref{def:obstacleunbounded}. Then the restriction $\tilde u=u\arrowvert_{B}$ of $u$ to any ball $B$ is a variational solution in the sense of 
Definiton~\ref{def:obstacle}. Consequently, Proposition~\ref{lem:EulerLagrange} implies that 
equation \e{n27b} holds for any $\phi\in C^\infty_0(B)$. 
Since this is true for any ball $B\subset \xR^N$, 
this means that equation \e{n27b} holds for any $\phi\in C^\infty_0(\xR^N)$. 

Conversely, if equation \e{n27b} holds for any $\phi\in C^\infty_0(\xR^N)$, then we can apply Proposition~\ref{lem:EulerLagrange} to deduce that 
for any ball $B\subset \xR^N$, the function $\tilde u=u\arrowvert_{B}\in H^1_{\loc}(B)$ is a variational solution to 
$\Delta \tilde u=\chi_{A\cap \{\tilde u>0\}}$ in the sense of Definition~\ref{def:obstacle}. 
This implies that $u$ is a variational solution in the sense of Definition~\ref{def:obstacleunbounded}.
\end{proof}

The one dimensional problem $N=1$ has an explicit solution
 $v=v(t,x)$ where $x\in\xR$.  It is the solution to  the problem 
\[
v_{xx}(t,x) = \chi_{\{x> 0, v>0\}}(x)
\] 
with the boundary condition 
\[
\lim_{x\to -\infty} \frac{v(t,x)}{\la x\ra} = t.
\]
The solution is the function $v_0\in C^{1,1}(\xR)$ defined by

\begin{minipage}[b]{0.50\linewidth}
\begin{equation*}
v_0(t,x): = \left\{\begin{array}{cl}  t(\frac t 2-x)  & \text{ if } x <0 \\ [1ex]
 \frac12 (t-x)^2 & \text{ if } 0 \le x \le t \\ [1ex]
 0 & \text{ if } x \ge t. 
\end{array} \right.
\end{equation*} 
\vspace{1.25cm}
\end{minipage}
\begin{minipage}[b]{0.45\linewidth}
\begin{tikzpicture}
  \draw[color=black,line width=1pt,->] (-2, 0) -- (3, 0) node[below] {$x_N$};
  \draw[color=black,line width=1pt,->] (0, -0.5) -- (0, 3) node[above] {$v_0(t,x_N)$};
  \draw[thick, domain=-1.2:0, smooth, variable=\x, blue] plot ({\x}, {1.5*(0.75-\x)});
  \draw[thick, domain=0:1.5, smooth, variable=\x, blue] plot ({\x}, {0.5*(1.5-\x)*(1.5-\x)});
  \draw[thick, domain=1.5:3, smooth, variable=\x, blue]  plot ({\x}, {0});
  \draw [color=black,line width=1pt] (1.5,-0.1)node[below] {{$t$}} -- (1.5,0.1) ;
  \draw [color=black,line width=1pt] (-0.1,1.125) -- (0.1,1.125) node[right] {{$\tfrac{1}{2}t^2$}};
\end{tikzpicture}
\end{minipage}

This function corresponds to the solution of the Hele-Shaw problem at rest, after the Baiocchi-Duvaut transformation.
Indeed, for $N\ge 1$, the solution at rest of the original Hele-Shaw problem (see Section~\ref{S:1.1}) is $(h,P)(t,x)=(0,-x_N)$, defined for $x_N<0$. 
Then, after the elementary transformations introduced in Section~\ref{S:2.3}, we get 
$(f,p)(t,x)=(t,\max\{0,-x_N+t\})$, defined for all $x\in \xR^N$ (remembering that we extended the pressure by~$0$ outside the fluid domain). 
Then one verifies that $v_0$ is given by
the Baiocchi-Duvaut change of unknown, 
indeed, for all $x_N\in \xR$, we have
\be\label{eq:v0}
v_0(t,x_N)=\int_0^t\max\{0,-x_N+\tau\}\dtau.
\ee

We are now in position to state the following theorem, which is central in our analysis. 
 
\begin{theorem} \label{thm:mainobstacle} 
Let $N\ge 2$ and set $d=N-1$. 
\begin{enumerate}
\item\label{item:th1} (Existence and uniqueness) Consider an open  connected subset $\Omega_0\subset \mathbb{R}^N$ 
such that,  
for some positive constant $C$,
\begin{equation}\label{n20}
\{x\in\mathbb{R}^N\sep x_N < -C \}\subset  \Omega_0 \subset \{ x\in\mathbb{R}^N\sep x_N <C \}.
\end{equation}
Set  $ A = \xR^N \backslash \Omega_0$. 
Then, for all $t>0$, there exists a unique variational 
solution $v=v(t)\in H^1_{\loc,+}(\xR^N)$ 
to the obstacle problem  
\[
\Delta v(t) = \chi_{\{ v>0\} \cap A} \qquad \text{(in the sense of Definition~\ref{def:obstacleunbounded})}\]
satisfying
\begin{equation*} 
\forall x\in\xR^N,\quad   v_0(t,x_N +C) \le v(t,x)  \le v_0(t,x_N-C),
\end{equation*}
where $v_0$ is given by~\e{eq:v0}. 

In particular, this solution satisfies
\begin{equation}\label{eq:thm19limit} 
\lim_{\sigma\to \infty} \sup_{x'\in \xR^d}\la \frac{ v(t,x',-\sigma e_N)}{\sigma } - t\ra=0 .
\end{equation} 

\item (Monotonicity) \label{item:monotonicity} The sets 
$\Omega_t= \{x\in\xR^N: v(t,x)>0\} $ are connected and
$$
0\le s \le t \quad \Rightarrow \quad\Omega_s\subset \Omega_t.
$$
Moreover, if $ \Omega_0^j$, $ j=1,2$ 
are two sets satisfying Assumption~\e{n20} and such that $\Omega_0^1\subset \Omega_0^2$, and $v^1(t)$ resp $v^2(t)$ are the solutions to the obstacle problem then
$v^1(t) \le v^2(t)$ and  
$\Omega_t^1\subset \Omega_t^2$ for all $t>0$.
\item\label{item:stability} (Stability) Let $ \Omega_0^j$, $ j=1,2$ 
be two sets satisfying Assumption~\e{n20} with the same constant $C$ and consider 
the corresponding solutions $v^j$ to the obstacle problem. Then 
\be\label{est:stability}
|\{ x\in \xR^N \sep v^1(t,x)>v^2(t,x)=0 \} |
\le |\{ x \in \Omega_0^1 \sep x\notin \Omega_0^2 \}|,
\ee
which is equivalent to
\[ |\Omega^1_t \cap (\mathbb{R}^N\backslash \Omega^2_t )| \le |\Omega^1_0 \cap (\mathbb{R}^N\backslash \Omega^2_0) |.\]

%\begin{equation} 
%\int_{\Omega^1_t \cap (\mathbb{R}^N\backslash \Omega^2_t)}
%(1+ |x|)^{-N-1} dx \le  C \int_{\Omega^1_0 \cap (\mathbb{R}^N\backslash \Omega^2_0)}
%(1+ |x|)^{-N-1} dx
%\end{equation} 
%for $ t \le 1$.

\item\label{item:semi-flow} (Semi-flow property) Let $0\le s \le t$, $ \Omega_s$ as above and let $ v(t,s) $  be the unique solution, as given by point \eqref{item:th1}, to 
\[ \Delta v(t,s) = \chi_{(\xR^N\backslash \Omega_s) \cap \{ v(t,s) > 0 \} }\]
which satisfies 
\begin{equation*} 
\forall x\in\xR^N,\quad   v_0(t-s,x_N +C+s) \le v(t,s,x)  \le v_0(t-s,x_N-C-s),
\end{equation*}

Then $\Omega_t = \{ v(t,s) >0 \} $ and 
$v(t,s) + v(s) = v(t)$.
\item (Time regularity) \label{item:time} 
There exists $c(N)$ so that for $ 0 \le s < t$ 
for any ball $B^{\xR^d}_R(x')$ of radius $R\ge C$
\[  \Big| | (B^{\xR^d}_R(x')\times \xR)\cap (\Omega_t \backslash \Omega_s)|- (t-s) |B^{\xR^d}_R(0)| \Big|  \le  c(N) (t-s) C^{\frac25} R^{d-\frac25}.  \]
\item\label{item:eventual} (Eventual monotonicity) 
The map 
\[  (C,\infty) \ni x_N \to v(t,x',x_N) \]
is monotonically decreasing for all $t\ge 0 $ and  $x' \in \xR^d$. If $t \ge 2C$ then there exists a lower 
semicontinuous function 
$f(t)\colon \xR^d\to \xR$ so that 
\[ \Omega_t= \{ x\in \xR^N: x_N < f(t,x') \}. \]
The function $f$ is monotonically increasing in $t$.
\end{enumerate} 
\end{theorem}
\begin{remark}
We will in addition prove eventual regularity in Section~\ref{sec:eventualreg}, thereby finishing the proof of Theorem~\ref{thm:mainobstacleintro}.
\end{remark}
\begin{remark}
The assumption~\eqref{n20} implies that the 
boundary of $\Omega_0$ is included in the 
strip $\{ x\in\mathbb{R}^N\sep \la x_N\ra\le C\}$. This automatically holds if $\Omega_0$ is  the 
subgraph of a bounded semicontinuous function $h$, i.e.
$\Omega_0 = \{ x: x_N < h(x') \} $. In other 
words, the assumption~\eqref{n20} 
allows to consider 
a general case in terms of initial topology and contains no restrictions on the regularity. The graph case is studied in \S\ref{S:maingraph}.
\end{remark}

The following corollary follows directly from point~\eqref{item:stability} in Theorem~\ref{thm:mainobstacle}. It 
states a continuous dependence with respect to the initial domain.

\begin{corollary}
Denote by $\mathcal{U}$ the set
\[
\mathcal{U}= \{ \mathcal{O} \subset \xR^N : \mathcal{O} \text{ is open and connected, satisfying \eqref{n20} for some $C>0$}\}.\]
The map
\[ [0,\infty) \times \mathcal{U} \ni (t,\Omega_0) \to \Omega_t  \in \mathcal{U} \]
is a semi-flow. It depends continuously on the initial data in the following sense.
Let $ \Omega_0^j$ $j=1,2$ be two open  connected subsets satisfying \e{n20} and, given $t>0$, set $\Omega_t^j=\{ v^j(t)>0\}$ where $v^j(t)$ is the solution given by Theorem~\ref{thm:mainobstacle}. Then, with $ A \Delta B$ the symmetric difference, 
\[ |\Omega^2_t \Delta \Omega^1_t| \le |\Omega^2_0 \Delta \Omega^1_0|.       \]
\end{corollary} 

The proof of Theorem \ref{thm:mainobstacle} will be provided in Section~\ref{S:proofthmmain}.
Before demonstrating this theorem, we take a moment in Section~\ref{S:contraction} to establish two crucial contraction estimates, which concern the difference between two solutions to the obstacle problem.
Finally, in Section~\ref{S:maingraph}, we explore the various implications of Theorem \ref{thm:mainobstacle} when we restrict the domains to subgraphs.

\subsection{A Moser type  estimate}\label{S:contraction}
In the next result we consider differences 
of solutions to the obstacle problem. 
An important observation is that we will 
prove two local estimates for the difference of two solutions in a given ball, 
without prescribing any boundary conditions. 
Its proof relies on a Moser's iteration argument. 

\begin{lemma} \label{thm:thm3} 
Suppose that $u^j\in H^1_{\loc,+}(\xR^N)$, $j=1,2$, are 
solutions 
to the obstacle problem in the sense of Definition \ref{def:obstacleunbounded}. 
Then there exists a constant $c>0$ such that, for any radius $R>0$ and any ball $B_R\subset\xR^N$, the 
two following inequalities hold:
\begin{equation}\label{eq:Moser}
\Vert  u^2-u^1 \Vert_{L^\infty(B_{R/2})} 
\le c R^{-N} \Vert u^2- u^1 \Vert_{L^1(B_R)},
\end{equation}
and 
\begin{equation}\label{eq:Cacciopoli}  \Vert \nabla (u^2-u^1) \Vert_{L^2(B_{R/2})} \le c R^{-1} \Vert u^2-u^1 \Vert_{L^2(B_R)}.
\end{equation}
\end{lemma} 
\begin{proof}
If $u^j\in H^1_{\loc,+}(\xR^N)$, $j=1,2$, are 
solutions 
to the obstacle problem, then Proposition~\ref
{lem:EulerLagrange2} implies that they solve the Euler-Lagrange equation 
\[
-\Delta u^j + \chi_{A\cap\{ u^j>0\} } = 0
\]
in the weak sense. Hence, with $ u = u^2- u^1$
\be\label{weak:u} - \Delta u +  \chi_{A\cap\{ u^2>0\} }-  \chi_{A\cap\{ u^1>0\} }=0.
\ee

Let $\eta\in C^\infty_0(B_R)$ be a nonnegative 
compactly supported test function, with $\eta=1$ on $B_{R/2}$ and such that $\lA \nabla\eta\rA_{L^\infty}\le 2/R$. Consider a real number $ p > 1$ and introduce
\[ U = \eta^2 |u^2-u^1|^{p-2} (u^2-u^1). \]
Notice that this function does not belong in general to $H^1_0(B_R)$. However, by using classical regularization arguments  
one can proceed as if $U\in H^1_0(B_R)$.

The key property is that, parallel to~\e{ineq:u-v}, we have
$$
(\chi_{\{u^2>0\}}-\chi_{\{u^1>0\}})U\ge 0.
$$
Hence, multiplying \e{weak:u} by $U$ and integrating by parts, 
\[ \int \nabla (u^2-u^1)\cdot \nabla (\eta^2 |u^2-u_1|^{p-2} (u^2-u^1)) \dx  \le 0. \] 
We write $ u = u^2-u^1$ to shorten the notation and obtain 
\[  
\begin{split} 
0 \, & \ge  (p-1) \int   \eta^2 |u|^{p-2} \la\nabla u\ra^2 \dx 
+ 2 \int \eta |u|^{p-2} u \nabla \eta\cdot \nabla u \dx  
\\ &= \frac{4(p-1)}{p^2} \int \eta^2 |\nabla |u|^{p/2}|^2  \dx 
+ \frac4p \int (\nabla \eta)  |u|^{p/2} \eta \nabla |u|^{p/2} \dx 
\\ &= \frac{4(p-1)}{p^2} \int |\nabla ( \eta |u|^{p/2} )|^2 \dx 
-\frac{4}{p^2}  \int |\nabla \eta|^2 |u|^{p}   \dx 
\\
&\quad -  \frac{4(p-2)}{p^2} \int (\nabla \eta) |u|^{p/2} \nabla (\eta  |u|^{p/2})\dx.   
\end{split} 
\]
This  implies (using the Cauchy-Schwarz inequality)     
\begin{equation}
\label{N312}
\int  \Big|\nabla \big[\eta |u_2-u_1|^{p/2}\big] \Big|^2 \dx  \le  
  \frac{p^2-2p+2}{(p-1)^2}  \int |\nabla \eta|^2  |u_2-u_1|^{p}\dx,
\end{equation}
which is the starting point for the Moser iteration (see the original paper of Moser \cite{MR0159138}) 
 which yields 
\begin{equation}\label{eq:Moser2}
\Vert u^2-u^1 \Vert_{L^\infty(B_{R/2})}\le c R^{-N} \Vert u^2-u^1 \Vert_{L^1(B_{R})}.
\end{equation}
Notice that the dependence in $R$ can be obtained by a scaling argument: 
replacing $u^j(x)$ by $u^j(x/R)$, we see that  
if \e{eq:Moser2} holds for $R=1$, then it holds for all $R>0$.

Eventually, the Caccioppoli inequality~\e{eq:Cacciopoli} follows by choosing $p=2$ above (here, the dependence in 
$R$ follows directly from the assumption 
$\lA \nabla\eta\rA_{L^\infty}\le 2/R$). 
\end{proof} 

\subsection{Proof of the theorem}\label{S:proofthmmain}

The proof of Theorem~\ref{thm:mainobstacle} is done in six steps.

\medskip 

\noindent \textbf{First step: Existence of a variational solution.}

\medskip 

Consider an open  connected subset $\Omega_0\subset \mathbb{R}^N$ 
such that,  
for some constant~$C>0$,
\begin{equation*}
\{x\in\mathbb{R}^N\sep x_N < -C \}\subset  \Omega_0 \subset \{ x\in\mathbb{R}^N\sep x_N <C \}.
\end{equation*}
In Section~\ref{sec:obstaclebounded}, we 
studied the obstacle problem under the 
assumption that the ambient 
set $U$ is bounded, and we will 
reduce the analysis from 
the unbounded case to 
the bounded one. To do this, we need to 
truncate the domain both horizontally and vertically. Horizontal truncation requires a few precautions in order to obtain a connected domain. In fact, given a parameter $R>0$ used to truncate horizontally, let us observe that the sets
$$
\{x\in \Omega_0 :\max_{1\le j \le N-1}\la x_j\ra<R\} 
$$
are not connected in general (see Figure~\ref{Fig:trunc}). To overcome this difficulty, we add artificial vertical layers at the edges $\{ x_j=\pm R\}$. Specifically, we set  (see Figure~\ref{Fig:trunc})
\[
\Omega_{0,R} =\Big(\Omega_0
\cup \Big\{(x',x_N) :  \max_{1\le j \le N-1} \la x_j\ra > R-1, x_N < C \Big\} \Big)\cap \big\{\max_{1\le j \le N-1}\la x_j\ra<R\big\}.
\]
Notice that, by construction and connectedness of $ \Omega_0$, the domain $\Omega_{0,R}$ is connected.
We then truncate vertically by fixing $L>C+t$ and keeping only those points such that $x_N>-L$. Namely, we set
$$
\Omega_{0,L,R}=\{ x \in \Omega_{0,R}: x_N>-L\}.
$$
Eventually, we periodize the domain in $x'$ to avoid boundary conditions at the fictitious walls $\{x_j=\pm R\}$. This means that we parametrize 
$ 2R\mathbb{S}^1$ by $(-R,R)$. Then
\[
(R\mathbb{S}^1)^{N-1} \times (-L,-C) \subset \Omega_{0,L,R}  \subset (R\mathbb{S}^1)^{N-1} \times (-L, C).
\]

In the previous Section~\ref{sec:obstaclebounded}, 
we studied the obstacle problem under the assumption 
that the set $U$ is bounded, and it is clear that 
the results proved in that chapter remain true in 
the periodic setting, i.e.\  
when $U=(R\mathbb{S}^1)^{N-1} \times (-L,L) $. Then, for $t>0$, and $L > C+t $
there exists\footnote{This follows from an 
elementary modification of Proposition~\ref{P15} 
(replacing the set $\mathcal{A}$ by $\mathcal{A}'=\{v \in H^1(U) : v\ge 0,~v\arrowvert_{x_N=-L}=tL,~u\arrowvert_{x_N=L}=0\}$).} 
a unique variational solution $v_{L,R}(t)$ 
to the obstacle problem
\be\label{N340}
\Delta v_{L,R}=\chi_{A_{L,R}\cap \{v>0\}}\quad\text{with}\quad
A_{L,R}=U\setminus \Omega_{0,L,R},
\ee
(in the sense of Definition~\ref{def:obstacle}) 
satisfying the following boundary conditions:
\be\label{N321}
v_{L,R}(t,x', -L) = t\Big(\frac{t}{2}+L\Big) , 
\qquad v_{L,R} (t,x', L) =0 \quad\text{for }x' \in (\mathbb{S}^1)^{N-1}.
\ee
Now consider the functions $\zeta_1(t,x)=v_0(t,x_N+C)$ and $
\zeta_2(t,x)=v_0(t,x_N-C)$ (recall the definition of $v_0$ in \eqref{eq:v0}), which satisfy
\begin{align*}
&\Delta \zeta_j=\chi_{A_j\cap \{\zeta_1>0\}} ,\quad 
A_1=\{x_N>-C\},~A_2=\{x_N>C\},\\
&\zeta_1(t,x',-L)=t\Big(\frac{t}{2}+L-C\Big),\quad  \zeta_2(t,x',-L)=t\Big(\frac{t}{2}+L+C\Big),\\
&\zeta_1(t,x',L)=0=\zeta_2(t,x',-L).
\end{align*}
Notice that these functions are obviously periodic in $x'$ since they do not depend on $x'$. 
Then, by construction, we have $\zeta_1\le v_{L,R}\le \zeta_2$ on 
$\partial U$ and $A_2\subset A\subset A_1$. Therefore, 
the comparison argument (see the point~\ref{item:2T19} 
in Theorem~\ref{theorem:19}) shows that
\begin{equation}\label{eq:compare}
v_0(t,x_N+C) \le v_{L,R}(t,x',x_N)  \le v_0(t,x_N-C). \end{equation} 

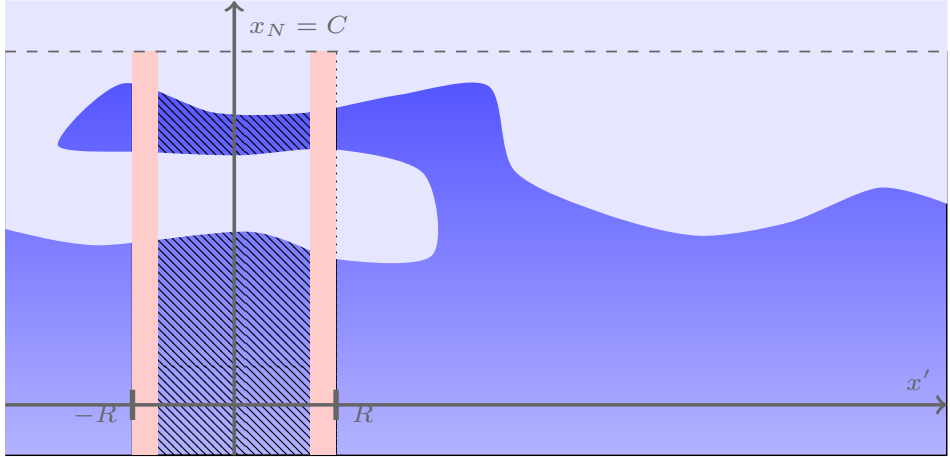
\begin{figure}[htb]
\centering
\resizebox{0.9\textwidth}{!}{%
\begin{tikzpicture}[x=5mm,y=5mm,decoration={mark random y steps,segment length=9mm,amplitude=2mm}]
\shadedraw [top color=blue!70!white,bottom color=blue!30!white]  
(-4,3) -- (14.5,3) -- (14.5,-5) -- (-4,-5) ;
\fill[pattern=north west lines,opacity=.6,draw] (-1.5,-5) rectangle (2.5,3);
\path[decorate] (-4, -0.5) -- (-1,-0.5) -- (6,-1) -- (6,0.5) -- (-4,1) -- (-3.5,2.2) -- (7,2.2) -- (6,-0.5) -- (14.5, 0);
\filldraw [color=blue!10!white]  plot[color=blue!70!white,variable=\x,samples at={1,...,\arabic{randymark}},smooth] 
   (randymark\x) -- (14.5,4) -- (-4,4) ;
\fill[pattern=north east lines,opacity=.6] (-1.5,-5) rectangle (-1,3);
\fill[pattern=north east lines,opacity=.6,white,fill=red!20!white] (-1,-5) -- (-1,3) -- (-1.5,3) -- (-1.5,-5) ; 
\fill[pattern=north east lines,opacity=.6] (2,-5) rectangle (2.5,3);
\fill[pattern=north east lines,opacity=.6,white,fill=red!20!white]  (2,-5) -- (2,3) -- (2.5,3) -- (2.5,-5); 
\draw [color=black!60!white,line width=1pt,->] (-4,-4)  -- (14.5,-4) node [above left] {\tiny{$x'$}} ;
\draw [color=black!60!white,line width=1pt,->] (0.5,-5)  -- (0.5,4) node [below right] {\tiny{$x_N=C$}} ;
\draw [dashed,color=black!60!white,line width=0.5pt] (-4,3)  -- (14.5,3)  ;
\draw [color=black!60!white,line width=1.5pt] (2.5,-4.3) -- (2.5,-3.7) node[below right] {\tiny{$R$}};
\draw [color=black!60!white,line width=1.5pt] (-1.5,-4.3) -- (-1.5,-3.7) node[below left] {\tiny{$-R$}};
\end{tikzpicture}
}
\caption{The domain $\Omega_{0,L,R}$ is connected}\label{Fig:trunc}
\end{figure}

We are now in a position to prove the existence of a variational solution by passing to the limit when $R$ and $L$ go to $+\infty$. 
To do this, consider a point $ x_0 \in \xR^N$ and observe that, 
if $R$ and $L$ are sufficiently large, then $v_{L,R}$ is well defined and uniformly bounded on the ball $B_3(x_0)$ and, in addition, $\chi_{A_{L,R}}=\chi_A$ on $B_3(x_0)$. 
Now, since $v_{L,R}$ satisfies the Euler-Lagrange equation~\e{N340}, we have $\Delta v_{L,R}\in [0,1]$, and hence $v_{L,R}$ is bounded in $W^{2,p}(B_2(x_0))$. 
Consequently, one can extract a subsequence that converges in $H^1(B_1(x_0))$. 
Using a diagonal argument, one can find a subsequence that converges in $H^1_{\loc}(\xR^N)$ 
to a function $v\in H^1_{\loc}(\xR^N)$ 
when $R,L$ goes to $+\infty$. Now, Proposition~\ref{P:weakl} implies that
$v$ is a variational solution to the obstacle problem or, equivalently, by Proposition \ref{lem:EulerLagrange2} a weak solution to
the Euler-Lagrange equation $\Delta v=\chi_{A\cap \{v>0\}}$. 
Furthermore, by passing to the limit in \e{eq:compare}, we see that $v$ inherits the bounds
\begin{equation} \label{eq:compare2}
v_0(t,x_N+C) \le v(t,x) \le v_0(t,x_N-C).
\end{equation}
This completes the proof of the existence of a variational solution satisfying~\e{eq:compare2}. 
The limit \eqref{eq:thm19limit} follows from \eqref{eq:compare2}. 

\medskip 

\noindent \textbf{Second step: uniqueness of the varational solution. }

\medskip 

We now prove that variational solutions to the obstacle problem that satisfy \eqref{eq:compare2} are unique. 

Consider 
two solutions  $ v^j$, $j=1,2$ 
to the obstacle problem satisfying \eqref{eq:compare2} and set $V=v^2-v^1$. 
By definition of $v_0$ 
(see~\eqref{eq:v0}), we have
\be\label{N301}
\sup_{x\in\xR^N}\la V(x)\ra\le 
\sup_{x\in\xR^N}\big(v_0(t,x_N-C)-v_0(t,x_N+C)\big)\le 2Ct.
\ee
Indeed, it is easily verified that
$\lA v_0(t)'\rA_{L^\infty}\le t$. 

We want to prove that $V=0$. To do so, we will prove that the $L^\infty$-norm of $V$ on balls $B_\rho(0)$ is of size at most $\rho^{-1/2}$ for any $\rho$ large enough, which immediately implies that $V=0$ (the origin plays no role here).

The key ingredient is the Moser's estimate proved in 
Lemma \ref{thm:thm3}, which allows to write
\[    \Vert V \Vert_{L^\infty(B_\rho(0))} 
\le c \rho^{-N} \Vert V \Vert_{L^1(B_{2\rho}(0))} . 
\]
We claim that
\be\label{N305}
\Vert V \Vert_{L^1(B_{2\rho}(0))} \les \rho^{N-\mez}. 
\ee
Here and below, $A\les B$ means that $A\le KB$ for some constant $K$ depending only on parameters considered fixed, as $N,C$ and $t$.

By combining this claim with the previous inequality, we immediately obtain
$$
\lA V\rA_{L^\infty(B_\rho(0))}\les \rho^{-1/2}
$$
for all $\rho$ large enough, which obviously implies
that $V=0$. 

We now turn to the proof of the claim \e{N305}. 
To shorten notation, set $B=B_{2\rho}(0)$. 
We decompose this ball 
$B$ into three pieces: 
\begin{align*}
B_1&=\{x\in B \, ;\, x_N>C+t\},\\
B_2&=\{x\in B \, ;\, -2C-t<x_N<C+t\},\\
B_3&=\{x\in B ; x_N<-2C-t\}.
\end{align*}
The integral of $V$ over $B_1$ is simply $0$ since $V(x)$ vanishes when $x_N>C+t$ (indeed, both solutions $v^1$ and $v^2$ vanishe for $x_N>C+t$ in light of \e{eq:compare2}). 
It is also elementary to estimate the
integral of $V$ over $B_2$. Indeed, the bound~\e{N301} implies that
$$
\int_{B_2}\la V(x)\ra\dx \le 2Ct \la B_2\ra \les \la B_2\ra
\les \rho^{N-1}.
$$
It remains to estimate the integral of $V$ over $B_3$. This is the most delicate part: the proof will be based on the Caccioppoli inequality proved in Lemma~\ref{thm:thm3}. 
We will apply this inequality to obtain some explicit decay rate for $V$ when $-x_N$ is large enough.

\begin{lemma}
There exists 
a constant $K$ such that, for all $x\in\xR^N$,
\be\label{N304}
x_N < -2C-t\quad\Rightarrow\quad |V(x)| \le K (x_N +2C+t)^{-\frac12}.
\ee
\end{lemma}
\begin{proof}
We start by applying inequality \eqref{eq:Cacciopoli} to get that, for all $x\in\xR^N$ and all $R>0$, 
\be\label{N311}
\begin{aligned}
\Vert \nabla V \Vert_{L^2(B_R(x))} 
&\le c R^{-1} \Vert V \Vert_{L^2(B_{2R}(x))} \\
&\le 
2^{\frac{N}{2}}c R^{-1+ \frac{N}2}\Vert V \Vert_{L^\infty(B_{2R}(x))}  \\
&\le 2^{\frac{N}{2}+1}cCt R^{\frac{N}2-1} \qquad \text{by }\e{N301}. 
\end{aligned}
\ee

Now assume that $ R \ge C+t $ and consider
$x'\in \xR^d$ with $d=N-1$. 
Since $V(x)=0$ for $x_N\ge C+t$, as already mentioned, 
we have
$$
\Vert V(\cdot, -C) \Vert_{L^2(B_R^{\xR^d}(x'))} 
\, = \Big\Vert \int_{-C}^{C+t}  V_{x_N} \dx_N \Big\Vert_{L^2\big(B_R^{\xR^{d}}(x')\big)}.  
$$
Therefore, the inequality~\e{N311} implies that
\be\label{N303} 
\begin{aligned}
\Vert V(\cdot, -C) \Vert_{L^2(B_R^{\xR^d}(x'))} 
& \le   (2C+t)^{\frac12} \Vert \nabla V \Vert_{L^2\big(B_{2R}^{\xR^N}(x',-C)\big)}
\\ & \le   2(2C+t)^{\frac12}Ct(2R)^{-1+ \frac{N}2}. 
\\ &\les R^{-1+ \frac{N}2}.
\end{aligned} 
\ee
Consider a point $x=(x',x_N)$ with 
$ x_N = -C-r$ for some $r>0$. 
Since $V$ is harmonic in the lower half-space $\{x_N<-C\}$, the
Poisson formula reads
\[  V(x) = \frac{\Gamma(N/2)}{\pi^{N/2}}  \int_{\xR^d}  \frac{r}{(r^2+|x'-y
'|^2)^{N/2}}     V(y',-C) \dy'. \]
Now we can decompose the integral as the sum of 
integrals over cubes whose edges are of
size $R$. By applying the Cauchy-Schwarz inequality on each such cube and then summing up, we obtain the bound
\[ \Vert V|_{x_N = -C-r} \Vert_{L^\infty(\xR^d)} \lesssim 
r^{-\frac{d}{2} }\sup_{x' \in \xR^d} \Vert V (\cdot,-C) \Vert_{L^2(B^{\xR^d}_r(x'))}.
\]
Assume now that $r>C+t$. Then
we can apply the bound \e{N303} with $R=r$ to infer that
\[ \Vert V|_{x_N = -C-r} \Vert_{L^\infty(\xR^d)}  \les r^{-\frac12}.
\]
This implies the desired result~\e{N304}. 
\end{proof}
The previous decay estimate~\e{N304} implies that the integral of $V$ over $B_3$ satisfies
$$
\int_{B_3}\la V(x)\ra\dx \les \rho^{N-1/2}.
$$
This completes the proof of the claim~\e{N305} 
and hence concludes the proof of the uniqueness of the variational solution and part \eqref{item:th1}. 

\medskip

\noindent{\textbf{Third step: Connectedness, monotonicity and stability estimate.}}

\medskip

Connectedness of the set $ \Omega_t$  is a consequence of the construction by the obstacle problem.  The monotonicity property and the stability estimate are also direct consequences of the analogous stability estimates for $v_{L,R}$ given by Theorem~\ref{theorem:19}. This proves part \eqref{item:monotonicity} and \eqref{item:stability}. 

\medskip 
\clearpage 

\noindent{\textbf {Fourth step: semi-flow property \eqref{item:semi-flow}.} 

\medskip 

Let us prove the semi-flow property, which
is statement~\eqref{item:semi-flow} 
in Theorem~\ref{thm:mainobstacle}. 
Consider two times $0<s<t$ and consider the solutions $v(t)$ and $v(s)$ as given by point~\eqref{item:th1} in Theorem~\ref{thm:mainobstacle}. Introduce the set
$\Omega_s=\{x\in \xR^N\,;\, v(s,x)>0\}$. Since $v(s)$ satisfies \begin{equation*} 
\forall x\in\xR^N,\quad   v_0(s,x_N +C) \le v(s,x)  \le v_0(s,x_N-C),
\end{equation*}
where $v_0$ is given by~\e{eq:v0}, we see that 
$\Omega_s$ satisfies 
\begin{equation}\label{N310}
\{x\in\mathbb{R}^N\sep x_N < -C+s \}\subset  \Omega_s \subset \{ x\in\mathbb{R}^N\sep x_N <C+s \},
\end{equation}
provided that $\Omega_0$ satisfies \e{n20}. Consequently, $\Omega_s$ satisfies the assumption of point~\ref{item:th1} in Theorem~\ref{thm:mainobstacle} (with $C$ replaced by $C+s$) and we can define a function $v(t,s)\in H^1_{\loc,+}(\xR^N)$ as the unique variational solution to the obstacle problem
\[ 
\Delta v(t,s) = \chi_{\xR^N \backslash \Omega_s \cap \{ v(t,s) >0 \}} \]
satisfying
\begin{equation*} 
\forall x\in\xR^N,\quad   v_0(t-s,x_N +C+s) 
\le v(t,s,x)  \le v_0(t-s,x_N-C-s),
\end{equation*}
We claim that
$$
v(t)=v(s)+v(t,s).
$$
To see this, set $\zeta=v(s)+v(t,s)$. Notice that
$$
\chi_{\xR^N\setminus \Omega_0\cap \{ v(s)>0\}}+
 \chi_{\xR^N \backslash \Omega_s \cap \{ v(t,s) >0 \}}=
\chi_{\xR^N \backslash \Omega_0 \cap \{ \zeta >0 \}}.
$$
Again, the verification consists in checking the different cases: 
\begin{alignat*}{5}
&\chi_{\xR^N\setminus \Omega_0\cap \{ v(s)>0\}}(x) =1 \, &&\Rightarrow 
\, &&\chi_{\xR^N \backslash \Omega_s \cap \{ v(t,s) >0 \}}(x) =0 \quad&&\text{and}\quad  &&\chi_{\xR^N \backslash \Omega_0 \cap \{ \zeta >0 \}}(x)=1,\\
&\chi_{\xR^N \backslash \Omega_s \cap \{ v(t,s) >0 \}}(x)=1 \, 
&&\Rightarrow 
\, &&\chi_{\xR^N\setminus \Omega_0\cap \{ v(s)>0\}} (x) =0 \quad&&\text{and}\quad  &&\chi_{\xR^N \backslash \Omega_0 \cap \{ \zeta >0 \}}(x)=1,
\end{alignat*}
and
$$
\chi_{\xR^N\setminus \Omega_0\cap \{ v(s)>0\}}(x)+
 \chi_{\xR^N \backslash \Omega_s \cap \{ v(t,s) >0 \}}(x)=0 
\quad\Rightarrow\quad  \chi_{\xR^N \backslash \Omega_0 \cap \{ \zeta >0 \}}(x)=0.
$$

Therefore, with 
$ A = \xR^N \backslash \Omega_0$, the functions $v(t)$ and $\zeta$ satisfies the same equation:
$$
\Delta v(t)=\chi_{A\cap \{v(t)>0\}}\quad,\quad \Delta\zeta=\chi_{A\cap\{\zeta>0\}}.
$$
Moreover, by construction, $v(t)$ and $\zeta$ satisfies the pointwise bounds
\begin{align*}
&v_0(t,x_N +C) \le v(t,x)\le 
v_0(t,x_N-C),\\
&
v_0(s,x_N +C)+v_0(t-s,x_N +C+s) \le \zeta(x)\\
&\zeta(x)\le 
v_0(s,x_N-C)+v_0(t-s,x_N-C-s).
\end{align*}
Now, remembering that $v_0$ is given by~\e{eq:v0}, we observe that 
$\la v(t)-\zeta\ra $ is bounded on $\xR^N$. This is enough to apply the uniqueness argument explained in Step $2$ 
and hence we obtain the wanted identity $v(t)=v(s)+v(t,s)$.

\medskip 

\noindent{\textbf {Fifth step: time regularity \eqref{item:time}.}

Let $t>0$ and consider the solution
$v(t)$ given by point~\eqref{item:th1} in Theorem~\ref{thm:mainobstacle}, and recall that $\Omega_t$ denotes the positivity set
$\Omega_t=\{x ; v(t,x)>0\}$. 

We now turn to the proof of the following estimate: There exists $c(N)$ so that for $ 0 \le s < t$ 
for any $y'\in\xR^d$, any ball $B^{\xR^d}_R(y')$ of radius $R\ge C$
\be\label{N710}
\Big| | (B^{\xR^d}_R(y')\times \xR)\cap (\Omega_t \backslash \Omega_s)|- (t-s) |B^{\xR^d}_R(y')| \Big|  \le  c(N) (t-s) C^{\frac25} R^{d-\frac25},
\ee
where $C$ is the constant in~\e{n20}. By the semi-flow property (see Theorem~\ref{thm:mainobstacle} \e{item:semi-flow}) it  suffices to consider $s=0$ and $0<  t \le C $.
We begin by proving a technical 
estimate which quantifies the fact that $\nabla v(t,x)$ is approximately equal to $te_N$ when $x_N$ goes to $-\infty$. 

We recall that
\be\label{N701}
\Delta v(t) = \chi_{\Omega_t \backslash \Omega_0},
\ee
and
\begin{equation}\label{N700} 
\forall x\in\xR^N,\quad   v_0(t,x_N +C) \le v(t,x)  \le v_0(t,x_N-C),
\end{equation}
where $v_0$ is given by~\e{eq:v0}.
We have 
$\lA v_0(t)'\rA_{L^\infty}\le t$, and hence we deduce that
\[    |v(t,x) - v_0(t,x_N)| \le Ct. \]
Recall also that, for $x_N<0$, there holds
$$
v_0(t,x_N)=t\left(\frac{t}{2}-x_N\right)
$$
and hence 
$$
|v(x)+ tx_N| \le Ct+\frac{t^2}{2}
\le 2tC\quad\text{for}\quad x_N<0 ,~t\le C.
$$
Also, as already seen, the inequalities \e{N700} implies that 
\be\label{n719}
\Omega_t \backslash \Omega_0 \subset \{ x: -C < x_N < C+t \}.
\ee
Remembering \e{N701}, we see that the previous inclusion implies that $ v$ is harmonic for 
$x_N<-C$. Now we recall a classical inequality: for all $R>0$, if $u$ is harmonic in a ball $B_R(x)$, then
\be\label{Poissoncenter}
\la \nabla u(x)\ra\le \frac{N}{R}\sup_{y\in B_R(x)}\la u(y)\ra.
\ee
(Notice that, up to replacing $N$ 
by a generic constant $K$, this estimate is obvious for $R=1$ by elliptic regularity and then the dependence in $R$ is obtained by a scaling argument. The value $N$ is obtained by using the Poisson formula.) Therefore, for any $x\in \xR^N$ such that $x_N < -C$, we can use \e{Poissoncenter} with $u(x)=v(x)+tx_N$ and $R=-x_N-C$ to obtain
\begin{equation}\label{eq:nablapoisson}
|\nabla v(x) +te_N|  \le \frac{N}{-x_N-C}     \Vert v+ t x_N  \Vert_{L^\infty(B_R(x)) }
\le \frac{2CtN}{-x_N-C}. \end{equation}

With this estimate in hand, we are now in position to obtain the wanted estimate~\e{N710}. 

Consider a point $y'\in \xR^d$ and a radius 
$R\ge C$. %,\quad  r = R^{\frac23} C^{\frac13} \ge 2C .%$$
Then we introduce a cut-off function $\eta\in C^\infty_c(\xR^d)$ radial, nonnegative, supported in $B_{R+\rho}(0) $, identically $1$ on $B_R(0)$ with second derivatives bounded by $4\rho^{-2} $, with $\rho\le R$ to be chosen later. We claim that 
\be\label{N717}\Big|\int_{\Omega_t \backslash \Omega_0} \eta(y'+x')\dx  - t \int_{\xR^d} \eta(y'+x') \dx' \Big| \lesssim  t C^{\frac25} R^{-\frac25}           |B^{\xR^d}_R(0)|\les 
t C^{\frac25} R^{d-\frac25}.
\ee
and 
\be\label{n718}
\begin{split} \hspace{1cm} & \hspace{-1cm} 
\Big| | (B^{\xR^d}_R(y')\times \xR)\cap (\Omega_t \backslash \Omega_0)|- t |B^{\xR^d}_R(y')| 
\\ & - \Big(\int_{\Omega_t\backslash \Omega_0} \eta(y'+x') \dx - t \int_{\xR^d} \eta( y'+x') \dx'\Big) \Big| 
\\ & \qquad \lesssim 
t R^{d-1} \rho. 
\end{split}
\ee

Notice that \e{n718} follows directly from \e{n719} and the assumptions on~$\eta$. To prove the claim~\e{N717}, we begin by 
observing that
$$
\int_{\Omega_t \backslash \Omega_0} \eta(x')\dx 
= \int_{\xR^N} (\Delta v(x))  \eta(x') \dx.
$$
Given a parameter $r\ge C$ to be chosen later, 
set
$$
Q(r,R)=(y'+[-R-\rho,R+\rho]^d)\times [-r,C+t].
$$
Observe that 
$\eta \Delta v$ is supported in $Q(r,R)$ so   
the divergence theorem implies that
\[
\begin{split}
\int_{\Omega_t \backslash \Omega_0} \eta(x')\dx 
&=\int_{Q(r,R)} (\Delta v(x))  \eta(x') \dx\\
&=-  \int_{\xR^d}    \eta(x')\partial_N v (x',-r) \dx'   + \int_{Q(r,R)}  \Delta_{\xR^d}  \eta( x')  v \dx. 
\end{split} 
\]
Now write
\[      \int_{\xR^d}    \eta(x')\partial_N v (x',-r) \dx' 
= -t\int_{\xR^d}  \eta \dx'  +    \int_{\xR^d}    \eta(x')(t+\partial_N v (x',-r)) \dx' 
\]
and by \eqref{eq:nablapoisson}
\be\label{n721}
\int_{\xR^d}    \eta(x')|t+\partial_N v (x',-r)|  \dx' \le \frac{2tCK}{r-C}  \int_{\xR^d} \eta \dx' \le \frac{2^NtCK}{r-C}R^{N-1} |B_1^{\xR^d}(0)|.
\ee 
On the other hand, 
$$
\int_{Q(r,R)}  |\Delta_{\xR^d}  \eta( x')  v |\dx  
\le \la Q(r,R)\ra \cdot \lA \Delta_{\xR^d}  \eta\rA_{L^\infty(\xR^d)}\cdot 
\lA v\rA_{L^\infty(Q(r,R))}.
$$
Now, since $\rho\le R$ and $t\le C$, we have
$$
\la Q(r,R)\ra\les (R+\rho)^d (r+C+t)\les R^d(r+2C).
$$
Also, by assumption on $\eta$, we have
$$
\lA \Delta_{\xR^d}  \eta\rA_{L^\infty(\xR^d)}\les \rho^{-2}.
$$ 
Eventually, it follows from the bound \e{N700} and the fact that $x_N\mapsto v_0(t,x_N)$ is decreasing, that
$$
\lA v\rA_{L^\infty(Q(r,R))}\le v_0(t,-r-C)=t\left(\frac{t}{2}+r+C\right)\le t(r+2C).
$$
We conclude that 
\be\label{n722}
\int_{Q(r,R)}  |\Delta_{\xR^d}  \eta( x')  v |\dx  
\les  t R^{N-1}\rho^{-2}  (r+C)^2.
\ee
Now, collect the three inequalities \e{n719}, \e{n721} and \e{n722} to obtain 
\[
\Big| | (B^{\xR^d}_R(y')\times \xR)\cap (\Omega_t \backslash \Omega_0)|- t|B^{\xR^d}_R(y')| \Big|  \lesssim 
t R^{N-2} \Big( \rho +  \frac{CR}r  + \frac{R r^2}{\rho^2} \Big) 
\]  
We choose $ r = C^{\frac35} R^{\frac{2}{5}} \le R$ and then 
$\rho = R^{\frac13} r^{\frac23}\le R$ 
so that the bracket is bounded by a multiple of 
\[   R^{\frac13} r^{\frac23} + \frac{CR}r \lesssim C^{\frac25} R^{\frac35}.
\]
This proves the desired estimate~\e{N710}.

\medskip

\noindent{\textbf{Sixth step: eventual monotonicity \eqref{item:eventual}.}

\medskip 

Let us prove the eventual monotonicity property~\eqref{item:eventual} 
in Theorem~\ref{thm:mainobstacle}. 
Consider a time $t\ge 0$ and introduce the unique 
solution 
$v=v(t)\in H^1_{\loc,+}(\xR^N)$ 
to the obstacle problem $\Delta v=\chi_{A\cap \{v>0\}}$ satisfying
\begin{equation}\label{N313} 
\forall x\in\xR^N,\quad   v_0(t,x_N +C) \le v(x)  \le v_0(t,x_N-C).
\end{equation}
Our goal is to prove that the map
\be\label{N329}
(C,\infty) \ni x_N \mapsto v(x',x_N) \quad\text{
is monotonically decreasing for all }x' \in \xR^d.
\ee
Let us assume this result for a moment. 
As already seen, the lower bound in \e{N313} 
implies that
$$
\{x\in \xR^N\,;\, x_N \le -C+t\}\subset \Omega_t=\{x\in \xR^N\,;\, v(t,x)>0\}.
$$
Consequently, if $t\ge 2C$, then 
\be\label{N328}
\{x\in \xR^N\,;\, x_N \le C\}\subset\Omega_t.
\ee
Now, let us assume that $x=(x',x_N)\in \Omega_t$ for some $x_N>C$. 
By definition of $\Omega_t$, this means that $v(t,x)>0$. 
And since $y\mapsto v(t,x',y)$ is decreasing for $y\ge C$, we deduce that $v(t,y)>0$ for any $y\in (C,x_N)$. In particular, the segment 
$((x',C),(x',x_N))$ is contained in $\Omega_t$. By combining this with \e{N328}, 
we immediately deduce that $\Omega_t$ is a subgraph: there exists a lower 
semicontinuous function $f(t)\colon \xR^d\to \xR$ so that 
\[ \Omega_t= \{ x\in \xR^N; x_N < f(t,x') \}. \]
In addition, the function $f$ is monotonically increasing in $t$.

We now turn to the proof of \e{N329}. To prove this monotonicity result, we use the moving plane method. 
Let $x=(x',x_N)\in\xR^N$ with $x_N>C$, and consider $y_N>x_N$. We want to prove that $v(x',y_N)\le v(x',x_N)$. This is equivalent to proving that 
\be\label{N315}
v(x',2k-x_N)\le v(x',x_N)\quad\text{where}\quad k=\mez (x_N+y_N).
\ee
The advantage of using this equivalent formulation will be apparent in a moment; it has to do with the identity:
$$
\Delta (v(x',2\kappa-x_N))=
(\Delta v)(x',2\kappa-x_N).
$$ 
In particular, the claim~\e{N315} will be a direct consequence of the following lemma applied with $\kappa=k$.

\begin{proposition}
Given $\kappa > C$, introduce the function 
$u\in H^1_{\loc,+}(\xR^N)$ defined by
\be 
u(x',x_N) = v(x',2\kappa-x_N).
\ee
On the lower half-space $\Pi_\kappa=\{x\in\xR^N\,;\,x_N\le\kappa\}$ we have
$$
u(x)\le v(x)\quad\text{for all}\quad x\in \Pi_\kappa. 
$$
\end{proposition}
To prove this result, introduce the function \(w: \mathbb{R}^N \to \mathbb{R}\) defined by \(w = u - v\), as well as its nonnegative part \(w_+ = (u - v)_+\). We aim to show that
\[w_+(x) = 0 \quad \text{for all} \quad x \in \Pi_\kappa.\]

To achieve this, we follow the strategy employed to prove the uniqueness result in the second step. On the one hand, this will be simpler since $w_+$ is bounded and supported in an horizontal strip (as we will see in a moment). On the other hand, this requires careful consideration since we are now working on the half-space \(\Pi_\kappa=\{x_N\le \kappa\}\) instead of the entire space \(\mathbb{R}^N\). Let us 
start with the proof of various bounds.

\begin{lemma} \label{thm:thm3w} 
\begin{enumerate} 
\item \label{item:bound} There exists a constant $M>0$ depending only on $C$ and $t$ such that 
\be\label{N721}
\forall x\in \Pi_\kappa,\quad 0\le w_+(x)\le M.
\ee
\item \label{item:moser} Given a radius $R>0$, denote by  $B_R^\kappa\subset\xR^N$ 
the half-ball centered at $\kappa e_N=(0,\kappa)\in \xR^{N-1}\times \xR$ located under the hyperplane $\{x_N=\kappa\}$:
$$
B_R^\kappa=B_R(\kappa e_N)\cap \Pi_\kappa=\left\{ x\in B_R(\kappa e_N) \,;\, x_N\le \kappa\right\}.
$$
Then one has the following Poincar\'e inequality:
\begin{equation}\label{N729}
\Vert w_+ \Vert_{L^2(B_{R}^\kappa)} 
\le \kappa \Vert \nabla w_+ \Vert_{L^2(B_R^\kappa)},
\end{equation}
along with the following Caccioppoli inequality:
\begin{equation}\label{eq:Moser3}
\Vert \nabla w_+ \Vert_{L^2(B_{R/2}^\kappa)} 
\le 2 R^{-1} \Vert w_+ \Vert_{L^2(B_R^\kappa)}.
\end{equation}
\end{enumerate} 
\end{lemma} 
\begin{proof}
\eqref{item:bound}  Let us prove that $w_+$ is bounded. One has the obvious bound $w_+\ge 0$. On the other hand, 
since $u\ge 0$ and $v\ge 0$, we have $w_+\le u$. Now, observe that, by definition of $u$ (see~\e{N328}), $u=0$ for $x_N\le 2\kappa-C-t$ since $v(t,x',x_N)=0$ whenever $x_N\ge C+t$. It remains only to prove that $w_+$ is bounded on the strip 
$\{ 2\kappa - C-t\le x_N\le \kappa\}$. However, remembering that $\kappa>C$ by assumption, on the latter strip, 
we have $C\le 2\kappa-x_N\le C+t$ and hence 
the function $u$ is bounded there since $v$ is bounded on any horizontal strip (because it satisfies~\e{N313}).

\eqref{item:moser}  We have seen in the first part of the proof that $w_+$ vanishes on $\{x_N=\kappa\}$ since 
$u$ and $v$ coincide on $\{x_N=\kappa\}$ and is supported in the strip 
$\{ 2\kappa-C-t\le x_N \le\kappa\}$ of width $\kappa-C-t\le\kappa$. 
Hence, \e{N729} is the classical Poincar\'e inequality. 
%, for which we repeat the proof for the sake of %completeness. 
% for any $(x',y)\in \Pi_\kappa$, there holds
%$$
%w_+(x',x_N)=\int_\kappa^{x_N} \partial_{x_N}w_+(x',y)\dy.
%$$
%The Cauchy-Schwarz inequality implies that
%%$$
%\la w_+(x',x_N)\ra^2\le \la x_N-\kappa\ra 
%\int_{x_N}^\kappa (\partial_{x_N}w_+(x',y))^2\dy
%\le \la x_N-\kappa\ra 
%\int_{-\infty}^{\kappa} (\partial_{x_N}w_+(x',y))^2\dy.
%$$
%Now, integrating over $\xR^{N-1}\times [2\kappa-C-%t,\kappa]$, we see that 
%\begin{align*}
%\int_{\Pi_\kappa}\la w_+(x',x_N)\ra^2\dx' \dx_N
%&=\int_{\xR^{N-1}\times [2\kappa-C-t,\kappa]}\la w_+%(x',x_N)\ra^2\dx' \dx_N \\
%&\le \kappa^2 \int_{\Pi_\kappa}
%(\partial_{x_N}w_+(x',y))^2\dy\dx'.
%\end{align*}
%This proves~\e{N729}.

We now turn to the proof of the Caccioppoli inequality~\e{eq:Moser3}. To do this, notice that 
$$
(\Delta u)(x)=(\Delta v)(t,x',2\kappa-x_N).
$$
Now, recall that $\Delta v=\chi_{A\cap \{v>0\}}$. 
Also, since $\kappa\ge C$ and since $A=\xR^N\setminus \Omega_0$ with $\Omega_0\subset \{x_N\le C\}$, we see that 
$\chi_{A}(t,x',2\kappa-x_N)=1$ for any $x\in \Pi_\kappa$ (notice that this is where the assumption $\kappa\ge C$ enters). This implies that, in the weak sense,
$$
\Delta u =\chi_{\{u>0\}}\quad\text{in }\Pi_\kappa.
$$
Consequently, the function $w=u-v$ satisfies, also in the weak sense, 
\be\label{weak:w}
\Delta w =  \chi_{\{ u>0\} }-  \chi_{A\cap\{ v>0\}}\quad\text{in }\Pi_\kappa.
\ee

Now, let $\eta\in C^\infty_0(B_R)$ be a nonnegative 
compactly supported test function, with $\eta=1$ on $B_{R/2}$ and such that $\lA \nabla\eta\rA_{L^\infty}\le 2/R$. 
Introduce
\[ W = \eta^2 w_+. \]
Notice that $W\in H^1_0(B_R^\kappa)$ since 
$w_+=0$ on $\{x_N=\kappa\}$.

The key property is that
$$
(\chi_{\{u>0\}}-\chi_{A\cap \{v>0\}})W\ge 0.
$$
Indeed, this is obvious if $W=0$ and otherwise $u>v$ so that 
$$
0\le \chi_{\{u>0\}}-\chi_{\{v>0\}}\le \chi_{\{u>0\}}-\chi_{A\cap \{v>0\}}.
$$ 
Hence, multiplying \eqref{weak:w} by \(W\),  
%integrating by parts on \(\{x_N\le \kappa\}\), 
we conclude that
\[ \int_{\Pi_\kappa} \nabla w\cdot 
\nabla (\eta^2  w_+) \dx  \le 0. \] 
Now, observe that
\[ \int_{\Pi_\kappa} \nabla w\cdot 
\nabla (\eta^2  w_+) \dx =
\int_{\Pi_\kappa} \nabla w_+\cdot \nabla (\eta^2  w_+) \dx
\]
and hence we conclude that
\[ 
\int_{\Pi_\kappa}  \big|\nabla (\eta w_+) \big|^2 \dx  \le    
  \int_{\Pi_\kappa} |\nabla \eta|^2  |w_+|^{2}\dx.
\]
This implies the wanted inequality~\e{eq:Moser3}.
\end{proof}

Now, let us define the function $F\colon [0,+\infty)\to [0,+\infty)$ by
$$
F(R)=\lA w_+\rA_{L^2(B_R^\kappa)}.
$$
It follows from \e{N721}, \e{N729} and \e{eq:Moser3} that $F$ is a nondecreasing function satisfying 
$$
F(R)\le K R^{N/2},\quad F(R)\le \frac{K}{R}F(2R),
$$
for some fixed constant $K>0$.  As a consequence, for $ k > \frac{N}2$  
\[    F(R) \le  \Big(\frac{K}{N} \Big)^k F(2^kR) \le K^{k+1} R^{\frac{N}2-k} \to 0 \qquad \text{ as } R \to \infty. \]   
It easily follows that $F=0$, hence $w_+=0$, which completes the proof. 

\subsection{Consequences for the graph case}\label{S:maingraph}

The following corollaries state several consequences 
of  Theorem~\ref{thm:mainobstacle} in the graph case. 
\begin{corollary} \label{cor:graph}
\begin{enumerate} 
\item \label{item:A}  (Semiflow on lower semicontinuous functions) Assume that 
$\Omega_0 = \{x\in\xR^N \sep x_N < f_0( x')\}$ 
for a bounded lower semicontinuous function $f_0$. 
Then there exists a function $f$ on  $\xR_+\times\xR^{d}$ such that $x'\to f(t,x') $ is bounded and lower semicontinuous for every $t\ge 0$ and 
such that the positivity set $\Omega_t$ of the corresponding variational 
solution is given by 
\[ \Omega_t = \{ x: x_N < f(t,x') \}.\] 
\item \label{item:B}  (Contraction) 
The flow map
\[ f_0 \mapsto f(t) \]
is a contraction in the $L^p$ norm for $1 \le p \le \infty$, more precisely 
\begin{equation}\label{eq:lp}  \Vert (f^1(t)-f^2(t))_+\Vert_{L^p} \le \Vert (f^1_0 - f^2_0)_+ \Vert_{L^p} \end{equation}
where $f^j(t)$ correspond to solution with initial data $f^j_0$.  
\item \label{item:C} (Norm decay) 
In particular any modulus of continuity of the initial data $f_0$ is also a modulus of continuity for $f(t)$. 
If $f_0$ has a modulus of continuity then 
\[ \xR \times \xR^d \to f(t,x') \in \xR \]
is continuous. Moreover, for all $ 1 \le p \le \infty$, there holds
\begin{equation} \label{eq:ljapunov}  \Vert f(t) \Vert_{W^{1,p}(\xR^d)} \le \Vert f_0 \Vert_{W^{1,p}(\xR^d)}.
\end{equation}
More precisely: $\Vert f(t) \Vert_{L^p}$ and $ \Vert \partial_k f(t) \Vert_{L^p} $
are Lyapunov functionals. 
\end{enumerate} 
\end{corollary} 
\begin{remark}[Orlicz spaces and fractional regularity]\label{R:5.12}
The proof gives much more than stated: We obtain \eqref{eq:lp} for Orlicz spaces by the same argument. 
Similarly we obtain that Orlicz norms and Sobolev Orlicz norms are Lyapunov functionals and 
even fractional Orlicz Sobolev norms of Besov type with regularity between $0$ and $1$ are Lyapunov functionals. 
\end{remark}

\begin{proof}
\eqref{item:A}  Let $ \Omega_0 = \{ x_N \le f_0\}$ 
and let $\Omega_0^s =\{ x\in\xR^N \sep x_N < f(x') +s\}$ for $s>0$. 
Let $\Omega_t$ and $\Omega^s_t $ be the positivity sets of the corresponding solutions. Then, 
by Theorem \ref{thm:mainobstacle} \eqref{item:monotonicity}, we have
\[ \Omega_t \subset \Omega^s_t = \{ x: x-se_N \in \Omega_t \} \] 
and hence there exists a bounded lower semicontinuous function $f(t,x')$ so that 
\[ \Omega_t = \{ x: x_N < f(t,x') \}. \]
\eqref{item:B}  Now, consider two  bounded lower semicontinuous initial data $f^{1,2}_0$ 
and denote by the associated solutions by $ f^{1,2}(t,x)$. Let $h=\Vert f^2_0-f^1_0 \Vert_{L^\infty}$
Then 
\[  f^2_0-h \le f^1_0 \le f^2_0 +h \]
and this relation persists by Theorem \ref{thm:mainobstacle} \eqref{item:monotonicity}, that is:
\[ f^2(t,x) -h \le f^1(t,x) \le f^2(t,x)+ h, \]
thus
\[ \Vert f^2(t,.) - f^1(t,.) \Vert_{L^\infty} \le h. \]
Similarly 
\[ \Vert (f^2(t,.) - f^1(t,.))_+  \Vert_{L^1} \le  \Vert (f^2_0 -f^1_0)_+ \Vert_{L^1} \]
by Theorem \ref{thm:mainobstacle} \eqref{item:stability}. 

Let us now prove the corresponding contraction bounds in $L^p$ for $ 1< p < \infty$. To do this, we observe that $f^2+s $ is a solution if $f^2$ is a solution. Thus, for all $s \in \xR $, we have 
\[   \Vert  (f^2(t)-f^1(t)-s)_+ \Vert_{L^1(\xR^d)} \le \Vert (f^2_0-f^1_0-s)_+ \Vert_{L^1(\xR^d)}.  \]
To shorten notations, we 
set $ f(t) = f^2(t)-f^1(t)$ and 
suppress the time variable $t$. Then, 
for any $1<p<\infty$, we have
\[
\begin{split} 
p(p-1)  \int_0^\infty s^{p-2} \Vert (f-s)_+ \Vert_{L^1} \ds
\, & = p(p-1) \int_0^\infty s^{p-2} \int_{\{f>s\}} f - s\dx \ds 
\\ & = p(p-1) \bigg(\int f \int_0^f s^{p-2} ds- \int_0^f s^{p-1} \ds \dx \bigg)
\\  & =   \int f_+^p  \dx.  
\end{split} 
\]
This implies the wanted contraction estimate 
\eqref{eq:lp}.

\eqref{item:C} 
A small modification of the previous argument shows that any modulus of continuity for $f_0$ is also a modulus of continuity for $f(t)$. In particular Lipschitz constants are preserved. Now suppose that $f_0$ has the modulus of continuity $\omega=\omega(r)$, continuous and strictly monotone. Let $t>0$. The function $f(t,\cdot)$ has the same modulus of continuity.
Let $t>0$, $ x' \in \xR^d$ and $ r\le 1 $ maximal so that 
$3 \omega(r)\le f(t,x')-f_0(x')$ and $ R = 2\Vert f_0 \Vert_{L^\infty} $. 
Then 
\[    \frac13 r^d (f(t,x') -f_0(x'))  \le |(\Omega_t \backslash \Omega_0) \cap B_R^{\xR^d}(x')| \le C t \Vert f_0 \Vert^d_{L^\infty}  \]
where the second inequality follows from Theorem \ref{thm:mainobstacle} \eqref{item:time}, 
and the first one is 
a consequence of the geometry since the triangle inequality implies that
$$
\la y-x'\ra<r\quad\Rightarrow \quad\frac13  (f(t,x') -f_0(x'))  \le f(t,y)-f_0(y).
$$

Thus, 
%if $ h(t,x')-h_0(x') > \omega(1) $ 
%\[  0\le  h(t,x')-h_0(x')\le  Ct \Vert h_0 %\Vert_{L^\infty}^d   \]
%which can only be true if $t$ is sufficiently large, 
%otherwise 
\[ 0 \le (f(t,x')-f_0(x')) \big[\omega^{-1}( \tfrac13 (f(t,x')-f_0(x')))\big]^d  \le Ct \Vert f_0 \Vert_{L^\infty}^d .    \]
The function $s \mapsto  s (\omega^{-1}(s))^d$
is continuous and strictly monotone.
Hence $t \mapsto f(t,x') $ is uniformly continuous at time $t=0$, and thus at all times.

Finally, since $f=0$ is a solution, we have $\Vert f(t) \Vert_{L^p} \le \Vert f_0 \Vert_{L^p}$, 
and since the problem is invariant under translations, we deduce that, for all $ h \in \xR^d$, 
\[  \Vert f(t,\cdot+h) -f(t) \Vert_{L^p} \le \Vert f_0(\cdot.+h) - f_0 \Vert_{L^p} \]
hence 
\[ \Vert |\nabla f(t)| \Vert_{L^p} =   \sup_{h \in B_1(0) \backslash \{0\}} |h|^{-1} \Vert f(t,\cdot+h)- f \Vert_{L^p} \le \Vert |\nabla f_0| \Vert_{L^p}.     \]
This completes the proof.
\end{proof}

\section{Applications of the boundary Harnack inequality} 
\label{sec:bh} 

Let $\Omega\subset \xR^N$ be an open set 
and consider a {\em positive} harmonic function $u\in C^2(\Omega)$. 
The classical Harnack inequality 
relates the values of $u$ at two points away from the boundary. More precisely, it ensures that, for any open subset $\omega\subset \Omega$, connected and such that $\overline{\omega}\subset \Omega$, there exists a constant~$C$ such that, 
for all positive harmonic function $u\in C^2(\Omega)$, 
\[
\sup_\omega u\leqslant C \inf_\omega u.
\]
This inequality is an essential 
tool to study the interior regularity of weak solutions to an elliptic or parabolic equation. 
The boundary Harnack inequality, whose statement is recalled below, is an inequality valid up to the boundary, which plays a similar role to study boundary value problems. In this chapter, we will recall this inequality and work out two applications for the problems we consider. Namely, we will prove a uniqueness result and a regularity result.

\subsection{The boundary Harnack inequality}

Let us recall the following definition. 

\begin{definition}
Let $\Omega$ be an arbitrary domain in 
$\xR^N$ and $ \Gamma \subset \partial \Omega$. We say $\Omega$  satisfies a uniform local boundary Harnack inequality at $\Gamma$ if there exist $ L $ and $A$ such that,  for all 
$x_0 \in \Gamma$, for all 
$r>0$ 
and for all nonnegative harmonic functions $u,v\colon B_{Lr} (x_0) \cap \Omega \to (0,+\infty)$  which vanish at the boundary $\partial \Omega$, there holds
\be\label{UBHI}
\forall (x,y)\in (B_r(x_0) \cap \Omega )^2,\quad \frac{u(x)}{v(x)}  \frac{v(y)}{u(y)} \le A.
\ee
\end{definition}
\begin{remark}\label{R:5.2}
\begin{enumerate} 
\item  The study of the boundary Harnack inequality for Lipschitz domains 
goes back 
to the works of 
Kemper~\cite{Kemper-1972} and Ancona~\cite{Ancona-1978} (see also Dahlberg~\cite{Dahlberg-1977} and Wu~\cite{Wu-1978}). 
De Silva and Savin \cite{MR4093736} have provided a short proof for graph domains below a Lipschitz function (which is the only case we shall need in this paper). Aikawa \cite{MR2464701} has given fairly weak conditions on domains which imply boundary Harnack estimates, using ideas based on 
\cite{CFMS-1981,MR676988,MR676987}.

\item  The study 
of boundary Harnack inequality 
can be extended to general elliptic equations in divergence form  
see Caffarelli and Salsa \cite[Chapter 11]{CS-book}, as well as 
in  nondivergent form, see Fabes, Garofalo, Mar\'{\i}n-Malave, and Salsa~\cite{FGMMS-1988}, and De Silva and Savin for fairly general equations. One 
can also extend this notion by assuming that the boundary Harnack inequality holds only for functions which vanish 
on the boundary, 
outside of a polar set.

\item 
The boundary Harnack is frequently used to study the regularity of 
solutions to elliptic problems, in particular for the study of 'the' harmonic measure (which strictly speaking is defined for a chosen interior point, and the boundary Harnack inequality for Lipschitz domains shows that the dependence on the interior point is mild) and 
for the obstacle problem where it is used to go from Lipschitz continuity of the free boundary to smoothness, see Athanasapoulos, Caffarelli \cite{AtCa-1985} and the  books by Caffarelli and Salsa~\cite{CS-book}, 
Petrosyan, Shahgholian and Uraltseva~\cite{MR2962060},  
as well as Figalli~\cite{Figalli-ICM2018}.
\end{enumerate} 
\end{remark}

The following proposition contains two classical consequences of the boundary Harnack inequality for Lipschitz domains, for which we give a short proof for completeness. To avoid confusion, let us recall that we allow Lipschitz functions to be unbounded.

\begin{proposition}[Boundary Harnack inequality]\label{prop:boundaryharnack}
Let $d=N-1\ge 1$ 
and consider a Lipschitz function $g:\xR^d \to \xR $.  
Set $\Omega=\{(x',x_N)\,;\, x_N<g(x')\}$. There exist three constants $A>1$, $K>0$ and $L>1$ depending only on $N$ and $\lA \nabla g\rA_{L^{\infty}}$ such that the following properties hold.

\begin{enumerate}
\item  For all $R>0$, $x_0 \in \partial \Omega$ and for all positive 
harmonic functions 
$u,v\colon B_{LR}(x_0)\cap \Omega\to (0,+\infty)$, 
continuously vanishing on $B_{LR}(x_0)\cap \partial\Omega$,
\begin{equation}\label{eq:harnack} 
\sup_{B_{R}(x_0)\cap \Omega}\frac{u}{v}\le A \inf_{B_{R}(x_0)\cap\Omega}\frac{u}{v}\cdot
\end{equation} 
\item  Moreover, for all $0<r<R/L$, for all $x_0\in \partial \Omega$, and for all couples of positive 
harmonic functions 
$u,v\colon B_{R}(x_0) \cap \Omega\to (0,+\infty)$, 
continuously vanishing on $B_{R}(x_0)\cap \partial\Omega$,
\begin{equation} \label{eq:Hreg} 
\underset{B_r(x_0)\cap \Omega}{\osc}\,\,\frac{u}{v}\le 
K\left(\frac{r}{R}\right)^\alpha\underset{B_R(x_0)\cap \Omega}{\osc}\,\,\frac{u}{v},
\end{equation} 
where the oscillation is defined by 
$\osc_{F} f=\sup_{F}f-\inf_F f$. 
\end{enumerate}
\end{proposition}

\begin{proof} 
As recalled in Remark~\ref{R:5.2}, the boundary Harnack inequality holds 
for $\Omega$ for any Lipschitz continuous 
$g$. 
The wanted inequality \e{eq:harnack} then follows from \e{UBHI}.

We now move to the proof of \e{eq:Hreg}. Let 
$x_0\in \partial\Omega$ and consider two real numbers $\rho$ and $R$ such that $0<\rho<R/L$. Define 
$$
\omega(\rho)=\underset{B_\rho(x_0)\cap \Omega}{\osc}\,\,\frac{u}{v}=
\sup_{x\in B_\rho(x_0) \cap \Omega } \frac{u(x)}{v(x)} 
-\inf_{y \in B_\rho(x_0) \cap \Omega }\frac{u(y)}{v(y)}.
$$
We claim that
\be\label{N50}
\omega(\rho)\le \left(1-\frac{1}{2A}\right)
\omega(L\rho).
\ee 
To see this, introduce
\[
\alpha= \sup_{B_{L\rho}(x_0)\cap \Omega}  \frac{u(x)}{v(x)}, \quad \beta= \inf_{B_{L\rho}(x_0)\cap\Omega}\frac{u(x)}{v(x)}.
\]
Pick a point $x_1 \in B_{\rho}(x_0) \cap \Omega$. Then 
\[ \frac{u(x_1)}{v(x_1)} \ge \frac{\alpha+\beta}2 \quad \text{ or } \quad  \frac{u(x_1)}{v(x_1)} \le \frac{\alpha+\beta}2. \]
The argument is essentially the same in either case, so assume the first inequality. 
Now consider the function $ u^* = u - \beta v$. 
By definition of $\beta$, $u^*\colon B_{L\rho}\cap \Omega\to (0,+\infty)$ is a positive harmonic function, 
continuously vanishing on $B_{L\rho}\cap \partial\Omega$, 
for which we can apply the 
boundary Harnack inequality~\e{UBHI}. It follows that
$$
\forall y\in B_{\rho}(x_0)\cap \Omega, \qquad \frac{u(x_1)}{v(x_1)}-\beta  \le A \Big( \frac{u(y)}{v(y) }-\beta \Big).
$$
Hence, by combining the two previous inequalities, we deduce that
\[
\forall y\in B_{\rho}(x_0)\cap \Omega, \qquad \frac{u(y)}{v(y) } \ge \frac{\alpha-\beta}{2A} + \beta. \] 
By taking the infimum over $B_{\rho}(x_0)\cap\Omega$, we infer that 
$$
\inf_{B_{\rho}(x_0)\cap\Omega}\frac{u(y)}{v(y)}\ge \frac{\alpha-\beta}{2A} + \beta,
$$
and so
\[
\omega(\rho)\le \left(1-\frac1{2A}\right) (\alpha-\beta).
\]
This proves the claim~\e{N50}. The classical iteration argument gives \eqref{eq:Hreg}.
\end{proof} 

\subsection{A uniqueness result} 

We will make use of the following consequence of the boundary Harnack principle in Proposition \ref{prop:boundaryharnack}.

\begin{proposition} \label{lem:unique} 
Let $d=N-1\ge 1$ and consider a Lipschitz function 
$g\colon\xR^{d}\to \xR$. Set $\Omega= \{x=(x',x_N)\in\xR^d\times \xR : x_N < g(x')\} $.  
There exists a 
function $u\in C^2(\Omega)\cap C^0(\overline{\Omega})$ such that 
$$
\Delta u=0 \quad\text{in }\Omega, \quad 
u>0 \quad\text{in }\Omega,\quad 
u=0 \quad\text{on }\partial\Omega.
$$ 
It is unique up to the multiplication by a constant. 
\end{proposition}
\begin{proof}
The existence of such an harmonic function 
can be seen through an approximation by bounded domains. Let $u$ and $v$ be non-negative harmonic functions 
which vanish at the boundary. Fix a point $x_0\in \Omega$. Up to multiplying $u$ by a positive constant, we 
may  assume without loss of generality that $v(x_0)= u(x_0)$. 
It follows from Proposition \ref{prop:boundaryharnack} 
that there exists a constant $C>0$ such that, for all $R>0$,
\[ C^{-1} \le  \frac{u(x)}{v(x)} 
\le C \] 
for all $x\in B_R(x_0)\cap \Omega$. 
Since we may choose $R$ as large as we like, we deduce that the inequality holds in $\Omega$, that is
\be\label{n731}
\forall x\in \Omega,\quad 
C^{-1} u(x) \le v(x) \le C u(x).
\ee
Let $r $ be sufficiently large and $R>Lr$. Again by Proposition \ref{prop:boundaryharnack}
\[ \osc_{B_r \cap \Omega}\frac{u}{v} \le A \big(r/R\big)^{\alpha} \osc_{B_R\cap \Omega} \frac{u}{v}
\le C \big(r/R\big)^{\alpha}, \]
where we used \e{n731} to obtain the second inequality. 
We then choose $R$ large to deduce that $u/v$ is constant, hence $u=v$.
\end{proof}
\begin{remark}\label{rem:gradient}
In particular, if $g$ is additionally bounded by $C$ 
there is such a harmonic function $u$ which in addition satisfies 
\[
 |u(x)+x_N|  \le  C,
\]
and hence this bound 
holds for all positive harmonic function vanishing on $\partial\Omega$ (up to a multiplicative constant).  
Moreover, since $ u+x_N $ is harmonic on $ B_R(x) $ if $R \le -x_N-C$, it follows from the Poisson formula (see~\e{Poissoncenter}) that
\begin{equation}\label{eq:gradientbound}
      |\nabla u + e_N| \le  \frac{CN}{-x_N-C }.
      \end{equation}
As a consequence $ \nabla u(x) \to -e_N$ uniformly as $ x_N \to - \infty$.
\end{remark}

\subsection{Global well-posedness for smooth data}

As a second application of the boundary Harnack inequality, we will 
prove a global well-posedness result for smooth solutions. 

\begin{theorem}\label{T:Cauchyglobal}
Consider an integer 
$d\ge 1$ and a real number $s>d/2+1$. For any initial data $h_0$ in $H^s(\xR^d)$, there exists a unique solution 
\be\label{n134}
h\in C^0([0,+\infty);H^s(\xR^d))\cap C^\infty((0,+\infty)\times \xR^d),
\ee
to the Cauchy problem
\begin{equation}\label{Hele-Shaw130}
\partial_{t}h+G(h)h=0,
\quad h\arrowvert_{t=0}=h_0.
\end{equation}
\end{theorem}
\begin{proof}
We have recalled (see Theorem~\ref{T:Cauchy}) that the Cauchy problem is well-posed locally in time for smooth initial data. More precisely, there exists a time $T>0$
and a unique local solution $h \in C^0([0,T]; H^s(\xR^d))$ with $T$ and $\sup_{0\le t \le T}  \Vert h(t) \Vert_{H^s}$ bounded by constants 
depending only on $d$, $s$ and $ \Vert h_0 \Vert_{H^s}$.
To obtain Theorem~\ref{T:Cauchyglobal}, we will combine four ingredients:
\begin{enumerate}
\item A continuation argument which states that 
smooth solutions persist as long as some 
H\"older norm of the slope is controlled.\footnote{This step is not mandatory. Instead of proving a blow-up criterion for Sobolev solutions involving 
the control of some 
H\"older norm, 
one can aim to directly 
solve the Cauchy problem for initial data in H\"older spaces. However, one needs a sharp result valid for any initial data 
$h_0$ in $C^{1,\alpha}$ for some $\alpha>0$. 
Antontsev, Gon\c{c}alves, and Meirmanov have addressed this issue in their work \cite{MR1942849} for the case when $N=2$. Let us mention that we will extend this analysis to the general problem in arbitrary dimensions in a future study.}
\item A positive lower and an upper bound on the velocity of the interface in terms of $\Vert h(t,.) \Vert_{C^{1,\alpha}}$. This related to the Taylor sign condition and it  is the subject of Proposition~\ref{prop:taylor} in the appendix. 
\item A maximum principle for the slope proven in Corollary \ref{cor:graph}.
\item The boundary Harnack inequality to control some H\"older norm of the slope in terms of its $L^\infty$-norm and the distance from the initial surface. 
\end{enumerate}

%The second pointfirt point other two are in themselves fairly routine adaptations 
%of already known results.
We emphasize that the third point is already known (see~\eqref{n32}), while the key ingredient to prove the final point is an observation by Athanasopoulos and Caffarelli~\cite{AtCa-1985}. 
In a sense, the main novelty here 
is the second point and the observation that one can use the formulation 
of the obstacle problem introduced in Section~\ref{S:2} to solve this problem.

We start with the following continuation argument. 
\begin{lemma}\label{P:Cauchy}
Let $\alpha\in (0,1)$. Consider a regular solution $h$ to \e{Hele-Shaw130}, as given by Theorem~\ref{T:Cauchy}, 
and define by $T^*$ its lifespan. Then the following alternative holds:
either $T^*=+\infty$ or
\be\label{n110}
\limsup_{t\to T^*}\lA h(t)\rA_{C^{1,\alpha}(\xR^d)}=+\infty.
\ee
\end{lemma}
\begin{proof}Such blow-up criteria are classical for quasi-linear evolution equations. They are systematically discussed by Taylor in his books  (\cite{MR1121019,MR2744149}) who deduces them from applications of the paradifferential calculus. 
On the other hand, one can prove the existence of solutions to the Hele-Shaw 
equation by using an approach based on paradifferential calculus. 
This is the approach followed in \cite{AMS,MR4090462}. Then the blow-up criterion~\e{n110} follows directly from the analysis there together combined with the tame estimates for the paralinearization of the Dirichlet-to-Neumann operator $G(h)$ proved by de Poyferr\'e in~\cite{dePoyferre1}.
\end{proof}

Consequently, 
to prove a global well-posedness result, it is sufficient 
to control the $C^{1,\alpha}(\xR^d)$-norm of the solution. 
To this end, we begin by observating that, by Corollary \ref{cor:graph}, any modulus of continuity is preserved. In particular the Lipschitz constant of the initial data is a Lipschitz constant for the solution at later times: 
\be\label{n501}
\sup_{x'\in \xR^d}\la \nabla_{x'} h(t,x')\ra\le \sup_{x'\in \xR^d}\la \nabla_{x'} h(0,x')\ra.
\ee
As a result, it will be sufficient to prove that 
some H\"older norm of $\nabla_{x'}h$ is controlled in terms of its $L^\infty$-norm. To do so, we use a method introduced by 
Athanasopoulos and Caffarelli~\cite{AtCa-1985}. 

Denote by $\varphi(t)$ the harmonic extension of $h$ in $\mathcal{O}(t)=\{x\in \xR^N\sep x_N<h(t,x')\}$ where $N=d+1$. Then we use the change of variables and unknowns already introduced in section~\ref{S:2.3}. 
Namely, we successively define the functions  
$P$, $p$ and $u$ on $[0,T^*)\times \xR^N$ by  
\begin{align*}
P(t,x)&=
\left\{
\begin{aligned}
&\varphi(t,x)-x_N\quad &&\text{if }x\in \mathcal{O}(t)\\
&0 \quad &&\text{if }x\in \xR^N\setminus \mathcal{O}(t),
\end{aligned}
\right.\\
p(t,x)&=P(t,x',x_N-t),\\
u(t,x)&=\int_0^tp(\tau,x)\dtau.
\end{align*}
As we have already seen (see Proposition~\ref{prop:eulerobstacle}), 
for any time $t\in (0,T^*)$, $u(t)$ belongs to $C^1(\xR^N)$ and satisfies
\be\label{nu-2}
\left\{
\begin{aligned}
&\Delta u(t)=\chi_{\Omega(t)\setminus\Omega(0)}\quad
\text{where}\quad \Omega(\tau)=\mathcal{O}(\tau)+\tau e_N,\\
&u(t)=0 \quad \text{on}\quad\partial\Omega(t),\\
&\nabla u(t)=0 \quad \text{on}\quad\partial\Omega(t),\\
&\lim_{s\to -\infty}\nabla u(t,x+se_N) = -t e_N.
\end{aligned}
\right.
\ee
We also set $f=h+t$ so that 
$\partial\Omega(t)=\{x_N=f(t,x')\}$. 

We begin by showing a monotonicity inequality. To this end, observe that, since $\Omega(0)$ is assumed to be a Lipschitz graph, there exists $\varepsilon > 0$ such that, for any vector $\nu \in \mathbb{R}^N$ of the form $\nu = e_N + v$ with $|v| < \varepsilon$, the following inclusion holds:
\[
\Omega(0) \subset \Omega(0) + s \nu \quad \text{for all} \quad s \geq 0.
\]
Now, using the obstacle problem formulation and the monotonicity property (2) from Theorem~\ref{thm:mainobstacle}, we have the inequality
\begin{equation}\label{n117}
u(t, x + s \nu) \leq u(t, x)
\end{equation}
for all $x \in \mathbb{R}^N$ and all $s, t \geq 0$. 

%As a consequence, the level set $\{ x \in \mathbb{R}^N \;|\; u(t, x) = \lambda \}$ is the graph of a Lipschitz function of the form $x_N = f_\lambda(t, x')$, satisfying the Lipschitz estimate
%\[
%f_\lambda(t, x') \leq f_\lambda(t, y') + %\varepsilon^{-1} |x' - y'|
%\]
%for all $x', y' \in \mathbb{R}^{N-1}$.

Our second goal is to prove that 
\be\label{dnu>0}
-\partial_\nu u=- \nu \cdot \nabla u(t) > 0 \quad \text{inside} \quad \Omega(t) \setminus \Omega(0)
\ee
for any unit vector $\nu$ sufficiently close to the vertical unit vector $e_N$. 
To this end, we introduce the finite difference operator $\delta_h$ defined by
\[
(\delta_h w)(x) = \frac{w(x + h \nu) - w(x)}{h}, \quad h \in (0,1].
\]
Let $n_0$ denote the outward-pointing unit normal vector to $\partial \Omega(0)$, given by
\[
n_0 = \frac{1}{\sqrt{1 + |\nabla_{x'} h_0|^2}} 
\begin{pmatrix}
- \nabla_{x'} h_0 \\
1
\end{pmatrix}.
\]
Since $h_0$ is Lipschitz continuous, for any unit vector $\nu$ sufficiently close to the vertical vector $e_N$, we have
\[
\nu \cdot n_0 > 0 \quad \text{on} \quad \partial \Omega(0).
\]
As a consequence, recalling that $\Omega(0) \subset \Omega(t)$, we observe that for all $t \in (0, T^*)$, all $h \in (0,1]$, and all $x$ such that $x + h \nu \in \Omega(t)$, then $ \delta_h \chi_{\Omega(t)\backslash \Omega(0)} = 0 $ unless $ x \in \Omega(0) $ and $ x+h\nu \notin \Omega(0)$. In the last case  $ \delta_h \chi_{\Omega(t)\backslash \Omega(0)} = 1$. Hence  the following inequality holds:
\[
\delta_h \chi_{\Omega(t) \setminus \Omega(0)}(x) \geq 0.
\]
It then follows from the first equation in~\eqref{nu-2} that $\delta_h u$ is a subharmonic function in $\Omega(t) - h \nu$. 
Moreover $ \delta_h u \le 0$ by \eqref{n117}.  
Since $ u \in C^1(\Omega(t))$ this implies that 
$ \partial_\nu u$ is subharmonic and nonpositive.  By the strong maximum principle either $ \partial_\nu u$ is identically $0$ which implies that $u$ is identically $0$, which is absurd, or, $ \partial_\nu u < 0 $ in $ \Omega(t)$. In particular $ \partial_N u < 0 $ in $ \Omega(t)$ and by the implicit function theorem any level set $L_\lambda=  \{ x: u(t,x) = \lambda\} $ is locally (and hence globally) 
the graph of a $C^1$ function $f_\lambda$, i.e. 
\[ L_\lambda = \{  x: x_N = f_\lambda(t,x') \}. \] 

We now aim to prove that the level sets are in fact of class $C^{1,\alpha}$ for some $\alpha > 0$. To this end, we apply the Boundary Harnack Inequality. More precisely, for each index $1 \leq j \leq d = N - 1$, we consider the pair of functions
\[
v_1 = -\partial_{x_N} u, \quad v_2 = -\frac{1}{\sqrt{1 + \varepsilon^2}}(e_N + \varepsilon e_j) \cdot \nabla u,
\]
where $\varepsilon > 0$ is chosen sufficiently small so that, in view of~\eqref{dnu>0}, both $v_1$ and $v_2$ are positive inside $\Omega(t) \setminus \Omega(0)$. This requires some additional care: in order to apply the Boundary Harnack inequality to harmonic functions defined in $\Omega(t) \setminus \Omega(0)$, we need a quantitative estimate ensuring that the distance between $\partial\Omega(t)$ and $\partial\Omega(0)$ is bounded from below. To obtain this, we rely on a refined maximum principle, proved in Appendix~\ref{A:DN}, which establishes a lower bound on the Taylor coefficient. We now turn to the details.

Recall (see~\eqref{n32}) that the boundaries of the sets $\Omega(t)$ are Lipschitz graphs, with Lipschitz constants no larger than that of $h(0,\cdot)$. 
In addition, by using the property~\eqref{boundedbelowlambda} in Remark \ref{R:4} after Theorem \ref{T:Cauchy}, we have
\[  h(t, x') \ge h(0,x') -t+\frac{t}\lambda  \] 
for some $ \lambda$ depending only on $s$, $d$ and $\Vert h_0 \Vert_{H^s} $. 
Clearly, the previous property is equivalent to
\be\label{n:f82}
f(t,x')\ge f(0,x')+\frac{t}\lambda.
\ee

Now let $T \leq t < T^*$, where $T^*$ denotes the lifespan, and $T$ is a positive time such that both $T$ and 
$\sup_{0 \leq t \leq T} \| h(t) \|_{H^s}$ are bounded by constants depending only on $d$, $s$, and $\| h_0 \|_{H^s}$. In particular, by Sobolev embedding, this provides a uniform bound on the $C^0([0,T]; C^{1,\alpha})$-norm of $h$, for some $\alpha > 0$. Once $T$ is so determined, we can find a curved strip of fixed positive width included in $\Omega(t) \setminus \Omega(0)$. Indeed, it follows from~\eqref{n:f82} that both functions $v_1$ and $v_2$ are harmonic in the strip
\[
U := \left\{ x : f(t,x') - \frac{1}{\lambda} T < x_N < f(t,x') \right\} \subset \Omega(t) \setminus \Omega(0).
\]
Moreover, $v_1, v_2 > 0$, and the ratio $|v_2 / v_1|$ is bounded by a constant depending only on the Lipschitz constant of $f(t, \cdot)$, hence by the Lipschitz constant of $f_0=h_0$.

The Boundary Harnack Inequality then implies that, for all $T \leq t < T^*$, the ratio $v_2 / v_1$ belongs to the Hölder space $C^{0,\alpha}(U')$ for some $\alpha > 0$, with a uniform bound depending only on $\| h_0 \|_{C^{1,\alpha}}$, where
\[
U' := \left\{ x : f(t,x') - \frac{1}{2\lambda} T < x_N < f(t,x') \right\}\subset U.
\]
By the implicit function theorem, the level sets of $u$ in $U'$ are of class $C^{1,\alpha}$ with uniform bounds. Passing to the limit, we conclude that $f(t, \cdot) \in C^{1,\alpha}$ with bounds depending only on the $H^s$-norm of the initial data. This concludes the proof of Theorem~\ref{T:Cauchyglobal}.
\end{proof}

\section{Eventual regularity} 
\label{sec:eventualreg} 

We now turn to the proof of eventual regularity, 
that is the last point of Theorem~\ref{thm:mainobstacleintro}. This proof depends on a regularity result of 
Caffarelli, Theorem 6 in Caffarelli \cite{MR1658612}.  The latter states that the free boundary is smooth unless the contact set is contained in a thin strip. 

\begin{proposition}\label{thm:6Caffarelli}
Let $ \rho >0$. There exists $M=M(N,\rho)$ so that the following is true: Suppose  $ u\in H^1_+(B_M(x_0)) $ is  a solution to the obstacle problem 
\[ \Delta u = \chi_{\{u>0\} } \]
and consider a point $ x_0 \in \partial\{ u>0\}$. 
Suppose that 
\[   \inf_{\nu \in \partial B_1(0)} \Big( \sup_{x\in B_1(x_0) \cap \{ u = 0 \} } \langle \nu, x\rangle -\inf_{x\in B_1(x_0)\cap \{ u = 0 \}  } \langle \nu, x \rangle \Big) >\rho.    \]
Then the level sets of $u$ of positive values and the free boundary in $B_{\rho/32}(0)$ are uniformly of class $C^{1,\alpha} $ for some $ \alpha(N) >0$.\label{prop:Caffarelli}
\end{proposition} 

More is true: Suppose the conclusion of the Proposition holds. Then the free interface is analytic, see Kinderlehrer and Nirenberg 
\cite{MR0440187}. An inspection of the proofs shows that there are uniform bounds of analytic norms depending only on $\rho$.

There are two situations in which we apply Caffarelli's result. First we deduce the eventual regularity assertion 
\eqref{part:eventual} in Theorem \ref{thm:mainobstacleintro}.

Consider an open  connected subset $\Omega_0\subset \mathbb{R}^N$ 
such that,  
for some $C>0$,
\begin{equation}\label{n20bb}
\{x\in\mathbb{R}^N\sep x_N < -C \}\subset  \Omega_0 \subset \{ x\in\mathbb{R}^N\sep x_N <C \}.
\end{equation}
As seen in Theorem~\ref{thm:mainobstacle}, 
for all $t>0$, there exists a unique variational 
solution $v=v(t)\in H^1_{\loc,+}(\xR^N)$ 
to the obstacle problem  
\[
\Delta v (t)= \chi_{\{ v(t)>0\} \backslash \Omega_0}
\]
satisfying
\begin{equation*} 
\forall x\in\xR^N,\quad   v_0(t,x_N +C) \le v(t,x)  \le v_0(t,x_N-C),
\end{equation*}
where $v_0$ is given by~\e{eq:v0}. In particular, by comparison, the set $\Omega_t=\{x\in\xR^N; v(t,x)>0\}$ satisfies 
\be\label{n740}
\{  x_N < -C+t\} \subset \Omega_t \subset \{     x_N < C+t \}.
\ee

We fix $ \rho=\frac14$ and denote by $M$ the correspondent constant of the proposition.

Let $t \ge 3C$ and 
$x_0 \in \partial \Omega_t $. Remember that the statement \eqref{item:eventual} in Theorem~\ref{thm:mainobstacle} ensures that $\partial\Omega_t$ is the graph of some function $f$ for $t\ge 2C$. In addition, the property \e{n740} implies that $B_{t+C}(x_0)\cap \Omega_0=\emptyset$ and hence
$$
\Delta v(t) = \chi_{v(t)>0 } \quad\text{in}\quad 
B_{t+2C}(x_0).
$$
Now consider with $ r  =  (t+C)/M$ (and $t$ 
large to be specified later on) the rescaled function
\[  u(y)=  r^{-2}  v \big(t, x_0 +   y r \big). \]
It satisfies 
$$
\Delta u = \chi_{\{ u>0\}}   \quad\text{in}\quad B_M(0)
$$
and
\[ \Big\{ y_N < -  \frac{CM}{C+t}   \Big\} \subset \{ u >0 \} \subset \Big\{ y_N< \frac{CM}{C+t}   \Big\}.   \]
We choose  $t\ge 8M C $ and we arrive at the situation of Proposition \ref{prop:Caffarelli}. It follows that $ \partial \{ u >0\} $ is the graph $x_N = h(t, x') $ for some function $ h \in C^{1,\alpha} $ with 
\[ \Vert h(x'/r )  \Vert_{C^{1,\alpha}} \le C(N), \quad \alpha = \alpha(N). \]
This implies the existence of $ \delta>0$ depending only on $N$ so that for $t \ge 8MC$ and $ \gamma >0$
\[  |\partial^\gamma  h(t, x')| \le  \left( \frac{C|\gamma|}{\delta t}  \right)^{|\gamma|}.       \]

The second application that we work out of Proposition~\ref{thm:6Caffarelli} has to do with the subgraph case: assume that 
$\mathcal{U}_0 = \{ x_N < f_0(x')\}$ 
for some bounded Lipschitz function $f_0$. 
Then Corollary~\ref{cor:graph} gives a solution $f=f(t,x) $ where $f(t)$ is Lipschitz continuous with the same Lipschitz constant as $f_0$ (and then the solution to the original Hele-Shaw equation~\e{n7} is the function 
$h(t,x')=f(t,x')-t$). We can apply Proposition \ref{prop:Caffarelli} to deduce that $f(t)$ (or equivalently $h(t)$) is of class 
$C^{1,\alpha}$ for some $ \alpha = \alpha(N)$ and 
\[
\Vert f(t) \Vert_{C^{1,\alpha}}\le \mathcal{F} \Big(\Vert f_0 \Vert_{W^{1,\infty}} , \inf_{x'}  f(t,x')- f(0,x'), C \Big).   \]

\begin{theorem} \label{thm:analytic} 
Suppose that $ f_0$ is Lipschitz continuous.
Assume that for some $x_0\in \xR^d$ and $t >0$ we have 
$f(x_0,t) > f_0(x_0)$. Then $f$ is analytic in $x$ near $x_0$. 
\end{theorem} 
\begin{proof}Denote by $L$ the 
Lipschitz constant $L$ for $f_0$ and hence for $f(t,\cdot)$.  
Let 
\[ r = (f(t,x_0)-f_0(x_0))/(1+L)    \]
so that $ B_r(x_0, f(t,x_0)) \cap \{ x_N < f_0(x_0) \} = \emptyset$.

We fix a constant $ \rho $ depending only on the Lipschitz constant $L$ so that for no $x'$ 
and $R$, 
$ \{ x_N < f(t,x')\} \cap B_R(x',f(t,x'))$ is contained in a strip of size $\rho R$. Let $M$ be the corresponding constant of Proposition \ref{prop:Caffarelli}.

Let $w(x,y)$ be the corresponding solution at time $t$ to the obstacle problem. Then 
\[   \tilde w (x,y) =  M^2r^{-2}  w\Big(x_0 + \frac{r}{M}  (x-x_0),f(t,x') + \frac{r}{M} y\Big)  \]
is a solution to the obstacle problem 
\[ \Delta \tilde w = \chi_{\{\tilde w >0 \}} \qquad \text{ on } B_M(0).  \] 
Proposition \ref{prop:Caffarelli} now implies that the free boundary in $B_{\rho/32}(0) $ of $ \tilde w$ is a $C^{1,\alpha} $ hyper-surface, hence $f(t,\cdot)\in C^{1,\alpha} $ near $ x_0 $.
More precisely, for any multiindex $\gamma$ of length at least $1$ 
\[  |\partial^\gamma f(t,x')| \le  c(L,n) (r(L,N))^{-|\gamma|}  |f(t,x')-f_0(x')|^{1-|\gamma|}.\]
This completes the proof.
\end{proof} 

This gives an alternative proof for global existence from local existence: Indeed, by Proposition~\ref{prop:taylor}, if $ f_0 \in C^{1,\alpha} $ there exist $t_0$ and $ \kappa >0$ such that
$f(t,x)> f(0,x)+ \kappa t $ for $ 0 \le t \le t_0$. As already seen (see~\e{n501}), by Corollary  \ref{cor:graph} the Lipschitz constant is preserved and hence by Theorem \ref{thm:analytic} the solution 
$f(t,\cdot)$ is analytic.

\section{Waiting time, subsolutions and supersolutions}\label{S:8}

In this chapter we consider Lipschitz initial data. As we have seen, for initial data which are smoother (say $h_0\in H^s(\xR^d)$ with $s>d/2+1$), there is global existence and, in addition, a smoothing effect: the solution $h(t)$ belongs to $C^\infty(\xR^d)$ for all positive time $t>0$. Our goal is to investigate
the critical case where $h_0$ is merely Lipschitz.

Consider an angle
$\alpha\in (0,2\pi)$ and denote by
$C_\alpha$ the cone with aperture $\alpha$ and apex 
at the origin, 
that is 
\[
C_\alpha = \{ x=(x',x_N) \in \xR^{N-1}\times \xR : |x'|<-\tan(\alpha/2) x_N  \}.
\]
Thanks to Lemma \ref{lem:unique}, we can introduce the unique homogeneous positive harmonic function 
$u_\alpha $ vanishing at the boundary of the cone; to clarify matters recall that this harmonic function is unique up to a multiplicative constant. 
Notice that the function $u_\alpha$ can be written in polar coordinates 
as $r^{\lambda(\alpha)}v(\sigma)$ with 
\[ \lambda(\alpha)(\lambda(\alpha)+d-2) v = -\Delta_{\xS^{d-1}} v \]
and where the homogeneity $\lambda(\alpha)$ is a continuous strictly monotonically decreasing  function of the angle $ \alpha$, satisfying 
\[
\lim_{\alpha\to 0} \lambda(\alpha) = \infty, 
\quad \lambda(\pi) = 1, \quad \lim_{\alpha \to 2\pi}
\lambda(\alpha) = \left\{ \begin{array}{cl} \frac12 & \text{ if } N=2 \\ 
                                                 0 & \text{ if } N \ge 3.
                                                 \end{array} \right. 
\]
By uniqueness (up to a multiplicative constant), we see that $u_\alpha$ is radial in the $x'$ variable. 
We normalize $u_\alpha$ so that
%$$
%\la \nabla u_\alpha\ra=1 \quad\text{on}\quad \partial %C_\alpha\cap \partial B_1(0).
%$$
%Then
$$
\la \nabla u_\alpha\ra=\la x\ra^{\lambda(\alpha)-1} \quad\text{on}\quad \partial C_\alpha.
$$
Define the critical angle $\alpha_2$ as the solution to 
$$
\lambda(\alpha_2)=2.
$$
The functions 
\[  (N-1)  x_N^2 -  |x'|^2 \]
are harmonic and positive 
on $C_{\alpha_2}$. They vanish
at the boundary of the cone $(N-1)x_N^2= |x|^2 $. As a result
\[
\alpha_2 =2\arctan(\sqrt{N-1}).
\]
Notice that in two  space dimensions $ \alpha_2 = \pi/2$.

%%hence 
%\[ \alpha_2 =2\arctan(\sqrt{n-1}).  \]
%In two  space dimensions $ \alpha_2 = \pi/2$. 
%We have
%%$$
%\lambda(\alpha_2)=2.
%$$

Let us also mention that, in the following, we will make extensive use of the notion of solution introduced in section~\ref{S:3.3}, and of the results proved there.

\begin{proposition}
\label{prop:waiting}
Let $ \Omega \subset \xR^N$ be open and connected and 
$p\colon [0,1)\times \Omega \to [0,\infty)$ be the pressure of a solution to the Hele-Shaw problem (in particular $p$ is monotonically growing in $t$). Let 
\[ x_0 \in \partial\{ p(0,.)>0\}\cap \Omega.   \]
\begin{enumerate} 
 \item (Waiting time) If there exists $\alpha\in (0,\alpha_2) $ and $r>0$ so that  $B_r(x_0) \subset \Omega$ and  
 \[ B_r(x_0) \cap \{ p(0,.)>0\} \subset B_r(x_0) \cap C_\alpha(x_0) \]
then there exists $ T>0$ so that, for all $ 0 \le t \le T $,
 \[ \partial B_r(x_0) \cap \{ p(t,.)>0\} \subset \partial B_r(x_0) \cap C_{\alpha_2}(x_0).\]
In particular, there holds $ p(t,x_0)=0$ for $ t \in [0,T]$.
\item (Immediate movement) If there exists $ \alpha_2 < \alpha < \pi  $ and $r>0$ so that $B_r(x_0) \subset \Omega $ and 
\[ B_r(x_0) \cap \{ p(0,.)>0\} \supset B_r(x_0) \cap C_\alpha(x) \]
then there exist $\varepsilon>0$ and $ T >0$ so that 
\begin{equation}\label{eq:move}
    p(t,x_0+ \varepsilon t^{\frac1{2-\lambda(\alpha)}}e_N )>0 
    \end{equation}
for all $0<t < T$.
\end{enumerate}
\end{proposition}
\begin{remark}
The first case establishes a waiting time: The front does not move immediately at $x_0$. The second cases establishes immediate movement. 

By combining this proposition with Theorem~\ref{thm:analytic}, we obtain the last conclusion of Theorem~\ref{T:di} about the immediate smoothing.
\end{remark}

\begin{proof} Let $\alpha$ be as in the proposition. 
We consider the case  $ \alpha < \alpha_2$, and define  
\[ p^+(t,x) := \left\{\begin{array}{cc} 0 \qquad &  \text{ if }  x \notin C_{\alpha+t}, \\ 
  u_{\alpha+t}(x) & \text{ if } x \in C_{\alpha+t}.  \end{array} \right. \] 

\begin{lemma} \label{lem:supersolution} The function $p^+(t,x)$ is a classical supersolution for $|x| \le 1$ and $0\le t \le \alpha_2-\alpha$.
\end{lemma}
\begin{proof}
With the convention above $u_{\alpha+t}$ is supported in the cone $C_{\alpha+t}$ and
\[ |\nabla u_{\alpha+t}|
=|x|^{\lambda(\alpha+t)  -1} \]
at the boundary of the cone $C_{\alpha+t}$.

Notice that the modulus of the normal velocity of the boundary of the support of $u$ is $|x|$   and 
\[   |x| \ge    |x|^{\lambda(\alpha+t)-1} = |\nabla u(x)|   \]
provided $|x|\le 1$ lies in the boundary of the support  since  $ \lambda(\alpha+t)\ge 2$ as long as $ \alpha+t \le \alpha_2$. 
This proves the lemma.
%The normal velocity is $|x|$ which is not smaller than $|x|^{\lambda(\alpha)-1}$ as long as $|x|\le 1$ . This implies the claim. 
\end{proof}
Let us now prove that Lemma \ref{lem:supersolution} implies the first claim on the waiting time. 
Without loss of generality we can assume that $x_0=0$ and by scaling $r=1$ in the assumption of Proposition \ref{prop:waiting}.
Increasing $\alpha$ slightly if necessary (so that still $ \alpha < \alpha_2$, or, equivalently, $ \lambda(\alpha) > 2$)  we see that $u_\alpha$ has a positive lower bound on  
\[ \partial B_1(0) \cap \supp p(0,.).\]
For every $\kappa,\lambda>0$ the pressure 
\[ \tilde p^+(t,x)= \kappa  p^+(\kappa t,  x) =     \kappa  u_{\alpha+ \kappa t} (x)  \]
is a super solution by Lemma \ref{lem:supersolution} and the symmetries of the problem in 
$[0,  \kappa^{-1})\times B_{1}  (0) $.
Next we choose $\kappa$ large so that 
\[  \tilde p^+(0,x) \ge p(0,x) \qquad \text{ for } |x| =  1
\]
which implies by the maximum principle for harmonic functions that 
\[  \tilde p^+(0,x) \ge p(0,x) \qquad \text{ for } |x| \le  1.
\]
Next we choose 
$T \le  \kappa^{-1}$ small so that
\[  \tilde p^+(t,x) \ge p(t,x) \quad \text{ for } t \le T \text{ and } |x| = 1 .\]
 Now Proposition  \ref{lem:subsup}
and Lemma \ref{lem:subobstacle} imply the first claim on the waiting time. 

\bigskip

In contrast to that there is no waiting time if a cone of a large angle is contained in the initial support. 
Let $ \alpha> \alpha_2$ and let  
\[ V := \big\{ x:  x_N < 0 \text{ and } \tan^2(\alpha/2) x_N^2 -|x'|^2=   1    \big\} \subset C_\alpha \]
be the regularised cone, the boundary of which is a sheet of a hyperboloid. Let $u_\alpha$ be the harmonic function on the cone from the previous chapter normalized by $ |\nabla v_\alpha| = |x|^{\lambda(\alpha)-1}$ at the boundary of the cone, and let $v_\alpha$
be the unique positive harmonic function in $V$ vanishing at the boundary and normalized by 
\[ v_\alpha( x-e_N) \le u_\alpha \le v_\alpha (x). \]
This defines $u_\alpha$ uniquely since 
\[ v_{\alpha}(x) - C |x|^{\lambda(\alpha)-1}  \le    u_\alpha (x) \le v_{\alpha}(x) \]
and there is at most one multiple of $u_\alpha$ which satisfies these inequalities, which hence by 
Proposition \ref{lem:unique} determines $u_\alpha$. Existence is seen by a compact approximation.

\begin{lemma} 
Let $ \alpha > \alpha_2$.
There exists $ \kappa= \kappa(\alpha, N) $ so that 
\[  p_-(t,x) = u_\alpha(x',x_N- \kappa t)   \]
is a subsolution on $ \xR \times \xR $. 
\end{lemma}

\begin{proof} 
By the Hopf maximum principle there exists $ \kappa $ with  
\[ |\nabla u_\alpha |\ge \sqrt{2} \kappa ( 1+ |x|)^{\lambda(\alpha)-1}   \]
on $ \partial V $. Since by assumption $ \alpha > \alpha_2$ and  $\lambda(\alpha) < 2 $ the vertical component of the exterior normal at $V$ is at least $\frac1{\sqrt{2}}$.
\end{proof} 

We use this subsolution to prove the  second claim. We assume that $\alpha > \alpha_2$ and, decreasing $\alpha$ if necessary (but keeping $ \alpha > \alpha_2$), and rescaling  we may assume that $p(t,x)$ has a positive lower bound $\delta$ on $C_{\alpha+\delta} \cap \partial B_1(0)$.
We denote $\lambda = \lambda(\alpha)$ and define 
\[ \tilde p(t,x) = \delta r^\lambda p_-(\delta tr^{\lambda-2}  ,x/r)= \delta r^\lambda u_\alpha(x'/r, (x_N- \delta \kappa r^{\lambda-1} t)/r   ).  \]
Here $r$ will be a small positive number and the function  
$\tilde p(t,.)$ converges uniformly to $\delta v_\alpha(t,.)$ 
translated upward by $\delta \kappa r^{\lambda-1}$ 
on compact sets as $ r \to 0 $. 

We observe
\begin{enumerate} 
\item $ \tilde p(0,x) \le \delta v_\alpha(x) \qquad \text{ on }  \xR^N $ by construction,
\item For all $0<r<1 $ by the maximum principle $ \tilde p(0,x) \le \delta v_\alpha(x) \le  p(0,x)$ for $|x|\le 1$.
\item   $ \tilde p(t,x) = 0 $ if $ |x|=1$, $t$ small  and $ x \notin C_{\alpha+2\delta \kappa r^{\lambda-1} t} $ provided $\delta \kappa r^{\lambda-1} t < \frac{\pi-\alpha}2$. This is true since $\tilde p(t,x) \le  v_\alpha(x', x-\delta \kappa r^{\lambda-1})$.  
\item $ \tilde p(t,x) \le v_\alpha(x) + C(\alpha) \delta^2 \kappa  r^{\lambda-1} t     $ for $|x| = 1$, $ x \in C_{\alpha+2\delta \kappa r^{\lambda-1} t}$.
\item $ \tilde p(t,x) \le p(t,x)$ for $|x|\le 1$ and $ t \le \frac{1}{\kappa} r^{1-\lambda} $
since $\tilde p$ is a subsolution and the previous points verify the inequality at the parabolic boundary. The assertion follows from Proposition  \ref{lem:subsup} and Lemma \ref{lem:subobstacle}
\item $ \tilde p(t,\varepsilon t^{\frac1{2-\lambda}}) >0 $ for $ t  > (\delta\kappa)^{-1}   \tan^{-1}(\alpha/2)  r^{1-\lambda}    $.
\end{enumerate}

We have $ \tilde p(t, se_N) > 0 $ if 
\[  s<   \delta \kappa r^{\lambda-1} t    -r \tan(\alpha/2) \]
for some $ r \le c (\delta \kappa t)^{-\frac1\lambda } $. We (approximately) maximize this expression with respect to $r$ and set 
\[   r = \frac12 \left(  \frac{\delta \kappa t}{\tan (\alpha/2)}\right)^{\frac{1}{2-\lambda}}   \]
and obtain positivity for 
\[ s <\Big( 2^{1-\lambda} -\frac12\Big)\tan^{\frac{\lambda-1}{2-\lambda} } (\alpha/2)( \delta \kappa  t)^{\frac{1}{2-\lambda}} .\]
This implies \eqref{eq:move} and completes the proof.
\end{proof}

\appendix

\section{Elliptic boundary value problems} 
\label{A:DN}

We collect in this appendix classical results about the Dirichlet-to-Neumann operator and prove a result of independent interest about the Taylor coefficient.  

\subsection{Definition of the Dirichlet-to-Neumann operator}
Here we recall some results from section~$3$ in \cite{ABZ3}, where the definition of the Dirichlet-to-Neumann operator is addressed in a general context (whether the fluid domain is bounded from below or infinitely deep, as we are considering here).

Consider a bounded  Lipschitz function $h\in W^{1,\infty}(\xR^d)$ with $d=N-1$ 
and a 
function $\psi\in H^\mez (\xR^d)$. 
Introduce the set
$$
\mathcal{O}=\{x\in \xR^N\sep x_N<h(x')\}.
$$
Then there exists a unique harmonic extension 
$\varphi$ of $\psi$ in $\mathcal{O}$, 
that is a function satisfying 
$$
\varphi\in C^\infty(\mathcal{O}),\quad 
\nabla \varphi\in L^2(\mathcal{O})
\quad\text{and}\quad 
\varphi\in L^2((1+\la x_N\ra)^{-3}\dx), 
$$
and solution to
\begin{align}
\Delta\varphi&=0\quad \text{in }\mathcal{O},\label{n80}\\
\varphi\arrowvert_{x_N=h}&=\psi,\label{n81}\\
\lim_{x_N\to -\infty}\partial_{x_N}\varphi(x)&=0.\label{n82}
\end{align}
The function $\varphi$ 
is obtained by using the Lax-Milgram theorem so \e{n80} holds in the weak sense (see \cite[Definition~$3.5$]{ABZ3}). 
Notice that 
$\varphi$ belongs to $H^1(U)$ where 
$U=\{x\in \xR^N: -b<x_N<h(x')\}$ with $b>\lA h\rA_{L^\infty}$. 
Therefore, its 
trace on $\partial\mathcal{O}$ is well-defined 
as a function in $H^{\mez}(\xR^d)$, so \e{n81} has a 
clear meaning in the sense of traces of Sobolev functions. Eventually, the condition \e{n82} holds in the sense that, for any multi-index $\alpha$,
$$
\lim_{x_N\to -\infty}\lA \nabla^\alpha\partial_{x_N}\varphi(\cdot,x_N)\rA_{L^2(\xR^{N-1})}=0.
$$
In addition, since $\Delta\varphi=0$, one can express the normal 
derivative in terms of the tangential derivatives and verify that its normal derivative on $\partial\mathcal{O}$ is well-defined 
as an element of $H^{-\mez}(\xR^d)$ (see \cite[Theorem 3.8]{ABZ3}). 

Therefore, one may introduce the Dirichlet-to-Neumann operator $G(h)$. This is the operator 
$G(h)\colon H^{\mez}(\xR^d)\to H^{-\mez}(\xR^d)$ defined by
\begin{align*}
G(h)\psi (x')&=\sqrt{1+|\nabla_{x'} h|^2}\partial_n\varphi\arrowvert_{x_N=h(x')}\\
&=\partial_{x_N}\varphi(x',h(x'))-\nabla_{x'} h(x')\cdot(\nabla_{x'}\varphi)(x',h(x')).
\end{align*}

\subsection{The Hopf maximum principle}

In this section, we prove a result of independent interest which shows that, in any dimension, and for general 
$C^{1,\alpha}$-domains, the Taylor coefficient is bounded from below and from above (see Remark~\ref{R:4}).

\begin{proposition}\label{prop:taylor} 
Let $d=N-1\ge 1$, $\alpha\in (0,1]$ and $C>0$. 
There exists $ \lambda >1$ 
such that, if 
$ \Vert h \Vert_{C^{1,\alpha}(\xR^d)} \le C $ and $ u$ is a harmonic function in $\Omega=\{x\in\xR^N\,;\, x_N < h(x')\} $
which vanishes at the boundary $\partial\Omega$, with 
\[ \lim_{s\to \infty}  \frac{u(-se_N)}s = 1,
\]
then, for all $x' \in \xR^d$,  
\[  \lambda^{-1} \le \partial_N u(x', h(x'))  \le \lambda. \]
\end{proposition}
\begin{proof}
In this proof, we follow the convention of summation over repeated indices, and we denote by $B_r^+(x)$ the half-ball $B_r(x) \cap \{ x_n >0 \}$. We deduce the proposition from a more local statement.

\begin{lemma}\label{lem:hopf}  There exists $\lambda$ depending on $\alpha\in (0,1]$, $C>1$ and $ \mu >1$ so that the following is true: consider a matrix with coefficients 
$a^{ij} \in C^{0,\alpha}( B^{+}_{2}(0))$ such that 
\[  |\xi|^2 \le \mu a^{ij}\xi_i \xi_j,  \qquad  |a^{ij}| \le \mu, 
\quad  \Vert a^{ij} \Vert_{C^{0,\alpha}} \le C \]
and $a^{ij} (0) = \delta^{ij} $.
Let $v$ be a nonnegative weak solution to 
\begin{equation*}
\left\{
\begin{aligned}
&Lv = \partial_i (a^{ij}(x) \partial_j v) = 0 \quad\text{in}\quad B_2^+(0),\\
&v(x',0)=0,
\end{aligned}
\right.
\end{equation*}
satisfying $\Vert v \Vert_{L^\infty(B_{2}^+(0))} \le C$
and $v(e_N) = 1$. Then 
\[ \lambda^{-1} \le \partial_N v(x',0) \le \lambda   \qquad \text{ in } B_1^{\xR^d}(0) \times \{ 0\}. \]
\end{lemma} 

Temporarily assuming this result, let us conclude the proof of Proposition~\ref{prop:taylor}. Without loss of generality we may assume $h(0)=0$ and $ \nabla h(0) = 0 $. It suffices to give lower and upper bounds for $\partial_N u(0)$. Let $R = 2\Vert h \Vert_{L^\infty}$. By comparison argument (valid for weak solutions), there are lower and upper bounds
\[
\Big(\frac{R}2+x_N\Big)_- \le u \le \Big(x_N-\frac{R}2\Big)_-.
\]
We define for $x=(x',x_N)\in B_2^+(0)$,
\[     v(x) = (u(-Re_N))^{-1} u\big(Rx', h(Rx')-Rx_N\big),   \]
which satisfies $v \ge 0 $, $ v(x',0)=0$ and
\[ \partial_i (a^{ij}(x') \partial_j v) = 0 \qquad \text{ in } B_2^+(0) \]
\[  v(e_N) = 1,  ~ \Vert v \Vert_{L^\infty(B^+_2(0))} \le 3 \]
where the coefficients are given by 
\[   a^{ij}(x') = \delta^{ij} - \delta^{iN} \partial_j h - \partial_i h \delta^{jN} + |\nabla h|^2 \delta^{iN} \delta^{jN}. \]
In particular, the ellipticity is controlled by the Lipschitz constant:  
\[ a^{ij}\xi_i \xi_j = |\xi' - \nabla h \xi_N|^2 + \xi^2_N\ge  (1+ |\nabla h|)^{-2} |\xi|^2  \]
and we have the following bounds:
\[ |a^{ij} | \le  1+ |\nabla h|^2,\quad 
\Vert a^{ij} \Vert_{C^{0,\alpha}} \le   (1+ \Vert h \Vert_{C^{0,1}} )\Vert h \Vert_{C^{1,\alpha}}.
\]
The claim of Proposition \ref{prop:taylor} follows now from Lemma \ref{lem:hopf}.
\end{proof} 

\begin{proof}[Proof of Lemma \ref{lem:hopf}]

We begin by establishing several technical results. 

$i)$ Consider the annulus $ \Omega = B_2(0) \backslash B_1(0)$. Then the problem 
\[ -\Delta u = \nabla \cdot f \qquad \text { in } \Omega, \qquad u= 0 \qquad \text{ on } \partial \Omega \]
has a unique solution which satisfies 
\[ \Vert u \Vert_{C^{1,\alpha}} \le c \Vert f \Vert_{C^{0,\alpha}(\Omega)}. \]
We denote the map $ C^{0,\alpha} \ni f \mapsto u \in C^{1,\alpha} $ by $K$.

Now, let us consider the variable coefficients equation: 
\[
-\partial_i (a^{ij} \partial_j u) = \partial_i f^i\quad \text { in } \Omega, \qquad  u= 0 \quad \text{ on } \partial \Omega,
\]
where $ a^{ij} = \delta^{ij} + b^{ij} $
with $\Vert  b^{ij} \Vert_{C^{0,\alpha}} < \varepsilon$.
The equation is equivalent to 
\[ u = K (b^{ij} \partial_j u + f^i)_i \]
and existence and uniqueness is immediate if  $\Vert b^{ij}\Vert_{C^{0,\alpha}} $ is small enough. Moreover 
\be\label{a10}
\Vert u \Vert_{C^{1,\alpha}}\le \frac{C}{1-C \Vert b \Vert_{C^{0,\alpha}}} \Vert f \Vert_{C^{0,\alpha}}.
\ee 

$ii)$ Our second observation is that the unique 
harmonic function $\psi$ in $ \Omega $ satisfying
$$
\psi(x)= 1 \text{ if } |x|= 1 \quad\text{and}\quad  \psi(x)=0 \text{ if }|x|=2
$$
is given by 
\[ \psi(x) =
\left\{
\begin{aligned} 
&-\frac{\log(|x|/2)}{\log 2} \quad && \text{ if } N=2 \\[2mm] 
&\frac{|x|^{2-N}-2^{2-N}}{1-2^{2-N}} \quad && \text{ if } N \ge 3. 
\end{aligned}
\right.
\] 
In particular, there is a constant $\gamma_N$ depending only on the dimension $N$ such that
\be\label{Npsi}
\forall x\in \Omega,\quad x\cdot \nabla \psi\le -\gamma_N.
\ee

We will prove that the same result holds for the solution $u$ to:
\be\label{Nhomogeneous}
-\partial_i (a^{ij} \partial_j u)= 0, \qquad u(x) = 1  \quad \text{if } |x|= 1, \qquad u(x) = 0 \quad \text{if } |x|= 2.
\ee 
To see this, we find a solution by the Ansatz $ u = U + \psi$  where 
\[ -\partial_i (a^{ij} \partial_j U) = \partial_i (b^{ij} \partial_j \psi) \]
and hence the estimate~\e{a10} implies that
\[ \Vert  u - \psi \Vert_{C^{1,\alpha}} \lesssim \Vert b \Vert_{C^{0,\alpha}}.
\]
Of course this implies a bound for the $L^\infty$-norm 
of $x\cdot\nabla (u-\psi)$ and hence, by 
writing $u=(u-\psi)+\psi$ and using \e{Npsi} we conclude that, 
if $ \Vert b \Vert_{C^{0,\alpha}}$ 
is sufficiently small, then 
\[ 
\forall x\in \Omega,\quad x\cdot \nabla u(x) \le c\Vert b \Vert_{C^\alpha} 
- \left\{ \begin{array}{cl} 1  & \text{ if } N=2 \\
   (N-2)2^{2-N} & \text{ if } N \ge 3 .
   \end{array} \right. 
\]
In particular, we have
\be\label{Nxnabla}
\forall x\in \Omega, \quad x\cdot \nabla u \le -c_N.
\ee

We are now in 
position to finish the proof of Lemma~\ref{lem:hopf}. Given $ 0< r < 1$ and $x\in\Omega$, define 
\[ a^{ij}_r(x) =   a^{ij}(r(2e_N + x)).
\]
Since $ a^{ij}(0) =\delta^{ij}$, we 
have 
\[  \Vert  a^{ij}_r - \delta^{ij}\Vert_{C^{0,\alpha}(\Omega)}
\le c r^\alpha \Vert a^{ij} \Vert_{C^{0,\alpha}(B_2)}.
\]
If $r$ is sufficiently small there exists a unique solution to the homogeneous problem
\be\label{Nhomogeneous2}
-\partial_i( a^{ij}_r \partial_j u)= 0 \quad\text{in }\Omega, \qquad u(x) = 1  \quad \text{if } |x|= 1, \qquad u(x) = 0 \quad \text{if } |x|= 2.
\ee 
In addition, $u$ satisfies the inequality \e{Nxnabla}.

Now consider the rescaled domain (see Figure~\ref{Fig:rescaled})
$$
\Omega_r= B_{2r}(2re_N)\setminus B_r(2re_N)
$$
and the rescaled function $u_r$ defined on $\Omega_r$ by 
by $u(x)=u_r(r(x+2e_N))$. This function $u_r$ satisfies 
\[
-\partial_i (a^{ij} \partial_j u_r) = 0 \quad\text{in}\quad 
\Omega_r,\quad 
(x+2e_N)\cdot\nabla u_r\le -c_N \quad \text{in}\quad \Omega_r.
\]
In particular, we have
\be\label{NdN0}
\partial_N u_r(0)\ge c_N/2.
\ee

We then compare $v$ and $u_r$. To do this, we begin by observing that 
$v$ is nonnegative and hence we can apply 
the Harnack inequality, which implies that there is a constant $c(r)>0$ such that
\[
v |_{x_N \ge r, |x'|\le 1 } \ge   c(r).
\]
Consequently, we have $v\ge c(r)u_r$ on the boundary 
$\partial (B_{2r}(2re_N)\setminus B_r(2r e_N))$. 
By the maximum principle, we infer that
\[  v \ge c(r)u_r\quad\text{on}\quad B_{2r}(2re_N)\setminus B_r(2r e_N). \]
Now, noticing that $v(0)=0$ and $u_r(0)=u(-2e_N)=0$, 
by using \e{NdN0}, we conclude that: 
$$
\partial_N v (0) =\lim_{\eps\to 0}\frac{v(\eps e_N)}{\eps}
\ge c(r)\lim_{\eps \to 0}\frac{u_r(\eps e_N)}{\eps}
=c(r) \partial_Nu_r(0)\ge \frac{c(r)c_N}{2}.
$$
Which completes the proof of the lower bound.

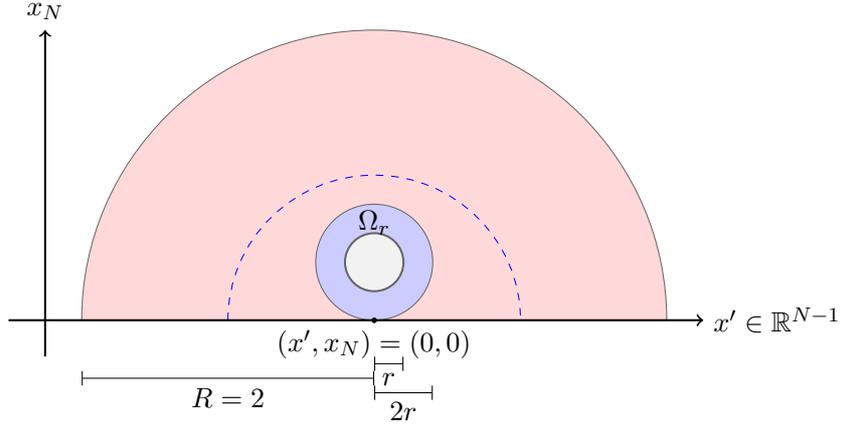
\begin{figure}[htb]
\centering
\resizebox{0.8\textwidth}{!}{%
\begin{tikzpicture}[rotate=90]
    \draw[black!60!white, fill=red!15!white] (4,0) arc(90:-90:4) -- (4,-4) 
      --  (4,-2)  -- (4,0);
    \draw[|-|]  (3.2,-4) --  (3.2,0) node [midway, below]{$R=2$};
    \draw[|-|]  (3.4,-4) --  (3.4,-4.4) node[midway,below]{$r$};
    \draw[|-|]  (3,-4) --  (3,-4.8) node[midway,below]{$2r$};
    \draw[fill=black] (4,-4) circle(.03);
    \draw[blue, dashed] (4,-2) arc(90:-90:2);
    \draw[black!60!white, fill=blue!20](4.8,-4) circle (0.8) ;
    \draw[black!60!white,thick, fill=gray!10](4.8,-4)  circle (0.4);
    \draw[thick,->] (3.5,0.5) -- (8,0.5) node[above] {$x_N$};
    \draw[thick,->] (4,1) -- (4,-8.5) node [right] {$x'\in \xR^{N-1}$};
    \draw[fill=black] (4,-4) circle(.03) node [below] {$(x',x_N)=(0,0)$};
    \node at (5.05,-4)[above] {$\Omega_r$};
 \end{tikzpicture}
}\caption{The comparison argument to get a lower bound. }\label{Fig:rescaled}
\end{figure}

We argue similarly for the upper bound: by comparison, we have 
\[ 0\le v \le 3C \quad\text{on}\quad B_2^+(0).\]
We extend the coefficient $a^{ij}$ to the full space and define (see Figure~\ref{Fig:rescaledm})
\[ \Omega_r= B_{2r}(-re_N) \backslash B_r(-e_N).\]
Then we consider the boundary value problem 
\[ \partial_i( a^{ij} \partial_j u) = 0 \qquad \text{ in } \Omega_r\]
with the boundary condition 
$$
u= 0 \quad\text{ on }\partial B_r(-re_N),\quad  u=1\quad \text{on }\partial B_{2r}(-e_N).
$$
Then $ v \in C^{1,\alpha} $ if $r$ is sufficiently small - we simply take $1-u$ in the construction above. 

Then 
\[ 0\le v \le 3C u \quad \text{ on } \partial (B_{2r}(-e_N) \cap \xR^N_+)     \]
and by the maximum principle $ v \le u $ on 
$B_{2r}(-e_N)\cap \xR^N_+$, hence 
\[ 0 \le \partial_N v(0) \le 3C \partial_N u(0). \]

This completes the proof of Lemma~\ref{lem:hopf} and hence the proof of Proposition~\ref{prop:taylor}.
\end{proof}

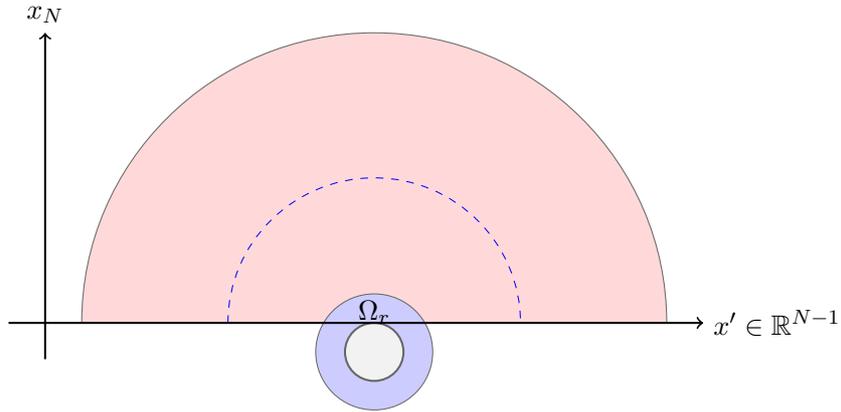
\begin{figure}[htb]
\centering
\resizebox{0.8\textwidth}{!}{%
\begin{tikzpicture}[rotate=90]
    \draw[black!60!white, fill=red!15!white] (4,0) arc(90:-90:4) -- (4,-4) 
      --  (4,-2)  -- (4,0);
    \draw[fill=black] (4,-4) circle(.03);
    \draw[blue, dashed] (4,-2) arc(90:-90:2);
    \draw[black!60!white, fill=blue!20](3.6,-4) circle (0.8) ;
    \draw[black!60!white,thick, fill=gray!10](3.6,-4)  circle (0.4);
    \draw[thick,->] (3.5,0.5) -- (8,0.5) node[above] {$x_N$};
    \draw[thick,->] (4,1) -- (4,-8.5) node [right] {$x'\in \xR^{N-1}$};
    \node at (3.85,-4)[above] {$\Omega_r$};
 \end{tikzpicture}
}\caption{The comparison argument to get an upper bound.}\label{Fig:rescaledm}
\end{figure}

\begin{flushleft}
Thomas Alazard\\ 
Universit\'e Paris-Saclay, ENS Paris-Saclay, CNRS\\
Centre Borelli UMR9010, avenue des Sciences, F-91190 Gif-sur-Yvette France\\
thomas.alazard@ens-paris-saclay.fr. 

\vspace{5mm}

Herbert Koch \\
Mathematisches Institut der Universit\"at Bonn,\\
Endicher Allee 60, 53115 Bonn, Germany,\\ koch@math.uni-bonn.de

\end{flushleft}


\begin{thebibliography}{10}

\bibitem{Abedin-Schwab-2020}
Farhan Abedin and Russell~W. Schwab.
\newblock Regularity for a special case of two-phase {H}ele-{S}haw flow via
  parabolic integro-differential equations.
\newblock {\em J. Funct. Anal.}, 285(8):Paper No. 110066, 83, 2023.

\bibitem{MR4520423}
Siddhant Agrawal, Neel Patel, and Sijue Wu.
\newblock Rigidity of acute angled corners for one phase {M}uskat interfaces.
\newblock {\em Adv. Math.}, 412:Paper No. 108801, 71, 2023.

\bibitem{MR2464701}
Hiroaki Aikawa.
\newblock Equivalence between the boundary {H}arnack principle and the
  {C}arleson estimate.
\newblock {\em Math. Scand.}, 103(1):61--76, 2008.

\bibitem{ABZ3}
Thomas Alazard, Nicolas Burq, and Claude Zuily.
\newblock On the {C}auchy problem for gravity water waves.
\newblock {\em Invent. Math.}, 198(1):71--163, 2014.

\bibitem{AMS}
Thomas Alazard, Nicolas Meunier, and Didier Smets.
\newblock Lyapunov functions, identities and the {C}auchy problem for the
  {H}ele-{S}haw equation.
\newblock {\em Comm. Math. Phys.}, 377(2):1421--1459, 2020.

\bibitem{AN3}
Thomas Alazard and Quoc-Hung Nguyen.
\newblock Endpoint {S}obolev theory for the {M}uskat equation.
\newblock {\em Comm. Math. Phys.}, 397(3):1043--1102, 2023.

\bibitem{Ancona-1978}
Alano Ancona.
\newblock Principe de {H}arnack \`a la fronti\`ere et th\'{e}or\`eme de {F}atou
  pour un op\'{e}rateur elliptique dans un domaine lipschitzien.
\newblock {\em Ann. Inst. Fourier (Grenoble)}, 28(4):169--213, x, 1978.

\bibitem{MR1942849}
Stanislav~N. Antontsev, C{\'a}tia.~R. Gon\c{c}alves, and Anvarbek~M. Meirmanov.
\newblock Local existence of classical solutions to the well-posed
  {H}ele-{S}haw problem.
\newblock {\em Port. Math. (N.S.)}, 59(4):435--452, 2002.

\bibitem{MR2072944}
Stanislav~N. Antontsev, Anvarbek~M. Meirmanov, and V.~V. Yurinsky.
\newblock Weak solutions for a well-posed {H}ele-{S}haw problem.
\newblock {\em Boll. Unione Mat. Ital. Sez. B Artic. Ric. Mat. (8)},
  7(2):397--424, 2004.

\bibitem{AtCa-1985}
Ioannis Athanasopoulos and Luis~A. Caffarelli.
\newblock A theorem of real analysis and its application to free boundary
  problems.
\newblock {\em Comm. Pure Appl. Math.}, 38(5):499--502, 1985.

\bibitem{Baiocchi-1971}
Claudio Baiocchi.
\newblock Sur un probl\`eme \`a fronti\`ere libre traduisant le filtrage de
  liquides \`a travers des milieux poreux.
\newblock {\em C. R. Acad. Sci. Paris S\'{e}r. A-B}, 273:A1215--A1217, 1971.

\bibitem{Baiocchi-1972}
Claudio Baiocchi.
\newblock Su un problema di frontiera libera connesso a questioni di idraulica.
\newblock {\em Ann. Mat. Pura Appl. (4)}, 92:107--127, 1972.

\bibitem{Baiocchi-1974}
Claudio Baiocchi.
\newblock Free boundary problems in the theory of fluid flow through porous
  media.
\newblock In {\em Proceedings of the International Congress of Mathematicians,
  Vancouver}, pages 237--243. Canadian Mathematical Congress Vancouver, 1974.

\bibitem{MR0745619}
Claudio Baiocchi and Ant\'{o}nio Capelo.
\newblock {\em Variational and quasivariational inequalities}.
\newblock A Wiley-Interscience Publication. John Wiley \& Sons, Inc., New York,
  1984.
\newblock Applications to free boundary problems, Translated from the Italian
  by Lakshmi Jayakar.

\bibitem{Brenier2009}
Yann Brenier.
\newblock On the hydrostatic and {D}arcy limits of the convective
  {N}avier-{S}tokes equations.
\newblock {\em Chinese Annals of Mathematics, Series B}, 30(6):683, 2009.

\bibitem{Brezis-Stampacchia-1973}
Ha{\"i}m Brezis and Guido Stampacchia.
\newblock Probl\`emes elliptiques avec fronti\`ere libre.
\newblock In {\em S\'{e}minaire {G}oulaouic-{S}chwartz 1972--1973:
  \'{E}quations aux d\'{e}riv\'{e}es partielles et analyse fonctionnelle},
  pages Exp. No. 11, 9. \'{E}cole Polytech., Paris, 1973.

\bibitem{MR1658612}
Luis~A. Caffarelli.
\newblock The obstacle problem revisited.
\newblock {\em J. Fourier Anal. Appl.}, 4(4-5):383--402, 1998.

\bibitem{CS-book}
Luis~A. Caffarelli and Sandro Salsa.
\newblock {\em A geometric approach to free boundary problems}, volume~68 of
  {\em Graduate Studies in Mathematics}.
\newblock American Mathematical Society, Providence, RI, 2005.

\bibitem{CFMS-1981}
Luis~A. Caffarelli, Sandro Salsa, Eugene~B. Fabes, and Stefano Mortola.
\newblock Boundary behavior of nonnegative solutions of elliptic operators in
  divergence form.
\newblock {\em Indiana Univ. Math. J.}, 30(4):621--640, 1981.

\bibitem{Caffarelli-Vasseur-AoM}
Luis~A. Caffarelli and Alexis Vasseur.
\newblock Drift diffusion equations with fractional diffusion and the
  quasi-geostrophic equation.
\newblock {\em Ann. of Math. (2)}, 171(3):1903--1930, 2010.

\bibitem{CaOrSi-SIAM90}
Russel~E. Caflisch, Oscar~F. Orellana, and Michael~Scott Siegel.
\newblock A localized approximation method for vortical flows.
\newblock {\em SIAM J. Appl. Math.}, 50(6):1517--1532, 1990.

\bibitem{Cameron}
Stephen Cameron.
\newblock Global well-posedness for the two-dimensional {M}uskat problem with
  slope less than 1.
\newblock {\em Anal. PDE}, 12(4):997--1022, 2019.

\bibitem{ICM1974}
QC~Canadian Mathematical~Congress, Montreal, editor.
\newblock {\em Proceedings of the {I}nternational {C}ongress of
  {M}athematicians. {V}olume 2}, 1975.
\newblock Held in Vancouver, B. C., August 21--29, 1974.

\bibitem{castro2016mixing}
\'{A}ngel Castro, Diego C{\'o}rdoba, and Daniel Faraco.
\newblock Mixing solutions for the {M}uskat problem.
\newblock {\em Invent. Math.}, 226(1):251--348, 2021.

\bibitem{CCFG-ARMA-2013}
\'{A}ngel Castro, Diego C\'{o}rdoba, Charles Fefferman, and Francisco Gancedo.
\newblock Breakdown of smoothness for the {M}uskat problem.
\newblock {\em Arch. Ration. Mech. Anal.}, 208(3):805--909, 2013.

\bibitem{CCFG-ARMA-2016}
\'{A}ngel Castro, Diego C\'{o}rdoba, Charles Fefferman, and Francisco Gancedo.
\newblock Splash singularities for the one-phase {M}uskat problem in stable
  regimes.
\newblock {\em Arch. Ration. Mech. Anal.}, 222(1):213--243, 2016.

\bibitem{CCFGLF-Annals-2012}
\'{A}ngel Castro, Diego C\'{o}rdoba, Charles Fefferman, Francisco Gancedo, and
  Mar\'{\i}a L\'{o}pez-Fern\'{a}ndez.
\newblock Rayleigh-{T}aylor breakdown for the {M}uskat problem with
  applications to water waves.
\newblock {\em Ann. of Math. (2)}, 175(2):909--948, 2012.

\bibitem{ChangLaraGuillenSchwab}
H\'{e}ctor~A. Chang-Lara, Nestor Guillen, and Russell~W. Schwab.
\newblock Some free boundary problems recast as nonlocal parabolic equations.
\newblock {\em Nonlinear Anal.}, 189:11538, 60, 2019.

\bibitem{changlara2016free}
Héctor~A. Chang-Lara and Nestor Guillen.
\newblock From the free boundary condition for hele-shaw to a fractional
  parabolic equation.
\newblock {\em arXiv:1605.07591}, 2016.

\bibitem{Cheng-Belinchon-Shkoller-AdvMath}
Ching-Hsiao Cheng, Rafael Granero-Belinch\'{o}n, and Steve Shkoller.
\newblock Well-posedness of the {M}uskat problem with {$H^2$} initial data.
\newblock {\em Adv. Math.}, 286:32--104, 2016.

\bibitem{MR2306045}
Sunhi Choi, David Jerison, and Inwon Kim.
\newblock Regularity for the one-phase {H}ele-{S}haw problem from a {L}ipschitz
  initial surface.
\newblock {\em Amer. J. Math.}, 129(2):527--582, 2007.

\bibitem{zbMATH05034334}
Sunhi Choi and Inwon Kim.
\newblock Waiting time phenomena of the {Hele}-{Shaw} and the {Stefan} problem.
\newblock {\em Indiana Univ. Math. J.}, 55(2):525--551, 2006.

\bibitem{cobb2023wellposedness}
Dimitri Cobb.
\newblock On the well-posedness of a fractional stokes-transport system.
\newblock {\em arXiv:2301.10511}, 2023.

\bibitem{CMT-1994}
Peter Constantin, Andrew~J. Majda, and Esteban Tabak.
\newblock Formation of strong fronts in the {$2$}-{D} quasigeostrophic thermal
  active scalar.
\newblock {\em Nonlinearity}, 7(6):1495--1533, 1994.

\bibitem{CV-2012}
Peter Constantin and Vlad Vicol.
\newblock Nonlinear maximum principles for dissipative linear nonlocal
  operators and applications.
\newblock {\em Geom. Funct. Anal.}, 22(5):1289--1321, 2012.

\bibitem{cordoba2011lack}
Diego C{\'o}rdoba, Daniel Faraco, and Francisco Gancedo.
\newblock Lack of uniqueness for weak solutions of the incompressible porous
  media equation.
\newblock {\em Arch. Ration. Mech. Anal.}, 200(3):725--746, 2011.

\bibitem{CG-CMP}
Diego C\'{o}rdoba and Francisco Gancedo.
\newblock Contour dynamics of incompressible 3-{D} fluids in a porous medium
  with different densities.
\newblock {\em Comm. Math. Phys.}, 273(2):445--471, 2007.

\bibitem{Cordoba-Lazar-H3/2}
Diego C\'{o}rdoba and Omar Lazar.
\newblock Global well-posedness for the 2{D} stable {M}uskat problem in
  {$H^{3/2}$}.
\newblock {\em Ann. Sci. \'{E}c. Norm. Sup\'{e}r. (4)}, 54(5):1315--1351, 2021.

\bibitem{MR0287357}
M.~G. Crandall and T.~M. Liggett.
\newblock Generation of semi-groups of nonlinear transformations on general
  {B}anach spaces.
\newblock {\em Amer. J. Math.}, 93:265--298, 1971.

\bibitem{Dahlberg-1977}
Bj\"{o}rn E.~J. Dahlberg.
\newblock Estimates of harmonic measure.
\newblock {\em Arch. Rational Mech. Anal.}, 65(3):275--288, 1977.

\bibitem{dalibard2023longtime}
Anne-Laure Dalibard, Julien Guillod, and Antoine Leblond.
\newblock Long-time behavior of the stokes-transport system in a channel.
\newblock {\em arXiv:2306.00780}, 2023.

\bibitem{MR4324293}
Noemi David and Beno\^{i}t Perthame.
\newblock Free boundary limit of a tumor growth model with nutrient.
\newblock {\em J. Math. Pures Appl. (9)}, 155:62--82, 2021.

\bibitem{dePoyferre1}
Thibault De~Poyferr{\'e}.
\newblock Blow-up conditions for gravity water-waves.
\newblock {\em arXiv:1407.6881}, 2014.

\bibitem{MR4093736}
Daniela De~Silva and Ovidiu~V. Savin.
\newblock A short proof of boundary {H}arnack principle.
\newblock {\em J. Differential Equations}, 269(3):2419--2429, 2020.

\bibitem{MR4655356}
Hongjie Dong, Francisco Gancedo, and Huy~Q. Nguyen.
\newblock Global well-posedness for the one-phase {M}uskat problem.
\newblock {\em Comm. Pure Appl. Math.}, 76(12):3912--3967, 2023.

\bibitem{Duvaut-1973}
Georges Duvaut.
\newblock R\'{e}solution d'un probl\`eme de {S}tefan (fusion d'un bloc de glace
  \`a z\'{e}ro degr\'{e}).
\newblock {\em C. R. Acad. Sci. Paris S\'{e}r. A-B}, 276:A1461--A1463, 1973.

\bibitem{MR611303}
C.~M. Elliott and V.~Janovsk\'{y}.
\newblock A variational inequality approach to {H}ele-{S}haw flow with a moving
  boundary.
\newblock {\em Proc. Roy. Soc. Edinburgh Sect. A}, 88(1-2):93--107, 1981.

\bibitem{Escher-Simonett-ADE-1997}
Joachim Escher and Gieri Simonett.
\newblock Classical solutions for {H}ele-{S}haw models with surface tension.
\newblock {\em Adv. Differential Equations}, 2(4):619--642, 1997.

\bibitem{FGMMS-1988}
Eugene Fabes, Nicola Garofalo, Santiago Mar\'{\i}n-Malave, and Sandro Salsa.
\newblock Fatou theorems for some nonlinear elliptic equations.
\newblock {\em Rev. Mat. Iberoamericana}, 4(2):227--251, 1988.

\bibitem{MR1658089}
Eugene Fabes, Osvaldo Mendez, and Marius Mitrea.
\newblock Boundary layers on {S}obolev-{B}esov spaces and {P}oisson's equation
  for the {L}aplacian in {L}ipschitz domains.
\newblock {\em J. Funct. Anal.}, 159(2):323--368, 1998.

\bibitem{Figalli2018free}
Alessio Figalli.
\newblock Free boundary regularity in obstacle problems.
\newblock {\em arXiv:1807.01193}, 2018.

\bibitem{Figalli-ICM2018}
Alessio Figalli.
\newblock Regularity of interfaces in phase transitions via obstacle
  problems---{F}ields {M}edal lecture.
\newblock In {\em Proceedings of the {I}nternational {C}ongress of
  {M}athematicians---{R}io de {J}aneiro 2018. {V}ol. {I}. {P}lenary lectures},
  pages 225--247. World Sci. Publ., Hackensack, NJ, 2018.

\bibitem{MR4179834}
Alessio Figalli, Xavier Ros-Oton, and Joaquim Serra.
\newblock Generic regularity of free boundaries for the obstacle problem.
\newblock {\em Publ. Math. Inst. Hautes \'{E}tudes Sci.}, 132:181--292, 2020.

\bibitem{figalli2021singular}
Alessio Figalli, Xavier Ros-Oton, and Joaquim Serra.
\newblock The singular set in the {S}tefan problem.
\newblock {\em J. Amer. Math. Soc.}, 37(2):305--389, 2024.

\bibitem{forster2018piecewise}
Clemens F\"{o}rster and L\'{a}szl\'{o} Sz\'{e}kelyhidi, Jr.
\newblock Piecewise constant subsolutions for the {M}uskat problem.
\newblock {\em Comm. Math. Phys.}, 363(3):1051--1080, 2018.

\bibitem{Gancedo-Lazar-H2}
Francisco Gancedo and Omar Lazar.
\newblock Global well-posedness for the three dimensional {M}uskat problem in
  the critical {S}obolev space.
\newblock {\em Arch. Ration. Mech. Anal.}, 246(1):141--207, 2022.

\bibitem{GJGSHP2023desingularization}
Eduardo Garc\'ia-Ju\'arez, Javier G\'omez-Serrano, Susanna~V. Haziot, and
  Beno\^it Pausader.
\newblock Desingularization of {S}mall {M}oving {C}orners for the {M}uskat
  {E}quation.
\newblock {\em Ann. PDE}, 10(2):Paper No. 17, 2024.

\bibitem{Granero-Scrobogna}
Rafael Granero-Belinch\'{o}n and Stefano Scrobogna.
\newblock On an asymptotic model for free boundary {D}arcy flow in porous
  media.
\newblock {\em SIAM J. Math. Anal.}, 52(5):4937--4970, 2020.

\bibitem{MR4367913}
Nestor Guillen, Inwon Kim, and Antoine Mellet.
\newblock A {H}ele-{S}haw limit without monotonicity.
\newblock {\em Arch. Ration. Mech. Anal.}, 243(2):829--868, 2022.

\bibitem{Gustafsson2}
Bj\"{o}rn Gustafsson.
\newblock Applications of variational inequalities to a moving boundary problem
  for {H}ele-{S}haw flows.
\newblock {\em SIAM J. Math. Anal.}, 16(2):279--300, 1985.

\bibitem{MR2203166}
David Jerison and Inwon Kim.
\newblock The one-phase {H}ele-{S}haw problem with singularities.
\newblock {\em J. Geom. Anal.}, 15(4):641--667, 2005.

\bibitem{MR676988}
David~S. Jerison and Carlos~E. Kenig.
\newblock Boundary behavior of harmonic functions in nontangentially accessible
  domains.
\newblock {\em Adv. in Math.}, 46(1):80--147, 1982.

\bibitem{MR676987}
Peter~W. Jones.
\newblock A geometric localization theorem.
\newblock {\em Adv. in Math.}, 46(1):71--79, 1982.

\bibitem{Kemper-1972}
John~T. Kemper.
\newblock A boundary {H}arnack principle for {L}ipschitz domains and the
  principle of positive singularities.
\newblock {\em Comm. Pure Appl. Math.}, 25:247--255, 1972.

\bibitem{MR2218549}
Inwon Kim.
\newblock Long time regularity of solutions of the {H}ele-{S}haw problem.
\newblock {\em Nonlinear Anal.}, 64(12):2817--2831, 2006.

\bibitem{kim2023density}
Inwon Kim, Antoine Mellet, and Yijing Wu.
\newblock A density-constrained model for chemotaxis.
\newblock {\em Nonlinearity}, 36(2):1082, 2023.

\bibitem{Kim-ARMA2003}
Inwon~C. Kim.
\newblock Uniqueness and existence results on the {H}ele-{S}haw and the
  {S}tefan problems.
\newblock {\em Arch. Ration. Mech. Anal.}, 168(4):299--328, 2003.

\bibitem{MR2210142}
Inwon~C. Kim.
\newblock Regularity of the free boundary for the one phase {H}ele-{S}haw
  problem.
\newblock {\em J. Differential Equations}, 223(1):161--184, 2006.

\bibitem{MR0440187}
D.~Kinderlehrer and L.~Nirenberg.
\newblock Regularity in free boundary problems.
\newblock {\em Ann. Scuola Norm. Sup. Pisa Cl. Sci. (4)}, 4(2):373--391, 1977.

\bibitem{MR1786735}
David Kinderlehrer and Guido Stampacchia.
\newblock {\em An introduction to variational inequalities and their
  applications}, volume~31 of {\em Classics in Applied Mathematics}.
\newblock Society for Industrial and Applied Mathematics (SIAM), Philadelphia,
  PA, 2000.
\newblock Reprint of the 1980 original.

\bibitem{MR1363758}
J.~R. King, A.~A. Lacey, and J.~L. V\'{a}zquez.
\newblock Persistence of corners in free boundaries in {H}ele-{S}haw flow.
\newblock {\em European J. Appl. Math.}, 6(5):455--490, 1995.
\newblock Complex analysis and free boundary problems (St. Petersburg, 1994).

\bibitem{KN-2009}
Alexander Kiselev and Fedor Nazarov.
\newblock A variation on a theme of {C}affarelli and {V}asseur.
\newblock {\em Zap. Nauchn. Sem. S.-Peterburg. Otdel. Mat. Inst. Steklov.
  (POMI)}, 370(Kraevye Zadachi Matematichesko\u{\i} Fiziki i Smezhnye Voprosy
  Teorii Funktsi\u{\i}. 40):58--72, 220, 2009.

\bibitem{KNV-2007}
Alexander Kiselev, Fedor Nazarov, and Alexander Volberg.
\newblock Global well-posedness for the critical 2{D} dissipative
  quasi-geostrophic equation.
\newblock {\em Invent. Math.}, 167(3):445--453, 2007.

\bibitem{MR2735914}
Bertrand Maury, Aude Roudneff-Chupin, and Filippo Santambrogio.
\newblock A macroscopic crowd motion model of gradient flow type.
\newblock {\em Math. Models Methods Appl. Sci.}, 20(10):1787--1821, 2010.

\bibitem{MR1925260}
Anvarbek~M. Meirmanov and Boris Zaltzman.
\newblock Global in time solution to the {H}ele-{S}haw problem with a change of
  topology.
\newblock {\em European J. Appl. Math.}, 13(4):431--447, 2002.

\bibitem{MR0159138}
J\"{u}rgen Moser.
\newblock On {H}arnack's theorem for elliptic differential equations.
\newblock {\em Comm. Pure Appl. Math.}, 14:577--591, 1961.

\bibitem{MR4090462}
Huy~Q. Nguyen and Beno\^{i}t Pausader.
\newblock A paradifferential approach for well-posedness of the {M}uskat
  problem.
\newblock {\em Arch. Ration. Mech. Anal.}, 237(1):35--100, 2020.

\bibitem{noisette2020mixing}
Florent Noisette and L\'{a}szl\'{o} Sz\'{e}kelyhidi, Jr.
\newblock Mixing solutions for the {M}uskat problem with variable speed.
\newblock {\em J. Evol. Equ.}, 21(3):3289--3312, 2021.

\bibitem{MR3162474}
Beno\^{i}t Perthame, Fernando Quir\'{o}s, and Juan~Luis V\'{a}zquez.
\newblock The {H}ele-{S}haw asymptotics for mechanical models of tumor growth.
\newblock {\em Arch. Ration. Mech. Anal.}, 212(1):93--127, 2014.

\bibitem{MR2962060}
Arshak Petrosyan, Henrik Shahgholian, and Nina Uraltseva.
\newblock {\em Regularity of free boundaries in obstacle-type problems}, volume
  136 of {\em Graduate Studies in Mathematics}.
\newblock American Mathematical Society, Providence, RI, 2012.

\bibitem{Pruss-Simonett-book}
Jan Pr\"{u}ss and Gieri Simonett.
\newblock {\em Moving interfaces and quasilinear parabolic evolution
  equations}, volume 105 of {\em Monographs in Mathematics}.
\newblock Birkh\"{a}user/Springer, [Cham], 2016.

\bibitem{SCH2004}
Michael Siegel, Russel~E. Caflisch, and Sam Howison.
\newblock Global existence, singular solutions, and ill-posedness for the
  {M}uskat problem.
\newblock {\em Comm. Pure Appl. Math.}, 57(10):1374--1411, 2004.

\bibitem{Silvestre-2012}
Luis Silvestre.
\newblock H\"{o}lder estimates for advection fractional-diffusion equations.
\newblock {\em Ann. Sc. Norm. Super. Pisa Cl. Sci. (5)}, 11(4):843--855, 2012.

\bibitem{szekelyhidi2012relaxation}
L\'{a}szl\'{o} Sz\'{e}kelyhidi, Jr.
\newblock Relaxation of the incompressible porous media equation.
\newblock {\em Ann. Sci. \'{E}c. Norm. Sup\'{e}r. (4)}, 45(3):491--509, 2012.

\bibitem{MR1121019}
Michael~E. Taylor.
\newblock {\em Pseudodifferential operators and nonlinear {PDE}}, volume 100 of
  {\em Progress in Mathematics}.
\newblock Birkh\"{a}user Boston, Inc., Boston, MA, 1991.

\bibitem{MR2744149}
Michael~E. Taylor.
\newblock {\em Partial differential equations {III}. {N}onlinear equations},
  volume 117 of {\em Applied Mathematical Sciences}.
\newblock Springer, New York, second edition, 2011.

\bibitem{VaVi}
Ioann Vasilyev and Fran\c{c}ois Vigneron.
\newblock Variation on a theme by {K}iselev and {N}azarov: {H}\"{o}lder
  estimates for nonlocal transport-diffusion, along a non-divergence-free {BMO}
  field.
\newblock {\em J. Inst. Math. Jussieu}, 21(5):1651--1675, 2022.

\bibitem{Vazquez-DCDS}
Juan~Luis V\'{a}zquez.
\newblock Recent progress in the theory of nonlinear diffusion with fractional
  {L}aplacian operators.
\newblock {\em Discrete Contin. Dyn. Syst. Ser. S}, 7(4):857--885, 2014.

\bibitem{Wu-1978}
Jang Mei~G. Wu.
\newblock Comparisons of kernel functions, boundary {H}arnack principle and
  relative {F}atou theorem on {L}ipschitz domains.
\newblock {\em Ann. Inst. Fourier (Grenoble)}, 28(4):147--167, vi, 1978.

\end{thebibliography}
\end{document}